\documentclass[11pt]{article}
\usepackage{fullpage}
\usepackage{authblk}
\usepackage[T1]{fontenc}
\usepackage[utf8]{inputenc}
\usepackage{amsthm}
\usepackage{amsmath}
\usepackage{amssymb}
\usepackage{graphicx}
\usepackage{dsfont}
\usepackage{cite}
\usepackage{amsmath}
\usepackage{array}
\usepackage{multirow}
\usepackage{caption}
\usepackage{color}
\usepackage{hyperref}
\usepackage{cleveref}
\usepackage{verbatim}
\usepackage{accents}
\usepackage{comment}
\usepackage{xcolor}
\usepackage{algorithm}
\usepackage{algpseudocode}
\usepackage{mathtools}
\usepackage{subcaption}

\usepackage[normalem]{ulem}
\usepackage{cancel}
\usepackage{multicol}

\theoremstyle{plain}

\theoremstyle{plain}
\newtheorem{prop}{\protect\propositionname}
\theoremstyle{plain}
\newtheorem{lem}{\protect\lemmaname}
\theoremstyle{plain}
\newtheorem{thm}{\protect\theoremname}
\theoremstyle{plain}
\newtheorem{cor}{\protect\corollaryname}  
\theoremstyle{definition}
\newtheorem{defn}{\protect\definitionname}
\theoremstyle{definition}
\newtheorem{assump}{\protect\assumptionname}
\theoremstyle{definition}
\newtheorem{rem}{\protect\remarkname}

\makeatother
  
\usepackage{babel} 

\providecommand{\claimname}{Claim}
\providecommand{\lemmaname}{Lemma}
\providecommand{\propositionname}{Proposition}
\providecommand{\theoremname}{Theorem}
\providecommand{\corollaryname}{Corollary} 
\providecommand{\definitionname}{Definition}
\providecommand{\assumptionname}{Assumption}
\providecommand{\remarkname}{Remark}

\newcommand{\EE}{\mathbb{E}}

\newcommand{\RR}{\mathbb{R}}

\newcommand{\gfk}{\nabla f(x_k)}
\newcommand{\gfkk}{\nabla f(x_{k-1})}
\newcommand{\vkh}{\hat{v}_k}
\newcommand{\vkkh}{\hat{v}_{k-1}}

\newcommand{\op}{\mathrm{op}}

\newcommand{\Proj}{\mathrm{Proj}}

\newcommand{\brac}[1]{\left( #1\right)}
\newcommand{\sqbrac}[1]{\left[ #1\right]}

\DeclarePairedDelimiter{\norm}{\lVert}{\rVert}
\DeclarePairedDelimiter{\abs}{\lvert}{\rvert}
\DeclarePairedDelimiter{\inner}{\langle}{\rangle}
\newcommand{\R}{\mathbb{R}}

\title{{\sc Gradient Descent with Stochastic Subspaces \\ via Persistence of Memory}}
\author[1]{Subhro Ghosh\thanks{\texttt{subhrowork@gmail.com}}}
\author[1]{Clement Z.Q. Ng\thanks{Corresponding author, \texttt{clementng@u.nus.edu}}}
\author[2]{Pierre-Louis Poirion\thanks{Corresponding author, \texttt{pierre-louis.poirion@riken.jp}}}
\author[2,3]{Akiko Takeda\thanks{\texttt{takeda@mist.i.u-tokyo.ac.jp}}}
\affil[1]{Department of Mathematics, National University of Singapore, Singapore}
\affil[2]{Center for Advanced Intelligence Project,
RIKEN, Tokyo, Japan}
\affil[3]{Department of Mathematical Informatics,
The University of Tokyo, Tokyo, Japan}
\date{}

\begin{document}

\maketitle
\begingroup
\renewcommand{\thefootnote}{}
\footnotetext{\text{Authors are listed in alphabetical order.}}
\endgroup

\begin{abstract}
Stochastic subspace methods have gained popularity as gradient descent based techniques for large scale optimisation problems, especially in distributed settings. In this paper, we introduce the technique of "persistence of memory" to greatly extend and improve the random subspace methods. To this end, we leverage a vector that is only {\it weakly correlated} with the gradient in order to provide a  {\it guiding structure} to the generative process of the random subspace along which the descent is going to take place. This {\it guidance vector} may be fixed for a large number of iterations, only to be refreshed at wide intervals (on whose size we can provide guarantees in terms of problem parameters). In important machine learning settings, such as optimisation problems embodying sparsity or a minibatch structure, we show that the {\it guidance vector} can be obtained in an effective and computationally inexpensive manner by leveraging the structured properties of the problem. En route, we establish to our knowledge the first theoretical analysis of classical SSD methods for sparse functions. In a local neighbourhood of the optimum, we demonstrate an {\it alignment phenomenon} of our gradient estimates with a low-lying eigenvector of the Hessian, allowing a once-for-all computation of the guidance vector which renders the method computationally favourable even in scenarios with unstructured objectives.
\end{abstract}

\newpage

\tableofcontents

\newpage
\section{Introduction}
In recent years, there are many literature being developed around the study of high-dimensional unconstrained optimization. In simple terms, such a problem can be formulated as
\begin{equation}
    \min_{x \in \RR^n} f(x).
\end{equation}
for a function $f:\RR^n \mapsto \RR$.
When the dimension $n$ is sufficiently large, as is the case in modern machine learning applications, classical gradient methods often become prohibitively expensive due to the need to compute full gradients at every iteration step. 

To address this issue, subspace optimization methods have been developed, which are, in turn generalisations of the more basic coordinate descent approach. In order to save computational costs and bypass the high dimensional constraints ($x \in \mathbb{R}^n)$, randomised variants of these methods have gained popularity; a key example of this being the so-called Stochastic Subspace Descent (SSD) method. For details on the substantial literature on subspace descent, randomised and otherwise, we refer the reader to \cite{kozak2021stochastic} and the references, therein as a partial list.

However, standard SSD methods generally incur a high cost in terms of iteration complexity, which is in part a consequence of the random noise inherent in their setup. In this work, we propose to enhance the state of the art in this problem by unveiling structured stochastic subspace generation techniques that vastly improve the iteration complexity while retaining computational tractability in a wide range of structured machine learning problems.

\subsection{Motivation}
In Stochastic Subspace Descent (SSD), at each update step, the gradients are computed along a randomly chosen low-dimensional subspace. In other words, instead of computing the gradient $\nabla f(x)$, SSD calculates the projection of the gradient vector onto a {\it random} subspace of dimension $d$. In formulaic terms, the $k$-th update step in SSD looks like 
\[
    x_{k} = x_{k-1} - \alpha P_{k-1} P_{k-1}^\top \nabla f(x_{k-1}),
\]
where $P_k \in \mathbb{R}^{n \times d}$ is a tall matrix with orthogonal columns and $d \ll n$. 

In general, this subspace is chosen {\it uniformly at random}; for instance, the specific device used by the initial work \cite{kozak2021stochastic} is that of Haar-distributed matrices $P_k$. This algorithm is particularly useful for high-dimensional problems, and the main advantage of this method is that instead of calculating the full gradient, it only calculates a projection of the gradient onto a lower dimension. Indeed in some cases, computing the full gradient of the function might be memory heavy, in such cases we only need to compute a small number of directional derivatives (using $d$ random directions) rather than $n$ of them (think of the gradient as $n$ directional derivatives in the coordinate directions).

In fact, this idea generalises coordinate descent - by allowing projections onto non-coordinate subspaces, and can be particularly effective in cases where the function exhibits low intrinsic dimensionality or when many of the gradient components are negligible. Other situations where subspace methods can be considered are discussed in the later subsections. 

Denoting the cost of 1 computation of directional derivative as $\xi$, the standard gradient descent has a iteration complexity of $O(\epsilon^{-2})$ and the cost per iteration is $O(n\xi)$ ($n$ number of directional derivatives), for a total computational cost of $O(n\xi\epsilon^{-2})$. In comparison, SSD has a decreased per iteration cost of $O(d\xi)$ as compared to $O(n\xi)$, but at the same time suffers from the curse of dimensionality in its iteration complexity ($O(\epsilon^{-2} n/d)$ instead). While each iteration is computationally cheaper, the total complexity will be similar to the classical methods to achieve the same degree of accuracy. Recognising this trade-off, recent research has focused on enhancing the efficiency of each subspace, by integrating techniques that "remember" past useful information. For instance, enhanced variants of SSD - such as those incorporating elements from stochastic variance reduction (SVRG) and trust-region frameworks - demonstrate how leveraging historical gradient information can effectively reduce noise and improve stability \cite{kozak2021stochastic, cartis2022randomised, dzahini2024stochastictrustregion}.

\subsection{Persistence of Memory.}
In this paper, we introduce the technique of "persistence of memory" to greatly extend and improve the random subspace methods. To wit, we leverage a vector that is only {\it weakly correlated} with the gradient in order to provide a  {\it guiding structure} to the generative process of the random subspace along which the descent is going to take place. This {\it guidance vector} (or {\it alignment vector}) may be fixed for a large number of iterations, only to be refreshed at wide intervals (on whose size we can provide guarantees in terms of problem parameters). In view of the persistence of the guidance vector throughout large cycles of the optimisation procedure, we call this approach {\it persistence of memory}.  
In fact, in significant setups such as the local regime described in Section \ref{sec:local}, the guidance vector may be fixed once and for all without any refresh being required. Thus, reinforcing the concept of persistence of structural memory in the random subspaces throughout the descent to the optimum.

In important machine learning setups, such as optimisation problems embodying sparsity or a minibatch structure,  the {\it guidance vector} can be obtained in a computationally inexpensive manner by leveraging the structured properties of the problem (refer to the applications in Section \ref{sec:applications} for examples). In unstructured scenarios, the guidance vector may be obtained from black-box gradient generation methods typically hypothesised in the optimisation literature; in fact, our results in the local regime show that our approach enjoys advantages even in such setups. 

Because the randomness is still retained in a major way in this approach, the method also enjoys the benefits due to randomness as the original method for a majority of the steps.

\subsubsection{Structural Overview of Our Method}
In this section, we provide a brief structural description of our method, in sufficient detail, so as to provide an overall discussion of our contributions. For a more in-depth description of the algorithm, we refer the reader to Sections \ref{sec:results}, and Section \ref{sec:applications} for estimating the {\it guidance vector} in structured setups.

To understand our approach, it is  beneficial to first recall the classical SSD algorithm. This consists of computing, at every iteration, the gradient of the objective function only along a randomly chosen subspace of dimension $d$. In practice, this can be accomplished, e.g. by picking at the $k$-th iteration (the projection $P_k$ on to) a uniformly random subspace of dimension $d$, usually denominated by its columns (in an orthogonal matrix representation). Then,  the projection of the gradient onto the random subspace can be easily computed in terms of the directional derivatives of the objective function along these columns. Typically, $d$ is taken to be low (so as to keep the cost per iteration small), but high enough to ensure that a reasonable representation of the gradient can still be obtained. By the famous Johnson-Lindenstrauss Lemma (and its derivatives), this can already be achieved when $d$ is of the order $\log n$, which is a ballpark scale for us to keep in mind throughout this paper. Of course, $d$ can also be taken to be somewhat larger, such as a small power of $n$, depending on the computational capacity available and other problem considerations.

If  the ambient dimension is $n$, then this costs $d/n$, a fraction in terms computational load compared to a full gradient computation. However, the randomness inherent in the SSD algorithm makes the iterative procedure take much longer to converge, more precisely, $O(n/d \cdot \epsilon^{-2})$ steps (where $\epsilon$ is the accuracy threshold for terminating the algorithm). Thus, what is gained in terms of computational cost {\it per iteration} is essentially lost in a vastly increased {\it number of iterations}, leaving the overall computational cost more or less unchanged.

Our method consists in a more nuanced random subspace generation mechanism that is sensitive to the problem at hand, thereby mitigating the problem of a large iteration complexity. At the same time, this achieves an overall computational advantage in  structural settings that are of fundamental importance in machine learning applications. For the purposes of discussion later, we define the notion of {\it alignment} between two vectors $u$ and $v$ to be $\frac{\langle u,v\rangle^2}{\|u\|_2^2 \|v\|^2_2}$. The alignment therefore ranges between 0 and 1, with alignment 1 indicating the vectors identifying the same straight line.

The cornerstone of our approach is a {\it guidance vector}, denoted by $v_k$ at the $k$-th step in the gradient decent algorithm. We require that this guidance vector is {\it weakly correlated} with the true gradient $\nabla f (x_k)$ at the $k$-th iterate $x_k$, in the sense that $\langle \hat v_k , \nabla f (x_k) \rangle^2 \ge \delta \norm{\gfk}_2^2$, where $\delta > 0$ is an alignment parameter and $\hat v_k$ is the normalised vector of $v_k$. Instead of generating a uniformly random subspace, we now generate a random subspace that is {\it conditioned to contain the guidance vector $v_k$}. In practice, this may be achieved simply by concatenating $v_k$ to a uniformly random subspace of dimension $d$ that is orthogonal to $v_k$. The random subspace generated will naturally carry much more information about the true gradient (which is the direction that we would ideally want to descend along).

Of course, obtaining such a guidance vector $v_k$ might incur a cost, especially if it is necessary to do it at every step. We therefore recommend generating the guidance vector only occasionally, a phenomenon that we call  {\it a refresh}. At any refresh step, the guidance vector will have an {\it initial alignment} $\gamma_0$ with the gradient, i.e. $\langle v_{\mathrm{refresh}}, \nabla f (x_{\mathrm{refresh}}) \rangle^2 = \gamma_0 \|  \nabla f (x_{\mathrm{refresh}}) \|_2^2$, which typically will much larger than the {\it alignment threshold} $\delta$. This vector $v_{\mathrm{refresh}}$ is going to be used as the guidance vector in the many subsequent steps of the descent (i.e., $v_k = v_{\mathrm{refresh}}$ until the next refresh step). The point is that, because the evolution of the argument $x_k$ in gradient descent is gradual, the  $v_{\mathrm{refresh}}$ will still have reasonably good alignment with the true gradient for many steps; in other words, we have {\it persistence of memory} of  $v_{\mathrm{refresh}}$ in the actual gradient. 

However, as the iteration proceeds, the alignment between the guidance vector and the true gradient worsens, starting from the initial alignment $\gamma_0$ but decaying towards the alignment threshold $\delta$. After a certain number of steps, the memory fades and we perform another refresh of the guidance vector (before the alignment falls below $\delta$). We continue in this vein until we reach the desired termination condition of our descent algorithm, which is typically enunciated in terms $\|  \nabla f (x)\|_2 \le \epsilon$.

We provide a quantitative and principled structure to the idea outlined above, with explicit guarantees on the iteration complexity and the number of iterations between two successive refresh steps, purely in terms of the parameters of the problem introduced above. For a detailed technical statement of the results, we refer to Theorem \ref{thm:conv_iter} and Corollary \ref{cor:ref_bound} respectively; for an overview of the results and discussion on their efficacy, we refer to the subsequent Sections \ref{sec:results} and \ref{sec:applications}.

\subsubsection{Iteration Complexity}
In Theorem \ref{thm:conv_iter}, we demonstrate that the iteration complexity of our method is 
\begin{equation}\label{eq:intro-iter}
    \frac{\alpha^{-1} \left(1- \frac{\alpha L}{2}\right)^{-1}(f(x_0) - f^\ast)}{\frac{4\alpha L (1-\alpha L)}{(1-2\alpha L)\log(\gamma_0/\tilde \delta)} \left(1 - \frac{d}{n-1}\right) \gamma_0 + \frac{d}{n-1} - \tilde t_0}\epsilon^{-2}
\end{equation}
where $\alpha$ is the step size, $L$ is the smoothness constant of the objective function $f$, $x_0$ is the initialisation, $f^\ast$ the true minimum, $\epsilon$ is the accuracy threshold for termination, $\gamma_0$ the starting alignment at every refresh, $\delta = \tilde \delta - \sqrt{\frac{d}{n-1}}$ the alignment lower bound that holds throughout the process and $\tilde t_0$ a function of the problem parameters which roughly equals $\frac{1}{2} \sqrt{\frac{d}{n}}$.

To illustrate our method clearly, we discuss the ballpark regime where we set the initial and threshold alignment parameters, namely $\gamma_0$ and $\delta$, to be $\Theta(1)$.  Given that the step-size $\alpha$ is typically (a small) constant times $1/L$ and $d \ll n$, we may then deduce from \eqref{eq:intro-iter} that the iteration complexity is $O(\epsilon^{-2})$. This, in particular, is a very substantial improvement from the classical SSD algorithm, which entails an iteration complexity of $O(n/d)$. With a common choice of $d$ being $\Theta(\log n)$, this accords us a speed-up almost by a factor of $n$ (up to $\log$ terms), which is very substantial in the high dimensional, large scale applications that these methods are designed for.

\subsubsection{Fast Generation of Weakly Correlated Vectors for Machine Learning Applications}
Functions with structure are of a wider interest to us and the wider community. In some problem settings (as described below), we show that by leveraging specific properties of the function, the computational cost of computing the {\it guidance vector} at each refresh step can be relatively cheap by sacrificing slightly on the initial alignment parameter. This trade-off can give us an overall decreased computational complexity.

\paragraph{Sparse Functions.}
The first class of such functions we test the algorithm on is when the function is sparse (i.e. the function depends only on a fixed subset of coordinates). Such functions are natural in machine learning, where some features might be highly correlated, resulting in redundant features and an intrinsic low dimensionality of the problem. We utilise well-established algorithms in the compressed sensing literature to construct the guidance vector in this setting. Suppose there exists a matrix $\Psi$ satisfying the Restricted Isometry Property (RIP)
\[
    (1-\delta_s) \norm*{z}^2 \leq \norm*{\Psi z}^2 \leq (1+\delta_s) \norm*{z}^2    
\]
for some constant $\delta_s$ and for all $s-$sparse vectors $z \in \mathbb{R}^n$, then many of these algorithms can find a vector that is close to the true sparse vector in $\ell_2$ norm, namely
\[
    \norm*{y_T - y^\ast} \leq \rho \norm*{y^\ast},
\]
where $y_T$ is the output of the algorithm after $T$ iterations, and $y^\ast$ is the true sparse vector. In particular, we analyse the Iterative Hard Thresholding (IHT) algorithm \cite{davies2009iht}, which has that $\rho$ decreases exponentially in $T$. This means that the cost to compute $v_0 = y_T$, $O(s\log(n))$ directional derivatives coupled with the cost of the IHT algorithm whcih scales with $s\log(n)\log(1/\rho)$, is substantially lower as compared to $O(n)$ many directional derivatives for the full gradient. For $\rho = \Theta(1)$, the logarithmic factor with $\rho$ is essentially order $1$. This leads to an overall decrease in the total computational cost compared to when the problem has no structure to $O(s\log(n/s) \epsilon^{-2} \xi)$.

En route our investigations, we also provide an analysis of classical SSD methods in the sparse setting; see Section \ref{sec:sparse_kozak}. To our knowledge, this is the first theoretical analysis of vanilla SSD in such a setup.

\paragraph{Additive Structure.}
Another class of functions that we analyse is functions that are a sum of functions, that is, $f(x) = \frac{1}{M} \sum_{i=1}^M f_i(x)$. Functions in this form are common in machine learning, such as evaluating the loss on the data set. Gradient descent is usually performed (in practice) using mini-batch as the direction, especially when $M \gg n$. We show in Section \ref{sec:mb_general} that the algorithm can also be applied to mini-batch gradient descent, with a slight tweak. For this scenario, one possible oracle of the alignment vector is to choose a mini-batch gradient of a larger batch size, namely let $B' = \{i_1, \cdots, i_{m'}\}$, where each $i's$ are iid sampled from $[M]$, then set
\[v_0 = \frac{1}{m'} \sum_{i \in B'} \nabla f_i(x_0),\]
where the batch size $m'$ can potentially be different from $m$. Under the bounded variance assumption and a specific rate of updating the guidance vector, we show that the iteration complexity scales as $O\left(\epsilon^{-4 \left(\frac{1+\beta}{1-\beta} \right)}\right)$, for some $\beta \in (0,1)$.

\paragraph{}
In both cases, the structure of the problem allows us to compute a guidance vector in the refresh steps at a cost that is much cheaper than the full gradient. At the same time, we also enjoy the benefits of the original SSD method, where the dimension per step we consider is vastly reduced.

We believe that the idea and concepts of persistence of memory can be fruitfully applied to wider range of stochastic optimisation.

\subsubsection{Computational Cost}
The computational overhead in our approach can be divided into two principal components -- one due to the random gradient sketch at every iteration, and the other due to the generation of the guidance vector at the refresh steps. 

The generation of the random subspace at every step is done as usual in an SSD-type approach, essentially via $d$ independent Gaussian random vectors (followed by certain orthogonal projections to get the desired structure of the subspace). We discuss the details of the generation of random subspace in Section \ref{sec:sub_alg}; here, we only observe that the cost of each random subspace generation is $O(nd^2)$, which is a similar order as the case of the classical SSD method.   

Once the random subspace is generated at a particular iteration, represented usually as an orthogonal basis, the gradient sketch along this subspace can be computed simply in terms of the directional derivatives along the rows of this matrix. If the computational cost of each directional derivative is $\xi$,  this yields a total cost of $O(nd^2 + d\xi)$ for each SSD step. 

From an algorithmic point of view, it might be even more inexpensive to use independent Gaussian random vectors orthogonal to the guidance vector in order to construct the random projection (as opposed to fully orthogonal columns). For a quick estimate, generation of $d$ such vectors is possible at $O(nd)$ cost (as opposed to $O(nd^2)$ for fully orthogonal vectors). In fact, independent Gaussian random vectors enjoy approximate orthogonality properties which would already suffice for our method to give effective results. However, in this paper we focus mostly on the orthogonal vectors setup for greater conceptual clarity and ease of presentation, leaving the analogous theoretical development of the independent Gaussian vectors alternative for another occasion. For the remaining part of the paper, we assume the cost of generating this matrix to be $O(nd)$, which only differs from that of the orthogonalised version by a factor of $d$, which we typically take to be orders smaller than $n$.

This brings us to the computational cost of the refresh steps. Applications in machine learning and statistics typically involve optimisation objectives that are highly structured. In this work, we investigate some of the most common structures widely present in such objectives -- namely, sparsity and mini-batch structure. We show that well-aligned guidance vectors in these settings can be computed in a fast and relatively inexpensive manner (in comparison to full gradient computations); see the discussion in Section \ref{sec:applications} for details. In the sparse setting with sparsity parameter $s$, this leads to an overall computational cost of $O(s\log(n) \xi \epsilon^{-2})$ (c.f. Theorem \ref{thm:cost_sparse}). In a mini-batch setting, this leads to a total cost of $O\left( \epsilon^{-4\frac{1+\beta}{1-\beta}} (nd + md\xi) + \epsilon^{\frac{-4}{1-\beta}} m'n\xi\right)$, where $m'$ is refresh batch size and may be taken as 
\[
m' > \frac{\gamma_0 \sigma^2}{(1 - \gamma_0)\| \nabla f(x) \|^2},
\]
where $\sigma^2$ is the variance of the stochastic gradient, $m$ is the mini-batch size in each iteration and $\gamma_0$ is the desired initial alignment. (For more details, we refer the reader to Section \ref{sec:mb_general}.) 

Specifically in the sparse setting, the overall computational cost may be seen to be smaller than standard SSD, which is $O(n\xi\epsilon^{-2})$. For a completely unstructured objective function, perhaps the only reliable way to compute a guidance vector would be compute a full gradient at the refresh steps. This cost may be estimated as $O(n\xi)$, which assumes the full gradient is computed via $n$ directional derivatives (c.f. Theorem \ref{thm:cost_vanilla}). It may be noted that even in this ``worst case scenario'' for our approach, the overall cost $O(n\xi \epsilon^{-2})$ is comparable to the standard SSD, and in turn to full space gradient descent.

\subsubsection{Parallelisation and Hardware Efficiency}
An important aspect of our approach is that it yields itself seamlessly to parallelisation techniques, which is a particularly salutary property in view of modern GPU-based computational architectures that are particularly strong in parallel computation. In particular, computing the gradient sketches at each (guided) SSD step consists of $d$ directional derivatives that can be computed independently. In conjunction with the fact that $d$ is typically small (e.g. $O(\log n)$), this leads to practically feasible parallelisation -- the entire sketch may fit comfortably within the parallel capacity of the accelerator unit used. Although the total arithmetic work associated with $d$ directional derivatives need not be equal to that of a single derivative evaluation, their wall-clock cost can be substantially reduced by parallel execution.

Low-dimensional sketches may also offer advantages from the perspective of the memory hierarchy. When the data and intermediate quantities required for the computation fit within the cache or fast on-device memory of the processor (say VRAM of the GPU), repeated transfers to slower levels of the memory hierarchy (say storage devices) can be reduced. Since data movement is an important cost on modern accelerator architectures, this cache locality provides a complementary hardware-level  motivation for keeping the optimisation computation low-dimensional.

In the case of the refresh steps, the computation of the guidance vector involves similar number of evaluations of independent directional derivatives. The number of such directional derivatives is larger than $d$, but as already noted, much smaller than the ambient dimension $n$ in structured settings. As such, the computation of guidance vectors in the refresh steps can also be easily parallelised, which leads to overall beneficial parallelisability properties of our approach.

\subsubsection{Interpolation Between Random Subspace and Full Gradient Descent}
Our method based on persistence of memory can be fruitfully viewed as an interpolation between classical SSD and full space gradient descent. A natural interpolation parameter is given by the initial alignment parameter $\gamma_0$ and the . If $\gamma_0=d/n$ (in the case where the guidance vector is a projection onto a uniform $d-$dimensional subspace), there is no concept of guidance to the random subspace and we are back to the setting of classical SSD. On the other hand, if $\gamma_0=\delta=1$, we are essentially operating with full, exact gradients at each iteration, and we are in the regime of full-space gradient descent.

Our approach focusses, in spirit, on the setting of $\gamma_0$ and $\delta$ being bounded away from both 0 and 1. The proximity to 0 and 1 of these parameters may be though of as modulating how close the method is to classical SSD and full-space GD respectively. In the intermediate regime of these parameters, which is where we work, we can profit from the beneficial aspects of both these standard methods. In structured settings such as those with sparsity and mini-batches, this is possible even at a computational advantage.

\subsubsection{Local Regime and Fast Alignment of Gradient Vectors}
In gradient descent for strongly convex functions, it is reasonably well-understood that if we start from a local neighbourhood of the optimum (in other words, a so-called {\it warm start}), then the normalised gradient $\nabla f(x_k) / \| \nabla f(x_k) \|$ converges to the smallest eigenvector of the Hessian of the objective at the true optimum. The essential reason for this is that, in such local regimes, the evolution of the gradient can be effectively captured by a power iteration involving this Hessian(see, eg, \cite{ortega2000iterate} for related discussions).

The persistence of memory approach to SSD does not lend itself to such a simple geometric recursion; indeed, the presence of the guidance vector significantly complicates the iterative behaviour even in the local regime. Leveraging a detailed and delicate analysis that traces the gradient descent dynamics in the local regime with a modified scaling and step size, we can demonstrate that the alignment between the normalised gradient and the minimum eigenvector of the Hessian still holds true. In this endeavour, we demonstrate that once we are in the local regime (ie, have a warm start), it suffices to fix the guidance vector once and for all. In other words, no refresh steps are necessary for the persistence of memory method to be effective in such a setting. 
We would like to point out that, in contrast, the classical SSD method does not appear to have such alignment properties in local regimes, which once again illustrates the benefits of our persistence of memory approach to SSD.

This makes our method in the setting of warm starts to be particularly attractive even for objective functions without any structure. Indeed, even in large dimensional optimisation problems, it is quite reasonable to justify the computation of the full gradient once and for all -- this will be made at the beginning of descent in the local regime, and will be repeatedly used in all iterations thereafter. For functions with structure, further computational savings may be obtained (bypassing the full gradient computation) in the presence of a salutary spectral properties of the Hessian at the optimum (such as a low-lying spectral gap) or sparsity in the function. The alignment holds with probability at least $1 - e^{-cd\tau^2}$, where $\tau \lesssim \lambda_2 - \lambda_1$. When the bottom gap $\lambda_2 - \lambda_1$ is $\Theta(1)$: for instance, when one feature has low variance while the remaining dimensions are well spread, as is common in representation learning, where the covariance of learned features often has one direction that is nearly degenerate while the rest of the representation remains well-conditioned. In such cases, $\tau$ need not decay with dimension, and the event still holds with high probability. When the function is sparse (i.e. $s$ grows as a small power of $n$ such as $s = n^\iota$), a cost of $O(s\log(n) \times (n\log(n) + \xi))$, as compared to $O(s\log(n) \times (n + \xi))$ to obtain a sufficiently good initial alignment $\gamma_0 = \Theta(1)$.
These considerations are discussed in detail in Section \ref{sec:local}. 

\subsection{Application Domains for Stochastic Subspaces}
The motivation for stochastic subspace methods is not simply the idea that the computation of a small number of directional derivatives can be cheaper than computing a full gradient. Indeed, reverse mode automatic differentiation (AD) has revolutionised gradient computations of differentiable function, by providing a principled way to evaluate the full gradient with an additional cost that is only a small constant multiple of the cost of evaluating the function \cite{baydin2018automaticdifferentiation,griewank2008evaluating}. Nevertheless, there are important settings where the arithmetic cost is not the principal bottleneck. In large-scale optimisation, the dominant cost may instead arise from memory, communication, differentiation through long computational trajectories, optimiser state or the absence of a differentiable gradient oracle. These settings provide natural applications for stochastic subspace methods.

\subsubsection{Memory Constrained Optimisation}
The computational efficiency of reverse mode AD comes with a potentially significant memory requirement. Intermediate quantities from the forward computation must be stored or recomputed during the backward sweep (c.~f.~the chain rule), implying that memory usage can become prohibitive for large or deeply nested computational graphs \cite{baydin2018automaticdifferentiation,griewank2008evaluating}. In contrast, directional derivatives $D_u f(x) = \inner*{\nabla f(x),u}$ can be propagated using forward mode AD without storing the complete reverse-mode tape. Consequently, when the subspace dimension $d$ is small, SSD trades the memory requirements of reverse differentiation for a small collection of forward directional computations. This distinction is particularly relevant in the training of large-models: MeZO demonstrated inference-level memory requirements (compared to training-level memory requirements) for fine-tuning of language-model through forward only computations \cite{malladi2023mezo}, while the more recent SubZero method explicitly uses random low-dimensional perturbation subspaces to improve zeroth-order LLM fine-tuning \cite{yu2025subzero}. Thus, even where full gradients are computationally efficient, subspace methods can enable optimisation when reverse-mode memory is the limiting resource.

\subsubsection{Communication-Constrained Optimisation}
In distributed and federated optimisation, communication can be substantially more expensive than local computations. A standard first-order method may require a worker to communicate vectors (gradient or update direction) in $\mathbb{R}^n$ in every round, with $n$ ranging from millions to billions of parameters in modern models. Suppose communicating parties instead share a subspace $P \in \mathbb{R}^{n \times d}$, such as through common pseudorandom seeds, only the $d$ projections need to be transmitted, reducing the communication cost from $O(n)$ to $O(d)$.

FedKSeed has demonstrated the potential scale of this reduction, via its  ability to perform federated full-parameter tuning of billion-parameter LLMs using random seeds and scalar directional information with less than 18KB bandwith in communication in their experiments \cite{qin2023fedkseed}. More recently, Ferret showed that the same principle can be coupled with local optimisation, by projecting local updates onto a low-dimensional random space before communication \cite{shu2024ferret}. This suggests that for $d \ll n$, substantial communication savings can compensate for the reduced information contained in each subspace update.

\subsubsection{Long-Horizon and Nested Optimisation}
Subspace methods are also attractive when differentiation must pass through a long inner computation, such as in bilevel optimisation, hyperparameter optimisation and meta-learning. Given an outer objective
\[
    \Phi(\lambda) = L_{\mathrm{out}}(x_T(\lambda)),
\]
reverse AD through the $T$ inner iterations may require storing, check pointing or reconstructing the optimisation trajectory $\{x_i(\lambda)\}_{i=1}^T$ \cite{franceschi2017hyperparameter}. Rather than maintaining the complete sensitivity $\frac{\partial x_t}{\partial \lambda}$, a subspace method can propagate only directional sensitivities $\frac{\partial x_t}{\partial \lambda}u_j$ for $j \in [d]$. When $d$ is small, these quantities can be computed sequentially through inner computations, providing a memory-efficient alternative to forming the full gradient on the hyperparameter. This is particularly relevant when both the number of outer variables and the length of the inner optimisation trajectory are large.

\subsubsection{Black-Box and Non-Differentiable Optimisation}
In many applications, the full gradient is not merely expensive, but potentially unavailable; for instance,  in simulation-based objectives, discrete performance metrics, physical experiments and propreitary machine learning models accessible only through inference interfaces. In these cases, directional finite differences
\[
    D_u f(x) \approx \frac{f(x + h u) - f(x- h u) }{2h}
\]
provide a natural mechanism for obtaining optimisation information within a low-dimensional subspace. Recent black-box prompt-tuning methods provide insight into this regime. E.g.,~ZOT performs zeroth-order prompt optimisation using only inference access \cite{zhan2024zot}, while ZIP uses a low-dimensional representation to reduce both the dimension dependence and query cost of zeroth-order prompt tuning \cite{park2025zip}. In such settings, SSD is not competing against an inexpensive full-gradient oracle; rather it provides an optimisation mechanism when such an oracle does not  even exist.

\subsubsection{Optimiser-State Memory}
Large-scale optimisation can also be limited by the auxiliary states needed to be maintained by adaptive optimisers. Methods like Adam store first and second moment estimates with dimension comparable to the parameter vector, introducing an additional $O(n)$ memory requirement. If useful updates are concentrated in a $d$-dimensional subspace, these statistics may indeed be maintained in a compressed representation. The recently proposed method GaLore demonstrates this principle by projecting layer-wise gradient into low-rank spaces to reduce optimiser state memory, while retaining full-parameter training \cite{zhao2024galore}. Although GaLore still computes the full backpropagated gradient, it illustrates the broader potential of subspace representations. An SSD method that directly computes only the required directional information, could in principle reduce both gradient related and optimiser-state memory.

\paragraph{}

Taken together, these examples illustrate that the principal advantage of stochastic subspace methods need not merely be a reduction in floating-point arithmetic. Instead, a low-dimensional optimisation interface can reduce the need to \emph{materialise, store, communicate or even access} full-dimensional derivative information. The resulting challenge is to then retain these computational advantages while selecting subspaces that contain sufficiently rich  information about descent directions, which is precisely the motivation for the guided subspace constructions considered in this work.

\subsection{Related Works}
\paragraph{Coordinate and Random Subspace First-Order Methods.} Coordinate descent replaces a full gradient step by updates along individual coordinates or blocks; its randomised complexity and sampling rules are developed by \cite{nesterov2012coordinate, qu2016arbitrarysampling}, and the survey of \cite{wright2015coordinate}. Greedy Gauss–Southwell selection can improve the rate when the extra selection cost is justified \cite{nutini2015gausssouthwell}. Beyond coordinate subspaces, \cite{kozak2021stochastic} study a stochastic
low-dimensional subspace method for settings in which gradients are not directly available, and \cite{kozak2023orthogonal} analyse zeroth-order optimisation with orthogonal random directions. These methods resample directions; they do not analyse reuse of a single historically informative direction.

\paragraph{Probabilistic Subspace Models.} Probabilistic-model analysis replaces the deterministic model accuracy by a conditional high-probability condition \cite{cartis2018probabilistic}. In a non-convex random subspace framework, \cite{cartis2022randomised} obtain high-probability $O(\epsilon^{-2})$ iteration complexity for safeguarded trust-region or quadratic regularisation schemes under a probabilistic subspace gradient condition. Related derivative-free trust-region analysis in random subspaces is given by \cite{dzahini2024stochastictrustregion}. These results are important comparators for the present alignment assumption, but their safeguards and oracle models differ from a fixed-step projected gradient update.

\paragraph{Gradient Sketches and Variance Reduction.} Another randomised first-order method SEGA builds a variance-reduced gradient estimate by accumulating random linear measurements of the gradient over time \cite{hanzely2018sega}. Its state is a reconstructed estimator, whereas the present method proposes to retain a single guidance direction and complete it by a fresh random orthogonal subspace. Randomised forward-mode gradient estimators based on directional derivatives are studied by \cite{shukla2023randomizedforward}; standard automatic-differentiation cost and memory distinctions are reviewed by \cite{baydin2018automaticdifferentiation, griewank2008evaluating}. Recent work also studies accelerated projected-gradient oracles \cite{omiya2026acceleratedsubspace} and adaptive low-rank subspaces for memory-efficient model training \cite{liang2024onlinesubspace, chen2025memoryefficientrso}.

\paragraph{Relation to Variance-Reduction Methods}
There is also a conceptual connection between our guided subspace construction and control-variate techniques in stochastic optimisation. Classical variance-reduced methods, such as SVRG and SAGA, exploit auxiliary gradient information that is correlated with a stochastic gradient estimator to reduce its variance while preserving the desired expectation \cite{johnson2013svrg,defazio2014saga}. Although our mechanism is different, where the guidance vector modifies the geometry of the sampled subspace rather than correcting a stochastic estimator, the underlying principle conceptually has a related flavour. Namely, auxiliary information correlated with the true gradient is used to improve the quality of the stochastic search direction. Exploring such potential connections is a natural direction  for future work.

\subsection{Structure and Notation of the Paper}
We define the initial point as $x_0$, and the solution space lies in $\mathbb{R}^n$. The lower dimension of the projection is denoted by $d$ (or in our case, $d+1$ as we consider a $d$-dimensional subspace concatenated to a fixed vector $v_k$). The alignment vector is denoted by $v_k$, whereby $\hat v_k$ denotes the normalised vector $v_k/\norm{v_k}$. When necessary, we will denote the normalised form of a vector $u$ by $\hat u$.

We use $L$ as the Lipschitz constant of the derivative of the function $f$, and $\alpha$ denotes the step-size (which may contain a subscript denoting the iteration when applicable). $\delta$ denotes the lower bound on the alignment we have and $\gamma_k$ represents a lower bound on the current alignment in iteration $k$. In the sections involving computational complexity, $\xi$ is assumed to be the cost of computing 1 directional derivative, and $\nu$, the cost of updating the alignment vector.

We present the paper as follows. In Section \ref{sec:results}, we describe the main algorithm and motivation behind it, followed by the main results on the computational cost of the algorithm under structured problems commonly found in machine learning tasks. We then take a look at how we only require to compute a well aligned vector "once and for all" under a local regime. In Section \ref{sec:applications}, we take a look at various oracles to find such a {\it guidance vector} and show that with a simple tweak to the algorithm, it can also be applied to functions of a finite sum structure. In Section \ref{sec:proof_ideas}, we briefly describe the proof techniques involved in showing the results in the previous 2 sections, which is then followed by some numerical demonstrations of the algorithm in Section \ref{sec:experiments}.
\section{Main Results}\label{sec:results}
\subsection{Preliminaries}
We are working under the standard assumption that the function is $L$-smooth, which is the same as saying that the gradient of the function is $L-$Lipschitz. Here, we provide the definition of such functions.
\begin{defn}[$L$-Smoothness]
    A function $f: \mathbb{R}^n \rightarrow \mathbb{R}$ has $L$-Lipschitz gradient if $\exists L > 0$ such that $\forall x, y \in \mathbb{R}^n$, we have
    \begin{equation}\label{eqn:l-smooth-1}
        \|\nabla f(x) - \nabla f(y)\| \leq L \|x - y\|.
    \end{equation}
\end{defn}
As a consequence, for any $x,y \in \mathbb{R}^n$, we have the following inequality:
\begin{equation*}
    f(y) \leq f(x) + \nabla f(x)^\top (y-x) + \frac{L}{2} \|x-y\|^2.
\end{equation*}
In particular, letting $x = x_{k-1}$ and $y = x_k$ gives us
\begin{equation}\label{eqn:lip_ineq}
    f(x_{k}) \leq f(x_{k-1}) + \langle \nabla f(x_{k-1}), x_{k} - x_{k-1} \rangle + \frac{L}{2} \|x_k - x_{k-1}\|^2
\end{equation}
which will be an inequality used frequently in the later proofs.

\subsection{Algorithm}\label{sec:sub_alg}
We will now provide the conditions under which our method converges to a point $x_N$ such that $\min_{k \in [N]} \mathbb{E} \left[ \|\nabla f(x_k)\|^2 \right] < \epsilon^2$. Let $f : \RR^n \rightarrow \RR$ be a $L-$smooth function. At each iteration, the point is updated as
\begin{equation}\label{eqn:update_step}
    x_{k} = x_{k-1} - \alpha P_{k-1}P_{k-1}^\top \nabla f(x_{k-1})
\end{equation}
where $P_{k-1} \in \RR^{n \times (d+1)}$ $(d \ll n)$ is a matrix of the form
\begin{equation}
    P_{k-1} = \begin{pmatrix}
      \hat{v}_{k-1} &  \tilde{P}_{k-1}
    \end{pmatrix},
\end{equation}
and $\Tilde{P}_{k-1}$ is a random $n \times d$ matrix on $v_{k-1}^{\perp}$ such that
\begin{equation}\label{eqn:randmat}
    \Tilde{P}_{k-1}^\top\Tilde{P}_{k-1} = I_d \qquad \& \qquad \EE \left[ \Tilde{P}_{k-1}\Tilde{P}_{k-1}^\top \right] = \frac{d}{n-1}\left(I_n - \hat{v}_{k-1} \hat{v}_{k-1}^\top\right).
\end{equation}

\begin{algorithm}[!h]
	\caption{SSD with Persistence of Memory (SSDPM)}\label{Alg: SSD}
	\begin{algorithmic}[1]
		\State{\textbf{Inputs:} $\alpha, d, \delta, \gamma_0$ \Comment{step size, subspace rank, alignment threshold, initial alignment}}
		\State{\textbf{Initialize:} $x_0$ \Comment{arbitrary initialization}}
        \State $j \gets 0$
        \State $\tilde \delta \gets \delta + \sqrt{\frac{d}{n-1}}$
        \State $r \gets \frac{(1-2\alpha L)}{4 \alpha L (1-\alpha L)}\log \left( \gamma_0/\tilde \delta \right)$ \Comment{Theoretical upper bound on re-uses for alignment vector}
		\For {$k = 1, 2, \ldots$}
        \If{$j = 0$}
            \State Generate a new $v_{k-1}$ as a guiding descent direction
            \State $\hat{v}_{k-1} \gets v_{k-1} / \|v_{k-1}\|$
        \Else
            \State $v_{k-1} \gets v_{k-2}$
        \EndIf
		\State Generate $\tilde{P}_{k-1}\in \RR^{n \times d}$ orthogonal to $v_{k-1}$
        \State $P_{k-1} \gets \begin{pmatrix}
            \hat{v}_{k-1} & \tilde{P}_{k-1}
        \end{pmatrix}$
		\State $x_{k} \gets x_{k-1} - \alpha P_{k-1} P_{k-1}^\top \nabla f(x_{k-1})$
        \State $j \gets (j+1) \bmod r$
		\EndFor
		\end{algorithmic}
\end{algorithm}
The key difference between our algorithm and the original SSD algorithm is in the generation of the guidance vector $v_k$ at each step. With this algorithm, we will then present the main theorem arising from this algorithm. The initial alignment $\gamma_0$ depends on the way the guidance vector is generated, which we will see 2 examples in the later sections.

\subsubsection{Discussion on Algorithm}
In the original SSD algorithm proposed by \cite{kozak2021stochastic}, the iterates are updated in a random projected direction of the gradient at that point. A $d-$dimensional subspace is randomly chosen and the gradient will be projected onto this subspace, whereby the iterates will then move in accordance to this direction. In comparison, our algorithm appends this random subspace with an additional {\it alignment vector} indicated with $v_k$, which has a correlation with the gradient of at least $\delta$ in expectation. More concretely, we have $\mathbb{E}[\inner*{\hat v_k, \nabla f(x_k)}^2] \geq \delta \norm*{\nabla f(x_k)}^2$, where $\hat v_k$ represents the normalised version of the vector $v_k$. Doing so ensures that the projected gradient is of size $(\delta + \frac{d}{n-1}) \norm*{\nabla f(x_k)}^2$ as compared to $\frac{d}{n} \norm*{\nabla f(x_k)}^2$ in the case of the original algorithm by \cite{kozak2021stochastic}. At the same time, to ensure that the matrix remains a projection, the random part of the projection will be sampled randomly from $(I - \hat v_k \hat v_k)$ instead. We denote this random matrix as $\tilde P_k$, and the projection matrix we consider will be
\begin{equation}
    P_k P_k^\top = \hat v_k \hat v_k^\top + \tilde P_k \tilde P_k^\top, \qquad P_k = \begin{pmatrix}
        \hat v_k & \tilde P_k
    \end{pmatrix}.
\end{equation}

Next, we will re-use this {\it alignment vector} for $r$ steps, until the correlation with the gradient falls below $\delta$. In practice, this is a variable which can be tuned by the practitioner, but our analysis in Corollary \ref{cor:ref_bound} (upper bound on reuse) gives us a theoretical upper bound on $r$ given the starting correlation (the correlation of the alignment vector with the gradient at the iteration where the alignment vector is first used). In the most vanilla case, we can simply choose $v_k = \nabla f(x_k)$, and the initial alignment will be $1$.

We would also like to note that the computation of the projected gradient can be much cheaper than the gradient, by first computing the $d$ directional derivatives of the gradient with respect to $P_k$. i.e. $(P_k^\top \nabla f(x_k))_j = (P_k)_{\cdot, j}^\top \nabla f(x_k)$, and then doing a matrix vector multiplication with the matrix $P_k$.

\subsection{Guarantees for Iteration and Computational Complexities}
Our main contribution is that the algorithm enjoys an improved iteration complexity as compared to the original SSD algorithm. Given an oracle that produces an alignment vector $v_0$ satisfying
\[
    \EE \left[ \inner*{\hat{v}_{0}, \nabla f(x_{0})}^2 \mid x_0 \right] 
        \geq \gamma_0 \norm*{\nabla f(x_{0})}^2
\]
for the initial alignment, while also satisfying
\begin{equation}\label{eqn:alignment_parameter}
    \EE \left[ \inner*{\hat{v}_{k-1}, \nabla f(x_{k-1})}^2 \mid x_0 \right] 
        \geq \gamma_{k-1} \EE \left[ \norm*{\nabla f(x_{k-1})}^2 \mid x_0 \right]
\end{equation}
for consequent steps with $\gamma_{k-1} \geq \delta$ (in other words, conditional on the iterate where the alignment vector is updated, consequent alignments are lower bounded by $\gamma_{k-1}$), we obtain rates comparable to the standard GD algorithm which does not have a factor of $\frac{n}{d}$ like the SSD algorithm. More precisely, we have the following theorem.
\begin{thm}[Iteration Complexity]\label{thm:conv_iter}
Let the alignment vector $v_0$ satisfy $\EE \left[ \inner*{\hat{v}_{0}, \nabla f(x_{0})}^2 \mid x_0 \right] \geq \gamma_{0} \|\nabla f(x_{0})\|^2$ with $\gamma_0 > \delta + \sqrt{d/n}$ when updated. Then the number of iterations to obtain a solution satisfying $\min_i \mathbb{E} \left[ \norm*{\nabla f(x_i)}^2 \right] < \epsilon^2$ is at least
\begin{equation}\label{eqn:conv_iter}
    N \geq \frac{\alpha^{-1} \left(1- \frac{\alpha L}{2}\right)^{-1}(f(x_0) - f^\ast)}{\frac{4\alpha L (1-\alpha L)}{(1-2\alpha L)\log(\gamma_0/\tilde \delta)} \left(1 - \frac{d}{n-1}\right) \gamma_0 + \frac{d}{n-1} - \tilde t_0}\epsilon^{-2},
\end{equation}
where $\tilde t_0 = t_0 \left(1 - \frac{d}{n-1}\right) \sqrt{\frac{d}{n-1}} \left(1 - \frac{4\alpha L(1-\alpha L)}{(1-2\alpha L)\log(\gamma_0/\tilde \delta)}\right)$ and $t_0 \in (4/9, 1/2 + O(d/n))$. $\tilde \delta = \delta + \sqrt{\frac{d}{n-1}}$ and $\alpha$ is the step size to be chosen satisfying
\begin{equation}\label{eqn:iter_step_bound}
   \frac{1}{L} \cdot \frac{K}{2 + K + \sqrt{4 + K^2}} < \alpha \leq \frac{1}{2L}, \quad K = \log\left(\frac{\gamma_0}{\tilde\delta}\right)\cdot
\frac{t_0\left(1-\frac{d}{n-1}\right)\sqrt{\frac{d}{n-1}}-\frac{d}{n-1}}
{\left(1-\frac{d}{n-1}\right)\left(\gamma_0+t_0\sqrt{\frac{d}{n-1}}\right)}.
\end{equation}
\end{thm}
\begin{rem}
    The lower bound on the step size is natural, considering that the numerator has a factor $\alpha^{-1}$. If step size is taken to be too small, say order $d/n$, then effectively we get Kozak's rate. Simplifying expression \eqref{eqn:iter_step_bound}, we have
    \[
        \alpha > \alpha^- \approx \frac{K}{4L} \approx \frac{\log(\gamma_0/\tilde \delta)}{9L\gamma_0} \sqrt{\frac{d}{n-1}},
    \]
    which is much larger than the step size of $d/2nL$ used in Kozak's algorithm.
\end{rem}

Observe that in the denominator, apart from a $d/n$ factor, we have an additional term approximately $\frac{4 \alpha L (1-\alpha L)\gamma_0}{(1-2\alpha L)\log(\gamma_0/\tilde \delta)}$. When $d \ll n$, the original SSD suffers from the curse of dimensionality due to the $d/n$ factor. In comparison, by asserting $\gamma_0 \geq \delta = \Theta(1)$ and having $\alpha$ not too small (which can be realised by the upper bound), the denominator is of order $\Theta(1)$, simplifying the iteration complexity above to $O(\epsilon^{-2})$. This is comparable to the standard gradient descent algorithm.
Additionally, $\delta$ controls the extent of how much randomness we want to incorporate into the algorithm; when $\delta \rightarrow \frac{d}{n}$, we recover the original SSD algorithm, where $v_k$ can be thought of the projection of the gradient onto a $d-$dimensional subspace in $\mathbb{R}^n$. On the other hand, when $\delta \rightarrow 1$, we recover the standard gradient descent, where at each iteration, the projected gradient carries the full gradient information $v_k = \nabla f(x_k)$. More specifically, if $\nabla f(x_0) = v_0$, then $\gamma_0 = 1$. On the other hand, if $v_0$ is a random vector in $\mathbb{R}^n$, then we have $\gamma_0 = 1/n$. More discussions of how to choose $v_0$ under specific structure of $f$ will be covered in Section \ref{sec:applications}.

Next, we present the overall computational complexity of the algorithm in the general setting, where there is no additional known structure of the objective function beyond $L$-smoothness. 
\begin{thm}[Computational Complexity]\label{thm:cost_vanilla}
Let $\xi$ be the cost for computing a directional derivative and $\nu$ be the cost of a refresh. Then, the computational cost for computing $N$ steps denoted by Equation \ref{eqn:conv_iter} is of the order
\begin{equation}\label{eqn:cost_vanilla}
    O\left( \epsilon^{-2} \gamma_0^{-1} \left[ d\log\left(\frac{\gamma_0}{\delta}\right) \left( \xi + n \right) + \nu \right] \right).
\end{equation}
\end{thm}
In particular, if we do not have a good way to obtain a refresh, and simply assume that we use $n$ directional derivatives to compute the gradient in the elementary coordinate space, we regardless retrieve the cost of the standard GD algorithm. If $\xi = O(n)$, we can simplify Equation \ref{eqn:cost_vanilla} to $O(\epsilon^{-2} n \xi)$. We note that the focus of our algorithm is in the special cases where the function has an underlying structure which allows for a guidance vector to be obtained much cheaper. We first present their computational complexities here, before going into details in Section \ref{sec:applications}.

Suppose the objective function is $s$-sparse, meaning that it only depends on $s$ of the coordinates of the input space, we have an improved complexity contributed by the reduced cost of a refresh step in comparison to computing a full gradient.
\begin{thm}[Sparse Computational Complexity]\label{thm:cost_sparse}
Let the objective function be $s$-sparse and the cost to compute a directional derivative to be $\xi$. The expected computational cost for the algorithm to produce $\min_{k} \mathbb{E} \left[ \norm*{\nabla f(x_k)}^2 \right] \leq \epsilon^2$ is given by
\begin{equation}\label{eqn:comp_cost_sparse}
    O\left( \frac{\epsilon^{-2}}{\gamma_0} \left[ \left( d\log(\gamma_0/\delta) + k' \log \left(\frac{1}{1-\gamma_0}\right) \right) \times n + \left( d\log(\gamma_0/\delta) + k' \right) \times \xi \right] \right)
\end{equation}
where $k' \gtrsim s\log(n/s)$ and $\gamma_0$ is the initial alignment desired.
\end{thm}
\begin{rem}
    In the scheme where $d = O(\log(n))$ and $\gamma_0 = \Theta(1)$, the first term for each of the factors for $n,\xi$ is dominated by the $k'$ term, giving us a complexity of $O(\epsilon^{-2} s \log(n) (n+\xi))$. Assuming that $\xi = O(n)$ minimally, since the function is $n-$dimensional, and computationally it requires at least $O(n)$ operations on the input space, then the complexity is $O(s\log(n)\xi \epsilon^{-2})$, scaling in the order of the sparsity $s$ instead of the full dimension $n$.
\end{rem}

For an objective function with a finite sum structure, we can perform a mini-batch variant of our proposed algorithm to obtain the following computational complexity.
\begin{thm}[Mini-batch Computational Complexity]\label{thm:mb_cost}
    Let $\nu$ be the cost of updating the alignment vector and $\xi$ be the cost of computing a directional derivative for one function $f_i$. Then, the cost of the algorithm to achieve a point $\min_k \mathbb{E} \left[ \norm*{\nabla f(x_k)}^2 \right] < \epsilon^2$ is given by
    \begin{equation}
    O\left( \epsilon^{-4\frac{1+\beta}{1-\beta}}  (nd + md\xi) + \epsilon^{\frac{-4}{1-\beta}}\nu \right),
    \end{equation}
    where $m$ is the batch size of the algorithm and $\beta \in (0,1)$ fixed. If $\nu = O(m' n \xi)$ (cost of computing $n$ directional derivatives for $m'$ functions), we have
    \begin{equation}
        O\left( \epsilon^{-4\frac{1
        +\beta}{1-\beta}} (nd + md\xi) + \epsilon^{\frac{-4}{1-\beta}} m'n\xi \right).
    \end{equation}
\end{thm}

\subsection{Local Regimes and Gradient Alignment Phenomena}\label{sec:local}
The main bottleneck of this algorithm is the fact that we have to update the alignment vector after a certain number of steps. However, it turns out that under special circumstances, we forgo the need to update this vector and it is still sufficiently well-aligned to the gradient. 

It is well-understood (Theorem 10.1.3 in \cite{ortega2000iterate}) that in gradient descent for strongly convex functions, if we start from a point where the gradient has non-zero correlation with the smallest eigenvector $u_1$ of the Hessian at the optimal solution, and if this point is sufficiently close to the optimal solution (in the local regime), then we have that the normalised gradient $\nabla f(x_k) / \| \nabla f(x_k) \|$ converges to $u_1$. We will show that our algorithm also enjoys a similar result in a finite horizon. Define the finite horizon
\begin{equation}
    K_\epsilon = \min \{k \mid \|\nabla f(x_k)\| \leq \epsilon\}.
\end{equation}

We will also assume additional properties of the function and the alignment vector, namely the function $f \in C^3$ and a weakly correlated alignment vector exists at every step with high probability, which increases to 1 as the correlation increases to 1.
\begin{assump}\label{assmp:align_prob}
    There exists an alignment vector with alignment $\delta$ exists with high probability. More formally, for all $k \leq K_\epsilon$, denote the event
    \begin{equation}
        \mathcal{K}_k(\delta) := \left\{\inner{\hat{v}_k, \nabla f(x_k)}^2 \geq \delta \norm*{\nabla f(x_k)}^2 \right\}
    \end{equation}
    and assume
    \begin{equation}\label{eqn:align_prob}
        \mathbb{P}\left[ \mathcal{K}_k(
        \delta)\mid \nabla f(x_k)\right] \geq 1 - p_{v}(\delta),
    \end{equation}
    where $p_v(\delta)$ is an exponentially decreasing function of $\delta$.
\end{assump}

\subsubsection{Strong Alignment in Local Regimes}
For this section, we will assume that the function is strongly convex and that the Hessian at the optimal point, denoted by $H$, has eigenvalues and eigenvectors $\{(\lambda_i, u_i)\}_{i=1}^n$, where $\lambda_1 < \lambda_2 < \cdots < \lambda_n$. Under this condition, we show that for any given threshold $\delta_0$, if $\inner{\nabla f(x_0), u_1} \neq 0$, then there exists a $k_1$ iterate such that $\inner{u_1, \nabla f(x_{k_1})}^2 \geq \delta_0 \norm*{\nabla f(x_{k_1})}^2$. What this means is that if the initial gradient is not orthogonal to $u_1$, eventually we have that the gradient of the function is strongly aligned to this vector.

This strong alignment in the $k_1^{th}$ step allows us to {\it freeze} the alignment vector. We show that even after this step, the gradient $\nabla f(x_k)$ remains strongly aligned with $u_1$, coupled with the fact that $v_{k} = v_{k_1}$ is also strongly aligned with $u_1$, we naturally get that $\nabla f(x_k)$ is also strongly aligned with $u_1$. The next proposition gives us a lower bound on this alignment.

\begin{prop}\label{prop:sum_angle}
    Let $u,v,w$ be unit vectors in $\mathbb{R}^n$ such that
    \begin{equation*}
        \inner{u,v}^2 \ge {\delta_1}, \qquad \inner{v,w}^2 \ge {\delta_2}.
    \end{equation*}
    If $\delta_1 + \delta_2 \ge 1$, then
    \begin{equation}\label{eqn:sum_angle}
    \inner{u,w}^2 \ge \left(\sqrt{\delta_1\delta_2} - \sqrt{(1-\delta_1)(1-\delta_2)}\right)^2.
    \end{equation}
\end{prop}

Note that the condition that $\delta_1 + \delta_2 \geq 1$ is natural, as we would expect that if the sum of the 2 angles is larger than 1, effectively in the worse case, $u, v$ will be orthogonal. Nonetheless strong alignment between $u_1, \nabla f(x_k)$ and $u_1, v_{k_1}$ means that we no longer have to refresh the alignment vector after the $k_1^{th}$ step, effectively making the algorithm as cheap as the original Kozak's algorithm, while enjoying the faster convergence due to the alignment vector.

Following that, we state the main theorem below.
\begin{thm}\label{thm:local-norefresh}
    Let $f \in C^3$ and strongly convex with parameter $\mu$ and $x^\ast$ be the unique minimiser. Suppose the Hessian of $f$ near $x^\ast$: $H = \nabla^2 f(x^\ast)$, is locally Lipschitz around radius $r_0 > 0$ with simple eigenvalues $\lambda_1 < \cdots < \lambda_n$. Let the initial gap satisfy $f(x_0) - f^\ast \leq \frac{1}{2} \mu r_0^2$, where $f^\ast$ is function value at $x^\ast$. Consider the finite horizon $K_\epsilon = \min \{k \mid \norm*{\nabla f(x_k)} \leq \epsilon\}$ and suppose Assumption \ref{assmp:align_prob} holds in this horizon, with $\delta$ satisfying
    $$\delta \gtrsim 1 - \min{ \left(\omega\frac{\mu_1}{\mu_2(1+\eta)}\frac{\lambda_2 - \lambda_1}{ \lambda_1},  \frac{\lambda_2 - \lambda_1}{(1+\eta) \omega^{-1} \frac{\mu_2}{\mu_1} \lambda_1 + (1-\omega^2)^{-1/2} \lambda_n \mu_1}\right)}^2,$$
    where $\omega > 0$ is the initial alignment $\abs*{\inner*{u_1, \nabla f(x_0)}} = \omega \norm*{\nabla f(x_0)}$ and $\mu_i = 1 - \alpha \lambda_i$. Then with Algorithm \eqref{Alg: SSD} with $\alpha \in (0, 1/\lambda_n)$, given $\delta_0 \in (1/2,1]$, there exists an iteration $k_1 < K_\epsilon$ such that
    \begin{equation}
        \mathbb{P} \left[\langle \nabla f(x_{k_1}), u_1 \rangle^2 \geq \delta_0 \|\nabla f(x_{k_1})\|^2\right] \geq 1 - e^{-\frac{1}{8}k_1 (1 - p_{decay}(\hat \tau, \delta, d))},
    \end{equation}
    where $p_{decay}(\hat \tau,\delta, d) = e^{-cd \hat \tau^2} + p_v(\delta)$, and $\hat \tau \in (0,1)$. Furthermore, for $k > k_1$, we fix the alignment vector
    \[
        v_k = v_{k_1} = \nabla f(x_{k_1})
    \]
    and scale the random part of the matrix $P_k$ by a factor of $\sqrt{\frac{n-1}{d}}$, i.e. $\tilde P_k \rightarrow \sqrt{\frac{n-1}{d}} \tilde P_k$ . Then with probability at least
    \[
    \left(1 - \sum_{i = 1}^n e^{-cd\tau_i^2} -e^{-cd\tau^2}\right)^{K_\epsilon - k_1},
    \]
    for $k \in (k_1, K_\epsilon)$, we still have $\frac{|u_i^\top \nabla f(x_k)|}{|u_1^\top \nabla f(x_k)|} = O(\lambda_i \tau_i)$, where
    \begin{align*}
        \tau_1 &\lesssim \frac{1-\alpha \lambda_1}{1 - \alpha \lambda_2} \frac{\lambda_2 - \lambda_1}{\lambda_1} \sqrt{\delta_0} \\
        \tau_i &\lesssim \frac{\lambda_i - \lambda_1}{\lambda_i} \sqrt{1-\delta_0} \\
        \alpha_{k_1} &= \frac{d}{L}\frac{1 - 4\tau \delta_0(1-\delta_0)}{(n-1)(-\tau + 4\delta_0(1-\delta_0)(1+\tau)) + d(2\delta_0-1)^2},
    \end{align*}
    with some $\tau \in (0,1)$ and $\alpha_{k_1}$ being the step size under the rescaled random matrix to ensure exponential decay of the gradient. The $\lesssim$ hides some constant and exponentially decaying term under the local regime, and the constant $c > 0$ is a universal independent of all other variables. The projected dimension $d$ is also required to fulfill
    \begin{equation}
        d \geq \frac{1}{c} \max \left\{ \log(n+1) \cdot \max \left\{ \frac{1}{\delta_0} \left(\frac{\lambda_1}{\lambda_2 - \lambda_1} \frac{1 - \alpha \lambda_2}{1 - \alpha \lambda_1}\right)^2, \frac{\max_{2 \leq i \leq n} \left(\frac{\lambda_i}{\lambda_i - \lambda_1}\right)^2}{(1-\delta_0)}, \frac{1}{\tau^2}\right\}, \frac{1}{\hat \tau^2} \log\left(\frac{1}{1-p_v(\delta)}\right)\right\}.
    \end{equation}
\end{thm}
\begin{rem}
    This means that we can control the alignment of the gradient with the non leading eigenvectors to be arbitrarily small, by controlling the closeness of $\delta_0 \rightarrow 1$. In other words, the gradient remains weakly correlated with the leading eigenvector, and consequently, is weakly correlated with the {\it frozen} alignment vector.
\end{rem}
\begin{rem}
   Although the step size $\alpha_{k_1} \approx \frac{d}{nL\delta_0(1-\delta_0)}$ is larger than the step size when $k < k_1$, it can be arbitrarily close to it by taking $\delta_0 \rightarrow 1$. This means that our step size does not suffer the same issue as in the Kozak's SSD algorithm, which is of order $\frac{d}{n}$.
\end{rem}
\begin{rem}
    Under certain scenarios (which we will discuss in the following section), the bottom eigengap $\lambda_2 - \lambda_1 = \Theta(1)$. This means that the bounds on $\tau_i$ solely depend on the alignment parameter $\delta_0$. In fact, since $\lambda_i > \lambda_2$ for $i > 2$, $\tau_i$ for the corresponding $i's$ is also allowed to be larger, making the success probability larger. Suppose $d = \frac{(\log(n))^m}{c\min_i\tau_i^2}$ for $m > 1$, then $d$ need not be too large to achieve a probability of $(1 - n^{1-m})^{K_\epsilon - k_1}$.
\end{rem}

\subsubsection{Application Scenarios in which Local Regime is Natural}
We will take a look at some applications where the structure of the problem leads to advantage in the local regime.

\paragraph{The Order of the Bottom Spectral Gap.}

For a matrix $H$, define the \emph{bottom gap} as
$$\delta(H) = \lambda_2(H) - \lambda_1(H).$$
In problems where a single feature direction is unusually weak (nearly degenerate) while the remaining directions are uniformly well-conditioned (variance of order one), we have $\delta(H) = \Theta(1)$. Spectral models consisting of isolated eigenvalues separated from a bulk have been studied extensively in random matrix theory and high dimensional statistics, such as through spiked covariance models and finite rank perturbations \cite{baik2005phase}. 

Such a spectral structure can be motivated by several settings in machine learning and statistics. For example, in linear regression with highly correlated inputs, a particular combination of nearly redundant features may have small but non zero variance, while the remaining directions retain $O(1)$ variance. More generally, how the non uniformity of the spectrum of the covariance affects overparameterised linear and kernel regression has been studied in benign and tempered overfitting \cite{bartlett2020benign, mallinar2022taxonomy, tsigler2023benign}. Related weak directions can also arise in latent variable models, such as VAEs, where \emph{posterior collapse} happens when one or more latent variables become uninformative. This has been related to the geometry of the objective function and the covariance structure of the data \cite{lucas2019dontblame}. Although these examples do not immediately imply an isolated bottom eigenvalue of the Hessian, it motivates the considering of spectral models containing a distinguished weak direction.

A rather straightforward example to look at is ridge regression models, which naturally lead to positive definite matrices of the form
\[
    H = \frac{1}{m} X^\top X + \lambda I, \qquad X = [x_1, x_2, \cdots x_m]^\top\in \mathbb{R}^{m \times n}.
\]
With $S = \frac{1}{m} X^\top X$, we get
\[
    \delta(H) = \lambda_2(S) - \lambda_1(S) = \delta(S).
\]
Essentially, this means that the regularisation does not change the bottom-gap, despite making the associated function strongly convex. In practice, $X$ might be our data set which is sampled from some population with covariance $\Sigma \in \mathbb{R}^{n \times n}$. 
Using the {\it spiked model} from random matrices as a concrete example, let the rows $x_i \in \mathbb{R}^n$ of $X$ be sampled independently from
\[
    x_i \sim N(0, \Sigma), \qquad \Sigma = \sigma^2 I_n - c \cdot u u^\top,
\]
where $u$ is a fixed unit vector, $\sigma^2 > c$ and $c > 0$ are fixed constants. Then $\lambda_1(\Sigma) = \sigma^2 - c$ and $\lambda_i = \sigma^2$ for $i \geq 2$, implying that the population bottom-gap is $c$ (i.e. $\delta(\Sigma) = c$). In this case, a sufficient condition on the bottom gap for $\Sigma$ ensures that the sample covariance also has a gap of the same order. Moreover, by Weyl's inequality
\[
    \delta(S) = \lambda_2(S) - \lambda_1(S) \geq \lambda_2(\Sigma) - \norm*{S - \Sigma}_{op} - \lambda_1(\Sigma) - \norm*{S - \Sigma}_{op} = \delta(\Sigma) - 2 \norm*{S - \Sigma}_{op}
\]
When $S$ concentrates around $\Sigma$ (i.e. $\norm*{S - \Sigma}_{op} = o(c)$), we get
\[
    \delta(H) = \delta(S) = c + o(1),
\]
implying that the bottom-gap for the problem is asymptotically $c$. \cite{koltchinskii2017concentration} showed that for iid Gaussian rows with covariance $\Sigma$,
\[
    \norm*{S - \Sigma}_{op} \lesssim \norm*{\Sigma}_{op} \left(\sqrt{\frac{r_{\text{eff}}}{m}} + \frac{r_{\text{eff}}}{m} \right), \qquad r_{\text{eff}} = \frac{\text{tr}(\Sigma)}{\norm*{\Sigma}_{op}}.
\]
Equivalently, if $\norm*{\Sigma}_{op} \left(\sqrt{\frac{r_{\text{eff}}}{m}} + \frac{r_{\text{eff}}}{m} \right) \ll c$, then $\delta(H) \approx c = \Theta(1)$. The same idea can be extended to generalised linear models of the form
\[
    H = \frac{1}{m} A^\top D^\ast A + \lambda I.
\]

\paragraph{Fast Estimation of Alignment Vector in Sparse Functions.}
Here, we show that in structured scenarios, a sufficiently good "once for all" estimation of the gradient could be computed relatively cheaply. In some sparse setups where the sparsity $s$ grows as a factor of $n$. i.e. $s = n^\iota$ for some small $\iota > 0$. As we will see later in Section \ref{sec:application_sparse}, we have the alignment between the estimation and the vector to be
\[
\gamma_0 = \left(1 - Ce^{-ck'}\right)(1-\rho^2),
\]
where $k'$ is the number of rows of the sensing matrix of the order $k' = O(s\log(n/s))$, $\rho$ is the error incurred by the reconstruction algorithm and $C,c > 0$ are universal constants.
Suppose we want an alignment $\gamma_0 = (1-Ce^{-ck'})(1 - \Theta(d/n))$. Equivalently, $\rho = \sqrt{\Theta(d/n)}$, then the number of iterations required is
\[
    T = \log \left( 1/\rho \right)
        = \frac{1}{2}\log \left( \frac{1}{\Theta(d/n)} \right) = \Theta(\log(n/d)).
\]
Coupled with the cost of each iteration in the recovery algorithm, we have that the cost to estimate such a vector is $O(T \times k'n + k'\xi) = O(s\log(n/s) \times (n\log(n/d) + \xi))$.
\section{Applications to Machine Learning}\label{sec:applications}

\subsection{Structured Optimisation and Fast Generation of Guidance Vectors}
In machine learning, the optimisation problems arising naturally out of inferential problems (or otherwise) typically have structural properties that make them amenable to algorithmic solutions, despite their typically high ambient dimensionality and other superficial complexities. A canonical example of this is accorded by the notion of sparsity, whereby a function depends only on a (relatively small) subset of the ambient coordinates, and its generalisation to the so-called manifold hypothesis, which posits that real world data may be envisaged to come from some low-dimensional manifold (whose specifics would in general be unknown to the practitioner). The natural goal, which has been achieved to a significant degree of success in the machine learning and statistics literature, is to leverage such structural properties to the effect that the complexity of algorithms scale with the intrinsic dimensionality (as opposed to the ambient dimensionality, which is typically much larger).

In this work, we bring this philosophy to bear on SSD approaches to optimisation, focussing on two foundational structures nearly ubiquitous in optimisation for ML -- sparsity and a mini-batch structure. We demonstrate that these structures can be effectively leveraged for fast and inexpensive generation of the guidance vector, which is a key component of our persistence of memory based approach to sparse SSD. This makes our method computationally attractive, with the additional advantage that our guidance vector generation mechanisms are also very strongly parallelisable, which is an additional benefit with regard to modern GPU-based or distributed computing architectures.

En route, we also establish to our knowledge the first theoretical analysis of classical SSD methods in the setting of sparse functions, which could be of independent interest.

\subsection{Sparse Functions}\label{sec:application_sparse}
\subsubsection{Structure of the Problem and Motivations}
A common structure exploited in machine learning objectives is when the function itself depends on only a small number of directions in its input space, even though it is nominally defined on a high-dimensional domain. 

We say $f: \mathbb{R}^n \to \mathbb{R}$ is \emph{sparse} (or has \emph{low-dimensional structure}, which can also be called "functions with low dimensionality" \cite{wang2016bayes}) if there exists an orthogonal projection matrix $\Pi$ on a $\mathbb{R}^s$ subspace ($s \le n$) such that for all $x \in \mathbb{R}^n$, 
\[
    f(x)=f(\Pi x).
\]
Let $R \in \mathbb{R}^{s \times n}$ be a matrix formed from an orthogonal basis of $\Pi \mathbb{R}^n$. We define $g: \mathbb{R}^s \mapsto \mathbb{R}$:
\[
    g(y)=f(R^\top y).
\]
Since $f(x)=f(\Pi x)$, we have that 
\begin{equation}\label{eqn:sparse_defn}
    g(Rx)=f(x).
\end{equation}
That is, $f$ only varies along the $s$-dimensional subspace spanned by the columns of $R$, and is constant along all directions orthogonal to it. This structure is often called a \emph{multi-index model} in the statistics literature, with the columns of $R$ referred to as the \emph{indices} or \emph{relevant directions}. Some examples include:

\begin{enumerate}
    \item {\bf Single-Index Models} ($s=1$): $f(x) = g(r^\top x)$ for a single direction $r \in \mathbb{R}^n$, such as generalised linear models $f(x) = \sigma(r^\top x)$ for a link function $\sigma$, or the activation of a single neuron; see e.g.\ \cite{hardle1993optimal} for classical estimation theory in this setting.
    \item {\bf Finite-Index Models} ($s > 1$ fixed): These are extensions of the single-index model, closely related to \emph{projection pursuit} \cite{diaconis1984nonlinear}, where for fixed $s > 1$, we have index vectors $\{w_1, \cdots, w_s\} \subset \mathbb{R}^n$ and $f$ is a function on the indices $w_j^\top x$, i.e. $f(x) =g(w_1^\top x, \cdots, w_s^\top x)$. Neural networks with one hidden layer are an example of a finite-index model: given weights $v_j \in \mathbb{R}^n, a_j \in \mathbb{R}$ and biases $b_j \in \mathbb{R}$, a one-hidden-layer neural network can be expressed as
    \[
        y = \sum_{j=1}^s a_j \sigma(v_j^\top x + b_j),
    \]
    where $\sigma$ is the activation function. Then, the indices are simply $z_j = v_j^\top x + b_j$ and the function is given by $g(z_1, \cdots, z_s) = \sum_{j=1}^s a_j \sigma(z_j)$.
    \item {\bf Axis-Aligned Sparsity}: When $R$ is restricted to a subset $S \subset \{1, \dots, n\}$ of $|S| = s$ coordinate directions (i.e.\ columns of the identity), this recovers the more familiar notion of coordinate sparsity, where $f$ depends on only $s$ of its $n$ input coordinates, as exploited by sparse regression methods such as the lasso \cite{tibshirani1996regression}.
    \item {\bf High-Dimensional Regression and Variable Selection.} In genomics and biomedical statistics, one often wishes to predict a phenotype or clinical outcome from a feature vector $x \in \mathbb{R}^n$ where $n$ (e.g.\ the number of measured genes, SNPs, or biomarkers) vastly exceeds the number of available samples. Domain knowledge typically suggests that only a small number $s$ of features, or linear combinations thereof, are causally relevant, motivating models of exactly the form \eqref{eqn:sparse_defn} -- both for statistical identifiability with limited samples and for interpretability of the resulting model \cite{fan2010selective}.
\end{enumerate}

Beyond these settings where sparsity is an explicit modelling assumption, such functions are frequently encountered in many applications. For instance, the loss functions of neural networks often have low rank Hessians \cite{gur2018gradient, sagun2017empirical, papyan2018full}. This phenomenon is also prevalent in other areas such as hyper-parameter optimization for neural networks \cite{bergstra2012search}, heuristic algorithms for combinatorial optimization problems \cite{hutter2014efficient}, complex engineering and physical simulation problems as in climate modeling \cite{knight2007parameter}, and policy search \cite{frohlich2019bayes}.

The sparse structure induces a restriction on the gradient of $f$. By the chain rule,
\[
    \nabla f(x) = R^\top \nabla g(R x),
\]
so $\nabla f(x)$ always lies in the $m$-dimensional row space of $R$, regardless of $x$. This is the key structural fact: although $f$ is defined on $\mathbb{R}^n$, its entire first-order behavior is confined to a $m$-dimensional subspace. In particular, if one knew $R$ in advance, optimising $f$ would reduce to optimising the $s$-dimensional function $g$ -- a dramatic reduction in complexity when $s \ll n$.

In practice, of course, $R$ is unknown and must be estimated alongside $g$, and much of the algorithmic interest in this setting lies precisely in how to \emph{identify} the relevant subspace efficiently; for instance using zeroth- or first-order queries whose number scales with $s$ rather than $n$.

\subsubsection{Fast Guidance Vectors in Sparse Settings}
Here, we look at how to obtain a guidance vector much cheaper than computing the full gradient when the sparsity of the problem is {\it axis-aligned} (i.e. the function only depends on a number of fixed coordinates, and $R$ is made up of coordinate vectors). The key to obtaining such a vector hinges on a result from compressed sensing, which allows us to measure a sparse vector up to an error relative to the vector norm, with cost that scales in the order of the $O(s\log(n/s) (n + \xi))$, in comparison to the $O(n\xi)$ cost for computing $n-$directional derivatives. (When $\xi = O(n)$, we have that the former scales much better in terms of $n$ than the latter.)

Given a sparse signal (vector) $y^\ast$, compressed sensing aims to recover this vector from linear measurements $x = \Psi y^\ast + e$, where the matrices $\Psi$ satisfy the restricted isometry property (RIP). A matrix $\Psi \in \mathbb{R}^{k' \times n}$ is said to satisfy the RIP with constant $\delta_s$ if for every $s-$sparse vector $z \in \mathbb{R}^n$,
\begin{equation}
    (1-\delta_s) \norm*{z}^2 \leq \norm*{\Psi z}^2 \leq (1+\delta_s) \norm*{z}^2.
\end{equation}
Some examples include the Gaussian, subsampled Hadamard and partial Fourier matrices. For example, if $\Psi$ is a Gaussian matrix, with entries $\Psi_{ij} \sim N(0, 1/k')$. For $k' \gtrsim s \log(n/s) \eta^{-2}$, we have $\mathbb{P} \left[\delta_{s} \leq \eta \right] \geq 1-C\exp(-ck')$, where $C, c$ are universal constants \cite{baraniuk2008ripgauss}.

Many compressed sensing recovery algorithms (CoSaMP, IHT, HTP, ...) produce $y_T$ with
\begin{equation}\label{eqn:cs_recov_error}
    \norm*{y_T - y^\ast} \leq \rho \norm*{y^\ast},
\end{equation}
after $T$ iterations at a cost of $O(\nu(\rho))$, when $e = 0$, $y^\ast$ is exactly $s-$sparse, provided that $\Psi$ satisfies a RIP condition of the relevant order with probability $\geq p$. The small error means that these 2 vectors have lower bounded correlation. To see this, we observe that $y_T$ lies in a ball of radius $\rho \norm*{y^\ast}$ around the vector $y^\ast$. The largest possible angle between $y_T$ and $y^\ast$ occurs when $y_T$ is tangent to the ball around $y^\ast$. Denoting the angle between these 2 vectors as $\theta$, we find that
\begin{equation*}
    \sin \theta \leq \frac{\rho \norm*{y^\ast}}{\norm*{y^\ast}} = \rho.
\end{equation*}
Thus,
\begin{equation*}
    \frac{\inner*{y_T, y^\ast}^2}{\norm*{y_T}^2 \norm*{y^\ast}^2}
        = \cos^2 \theta 
        \geq 1 - \rho^2,
\end{equation*}
and correspondingly
\begin{equation}
    \inner*{\hat {y_T}, y^\ast}^2 \geq (1-\rho^2) \norm*{y^\ast}^2.
\end{equation}
Applying this to our algorithm, we let $y^\ast = \nabla f(x_0)$, $x = \Psi \nabla f(x_0)$ and $y_0 = 0$. Then, with $v_0 = y_T$, we will obtain
\begin{equation*}
    \inner*{\hat v_0, \nabla f(x_0)}^2 \geq (1-\rho^2) \norm*{\nabla f(x_0)}^2.
\end{equation*}
Of course, this only occurs conditioned on the event $\mathcal{A}$ that the algorithm satisfies Equation \eqref{eqn:cs_recov_error} (which is normally only dependent on the choice of $\Psi$ used and independent of $x_0$). In the complement event, we can use 0 as a lower bound to obtain,
\begin{equation}\label{eqn:sparse_gamma_0}
    \mathbb{E} \left[\inner*{\hat v_0, \nabla f(x_0)}^2\right] 
        \geq \mathbb{E} \left[\inner*{\hat v_0, \nabla f(x_0)}^2 1_{\mathcal{A}}\right] 
        \geq (1-\rho^2) \mathbb{E} \left[\norm*{\nabla f(x_0)}^2 1_\mathcal{A} \right]
        = \gamma_0 \mathbb{E} \left[ \norm*{\nabla f(x_0)}^2 \right],
\end{equation}
where $\gamma_0 = (1-\rho^2) \times p$, with $\mathbb{P} [\mathcal{A}] \geq p$.

\paragraph{Iterative Hard Thresholding (IHT).} The algorithm is as follows: starting with $y_0 = 0$ and $x = \Psi y^\ast$, the iterates $y_t$ are updated by
\begin{equation}\label{eqn:iht_init}
    y_{t+1} = H_s \left[ y_t + \mu \Psi^\top (x - \Psi y_t) \right],
\end{equation}
where $\mu = \frac{1}{1 + \delta_s}$ and $H_s$ is the thresholding function that sets all but the top $s$ elements to $0$. By Corollary 1 of \cite{davies2009iht}, if $\Psi$ has RIP with $\delta_{3s} < 1/15$, then $y_t$ satisfies
\begin{equation}
    \norm*{y_t - y^\ast} \leq 2^{-t} \norm{y^\ast}.
\end{equation}
This means that $T = \left\lceil \frac{\log(1/\rho)}{\log(2)} \right\rceil$ to get Equation \eqref{eqn:cs_recov_error}. The follow proposition follows for the cost of a single refresh:
\begin{prop}\label{prop:sparse_refresh_cost}
    Let $f$ be a $s-$sparse function and let $\xi$ be the cost of computing a directional derivative. For an initial alignment $\gamma_0 \in (0,1)$, the IHT algorithm with Gaussian matrix of appropriate variance requires a computational cost of
    \begin{equation}\label{eqn:sparse_refresh_cost}
        O \left( \log\left(\frac{1 - Ce^{-ck'}}{1 - Ce^{-ck'} - \gamma_0}\right) k' n + k'\xi\right)
    \end{equation}
    where $k' \gtrsim s\log(n/s)$ and $C,c>0$ are constants independent of $n,s$.
\end{prop}

\begin{rem}[Structured Random Matrices]
While Gaussian $\Psi$ is convenient for analysis, it requires $O(k'n)$ storage and $O(k'n)$ time per matrix-vector product. Structured alternatives allow both to be reduced substantially, at the cost of a worse (but still logarithmic) dependence on $n$ in the sample complexity $k'$.

A \emph{subsampled Fourier} (or Hadamard) matrix formed by selecting $k'$ rows uniformly at random from the $n\times n$ discrete Fourier (Walsh-Hadamard) transform and rescaled by $1/\sqrt{k'}$, satisfies RIP of order $s$ with constant $\eta$ with high probability provided
\begin{equation}
    k' \gtrsim s\log^2(s)\log(n)\,\eta^{-2},
\end{equation}
the current best known bound, due to \cite{haviv2017restricted} (improving on the earlier $s\log^4(n)$-type bounds of \cite{rudelson2008sparsereconstruction}) . This is known to be close to optimal; \cite{blasiok2019improved} showed that $k' = \Omega(s\log s\log(n/s))$ rows are necessary for subsampled Hadamard matrices to satisfy RIP at all. Crucially, both $\Psi z$ and $\Psi^\top w$ can be computed via the FFT in $O(n\log n)$ time which is independent of $k'$, rather than the $O(k'n)$ required for a dense Gaussian $\Psi$. This improves the per-iteration IHT cost in Proposition~\ref{prop:sparse_refresh_cost} from $O(k'n)$ to $O(n\log n)$ whenever $k' = \omega(\log n)$ (the typical regime, since $k'\gtrsim s\log(n/s)$).
\end{rem}

\subsubsection{Analysis of Classical SSD in Sparse Settings}\label{sec:sparse_kozak}

Here we analyse the classical SSD method (i.e., \cite{kozak2021stochastic}) under the assumption that the function $f$ is sparse in some {\it unknown} basis (i.e. the function has a {\it low intrinsic dimension} and not necessarily limited to the axis-aligned sparsity discussed above). More precisely, we do not have any constraints on $R$ other than it is an $s-$dimensional orthogonal matrix in $\mathbb{R}^n$. Leveraging Equations \ref{eqn:sparse_defn}, we can actually prove that if $f$ has a low intrinsic dimension $s$, then the original SSD method only suffers from a $\frac{s}{d}$ multiplicative factor, with respect to the number of iterations, instead of the $\frac{n}{d}$ multiplicative factor.

More precisely we have the following theorem:
\begin{thm}\label{thm:sparse_classic}
Assume that $f$ has an intrinsic dimension of $s<n$. Let $\Pi,R,g$ be defined as above. Let $d,\alpha$ satisfy
\begin{equation*}
    \max\!\left\{1,2\log\!\left(\frac{2n^2}{9s}\right)\right\} \le d\le \frac{s}{16}, \qquad \alpha=\frac{n}{18sL}.
\end{equation*}
Then to obtain $\min_{1\le k\le N}\mathbb E\!\left[\|\nabla f(x_{k-1})\|^2\right] \le \epsilon^2$, we require
\begin{equation}\label{eqn:sparse_iter}
    N\ge \frac{36Ls}{d\,\epsilon^2}\bigl(f(x_0)-f^\ast\bigr),
\end{equation}
where $f^\ast$ is the optimal solution.
\end{thm}

\subsection{Finite Sum Structure with Minibatch Gradient Descent}
Many machine learning tasks are formulated within the framework of {\it empirical risk minimisation} (ERM) \cite{vapnik1991erm, vapnik1998statistic}. Suppose data $(x,y)$ are drawn from some distribution $\mathcal{D}$ and the aim is to find parameters $\theta$ minimising the population risk
\[
    R(\theta) = \mathbb{E}_{(x,y) \sim \mathcal{D}} \left[\ell(\theta; x,y)\right],
\]
where $\ell(\theta;x_i, y_i)$ is the loss incurred by parameters $\theta$ on the single example $(x, y)$ - e.g. squared loss $\ell = \frac{1}{2} (f_\theta(x) - y)^2$ for regression, or cross-entropy loss for classification. Since $\mathcal{D}$ is unknown, $R(\theta)$ cannot be computed or optimised directly. Instead, in practice, given a dataset of M i.i.d. samples $\{(x_i, y_i)\}_{i=1}^M \sim \mathcal{D}$, the standard approach is to minimise the {\it empirical risk}
\[
    f(\theta) = \frac{1}{M} \sum_{i=1}^M \ell(\theta;x_i,y_i),
\]
which serves as an unbiased estimator for the loss $R(\theta)$. This substitution of an intractable expectation by a finite sum over observed data is precisely what gives $f$ its {\it finite-sum structure}, and it is this structure that the mini-batch exploits \cite{bottou2018optimization}.

Because $f$ is the sum of per-example losses, the gradient is also an average:
\[
    \nabla f(\theta) = \frac{1}{M} \sum_{i=1}^M \nabla_\theta \ell(\theta;x_i, y_i).
\]
This linearity is what enables mini-batching possible: the gradient of an average is the average of the gradients. Full-batch gradient descent requires computing gradients over the entire dataset at every iteration, making each update prohibitively expensive when the dataset contains millions of training examples, as is common in modern machine learning applications \cite{bottou2018optimization,goodfellow2016deep}. This computational burden is further compounded by the high dimensionality of contemporary neural networks, which may contain billions of trainable parameters in large language models \cite{chen2025memoryefficientrso,liang2024onlinesubspace}. Consequently, practical learning algorithms typically employ stochastic or mini-batch gradients, since gradients computed from small subsets of the data often provide sufficiently informative descent directions while substantially reducing the computational cost per iteration \cite{bottou2018optimization}.

Concretely, instead of using all $M$ points, mini-batch gradient descent uniformly samples $m$ points from $[M]$ (we will consider sampling with replacement, but similar results hold for sampling without replacement), and computes the gradient estimator
\[
    g_B(\theta) = \frac{1}{m} \sum_{i \in B} \nabla_{\theta} \ell(\theta; x_i, y_i).
\]
The key property of this gradient estimator is that it is unbiased:
\[
    \mathbb{E}_B \left[ g_B(\theta) \right] 
    = \mathbb{E}_B \left[\frac{1}{m} \sum_{i \in B} \nabla_\theta \ell(\theta;x_i, y_i)  \right] 
    = \nabla f(\theta),
\]
which is what allows the method to still converge (in expectation) despite each step being "noisy" - the noise averages out over iterations. In this section, we show that our algorithm can also be applied in the mini-batch setting.

\subsubsection{Structure of The Problem and Motivations}
We will consider functions of the form 
\begin{equation}\label{eqn:mb}
    f(x) = \frac{1}{M}\sum_{i=1}^M f_i(x),
\end{equation}
where the size $M$ may potentially be larger than the dimension of the problem $n$. Here, we will have an additional assumption on the variance of the mini-batch gradient, which is a common assumption when dealing with mini-batch gradient descent \cite{bottou2018optimization}.
\begin{assump}\label{assump:variance-minibatch}
    (Variance-bound on mini-batch gradient): For a batch $B_m$ iid sampled from $[M]$ with $|B_m| = m$, we have $\mathbb{E} \left[ \left\| \frac{1}{m} \sum_{i \in B_m} \nabla f_i(x) - \nabla f(x) \right\|^2 \right] \leq \sigma_m^2$.
\end{assump}
If we assume the that this holds for any $m = 1$ with $\sigma_1^2$, then through simple algebra, we can set $\sigma^2_m = \sigma_1^2/m$. This means that the variance has an order of $m^{-1}$, where $m$ is the batch size.

\paragraph{Refresh Schedule.}
In the convergence of the standard SGD, the existence of the variance term requires the use of a vanishing step-size to control the impact of this factor. Similarly, our algorithm will employ a decaying step size to control the variance, which at the same time allows us to use an increasing number of steps before refresh. To this end, we consider the refresh schedule $r_{i+1} - r_i = i^\beta$, where $r_i$ is the iteration where the the alignment vector is updated for the $i^{th}$ time, and $\beta > 0$ is a parameter that we can choose depending on the problem. There can be other refresh schedule which could be considered, but our analysis will mainly cover this example and the efficacy of other schedules could be verified in a similar manner. More details on the analysis of the refresh schedule is in Section \ref{sec:mb_general}.

\paragraph{Mini-batch Variant Pseudocode.}
Here, we present the mini-batch variant of the proposed algorithm.
\begin{algorithm}
	\caption{Subspace SGD with Persistence of Memory (SSGDPM)}\label{Alg: SSD_minibatch}
	\begin{algorithmic}[1]
		\State{\textbf{Inputs:} $d, \delta, m, \beta$ \Comment{subspace rank, alignment threshold, batch size, refresh parameter}}
		\State{\textbf{Initialize:} $x_0$ \Comment{arbitrary initialisation}}
        \State $i \gets 0$
        \State $r_i \gets 0$
		\For {k = 1, 2, \ldots}
        \If{$i = 0$ or $k - r_i > i^\beta$}
            \State Generate a new $v_{k-1}$
            \State $i \gets i + 1$
            \State $r_i \gets k$
        \Else
            \State $v_{k-1} \gets v_{k-2}$
        \EndIf
		\State Generate $\tilde{P}_{k-1}$ orthogonal to $v_{k-1}$
        \State $\hat{v}_{k-1} \gets v_{k-1} / \|v_{k-1}\|$
        \State $P_{k-1} \gets \begin{pmatrix}
            \hat{v}_{k-1} & \tilde{P}_{k-1}
        \end{pmatrix}$
        \State Sample $B_{k-1}$ iid from $[M]$ with $|B_{k-1}| = m$
		\State $x_{k} \gets x_{k-1} - \alpha_{k-1} \frac{1}{m}\sum_{i \in B_{k-1}}P_{k-1} P_{k-1}^\top \nabla f_i(x_{k-1})$
		\EndFor
		\end{algorithmic}
\end{algorithm}

\subsubsection{Analysis for Convergence Iteration} \label{sec:mb_general}

Here, we show that under bounded variance assumption, the convergence rate of the algorithm can be arbitrarily close to the standard SGD rate. Before that, we will take a look at 2 interesting properties of under this setup. Firstly, for a desired initial alignment $\gamma_0$, suppose we consider the same oracle for the alignment vector, a batch size $m'$ satisfying
\begin{equation}
    m' \geq \frac{\gamma_0 \sigma_1^2}{(1 - \gamma_0)\| \nabla f(x) \|^2}
\end{equation}
is required. Secondly, the alignment between consecutive gradients will incur an additional variance error, which is described below.
\begin{prop}\label{prop:minibatch-refresh}
    Assuming that $f(x) = \frac{1}{M} \sum_{i=1}^M f_i(x)$ where each $f_i$ is $L$-smooth, then Algorithm \ref{Alg: SSD_minibatch} with $\beta \in (0, 1)$ has that $\mathbb{E} \left[ \langle v_k, \nabla f(x_k) \rangle^2 \right] \geq \delta \mathbb{E} \left[ \|\nabla f(x_k) \|^2 \right] - \eta_k \sigma_m^2$ for all $k \geq 1$ under the step-size schedule of $\alpha_k = \frac{1}{L\sqrt{k+1}}$ and $\eta_k = \sqrt{k} - \sqrt{r_i}$, where $r_i$ is the iteration of the $i^{th}$ update to the alignment vector.
\end{prop}

With these in place, we have the convergence rate of the mini-batch variant of the proposed algorithm under minimal assumptions.
\begin{thm}\label{thm:iter_mb_general}
    Suppose the objective function is \eqref{eqn:mb} and Assumption \eqref{assump:variance-minibatch} holds. Suppose $r_{i+1} - r_i = i^\beta$ for some $\beta \in (0,1)$, and let $k_0 > 0$ satisfy $k_0 = r_{i^\ast}$ with $i^\ast \geq \left(\frac{5}{\log(\gamma_0/2\delta)}\right)^{2/(1-\beta)}$. Let $\tau$ denote the index of the minimally obtained gradient, i.e. $\tau = \arg \min_{t \in [N]} \mathbb{E}[\| \nabla f(x_t)\|^2]$. Then to achieve an $\epsilon$ gradient, i.e. $\mathbb{E} \left[\| \nabla f(x_\tau) \|^2 \right] < \epsilon^2$, we will require
    \begin{equation}
        \epsilon^{-2} \left( 1- \left(\frac{k_0}{N}\right)^{\frac{\beta}{1 + \beta}} \right) + \frac{\sqrt{k_0}}{N^{\frac{\beta}{1+\beta}}} \lesssim N^{\frac{1-\beta}{2(1+\beta)}}
    \end{equation}
    iterations. Equivalently, we have $N = O \left(\epsilon^{-4 \left( \frac{1 + \beta}{1 - \beta} \right)} \right)$.
\end{thm}

\begin{rem}
    Although this analysis assumes that the refresh schedule follows a specific structure, in general, this can also be fixed like in the vanilla algorithm, or even have another formulation by itself. The analysis for this would be similar to the schedule we are assuming here, which is discussed in Appendix \ref{sec:mb_convergence}.
\end{rem}

\section{Proof Ideas}\label{sec:proof_ideas}
In this section, we will briefly describe the techniques and tools required to show the results in Section \ref{sec:results} and \ref{sec:applications}. For a more detailed analysis, we refer the reader to the appendix.

\subsection{General Algorithm}\label{sec:idea_general}
In the simple variant of our algorithm, our aim is to obtain an $\Theta(1)$ dependence on the dimension of the projection, as compared the SSD algorithm. From the $L-$smoothness property of the function, we get
\begin{equation}\label{eqn:convergence_initial_eqn}
    f(x_k) - f(x_{k-1}) \leq -\alpha \norm*{P_{k-1}^\top \nabla f(x_{k-1})}^2 + \frac{\alpha^2 L}{2} \norm*{P_{k-1} P_{k-1}^\top \nabla f(x_{k-1})}^2.
\end{equation}
In expectation, the projected gradient is lower bounded by $((1 - d/(n-1))\gamma_{k-1} + d/(n-1)) \norm*{\nabla f(x_{k-1})}^2$, by the construction of our algorithm. The rest of the proof would then follow by adding steps $1$ to $N$ and rearranging the terms. We note that the SSD algorithm does not have the $\gamma_{k-1} > \delta$ term in this expression, which would result in an overall dependence of $n/d$ in the final expression for the iteration complexity.

In the next part, we show that given an initial alignment of $\gamma_0$, the alignment vector remains {\it weakly correlated} with the next gradient. This is possible as gradient updates are local and we ask how quickly the correlation with the moving gradient decays. We first rewrite
\begin{align*}
    \inner*{\nabla f(x_k), \hat v_k} 
        &= \inner*{\vkkh, \nabla f(x_{k-1})} + \inner*{\vkkh, \nabla f(x_k) - \nabla f(x_{k-1})},
\end{align*}
with $v_k = v_{k-1}$. Squaring this and using Cauchy-Schwartz Inequality on the second term yields
\[
    \inner*{\vkh, \nabla f(x_k)}^2 \geq \inner*{\vkkh, \nabla f(x_{k-1})}^2 - 2\norm*{\nabla f(x_k) - \nabla f(x_{k-1})} \abs*{\inner*{\vkkh, \nabla f(x_{k-1})}}.
\]
The first term follows by the induction hypothesis. The second term is controlled using the $L-$smoothness of the property of the gradient, which means we have to bound the term $\norm*{P_{k-1} P_{k-1}^\top \gfk}$. This is done through careful analysis of the projection $P_{k-1} P_{k-1}^\top$ with conditional arguments. To translate back to the correct norm (i.e. $\norm*{\gfk}$ instead of $\norm*{\gfkk}$, we need an upper bound for $\norm*{\nabla f(x_k)}$ in terms of the norm of the previous gradient, givingus a formula for $\gamma_k$ in terms of $\gamma_{k-1}$. This recurrence relation then allows us to find the maximum number of steps before the alignment drops below $\delta$. Because gradient updates are local in terms of the step size, we are expected to have the maximum number of steps before a refresh is required to also depend on the step size, as shown in Corollary \ref{cor:ref_bound}.

The last part of this section is dedicated to finding a tight upper bound for the sum of $\gamma_k$. Although we can simply lower bound the individual alignments by $\delta$ to get $r\delta$ for $r$ steps of re-using the same alignment vector, doing so will not give us a sharp bound in terms of the dependence on $\delta$. We explicitly study the recurrence relation of $\gamma_k$ to obtain a sharper lower bound on the sum. Instead of $N\times\delta$, we see that it is $N \times \left(\frac{d}{n-1} + \frac{4\alpha L}{\log(\gamma_0/\delta)}\gamma_0\right)$.

All bounds in the proof are derived by first conditioning on the previous iteration, to get a descent inequality for the current step, before taking full expectations.

The computational complexity can then be obtained from the iteration complexity by noting that the cost per iteration during refresh is $nd + \nu$, while the cost per iteration during the non-refresh steps is $nd + d\xi$.

\subsection{Mini-batch Setting}
Here, we will discuss the general idea to prove convergence for the mini-batch variant of the algorithm, and refer the readers to the more detailed analysis in the respective Appendix \ref{sec:mb_convergence}. The general steps to take are rather consistent with the general algorithm, as described in Section \ref{sec:idea_general}. The only difference here is that the descent direction is $P_{k-1} P_{k-1} \nabla g_{k-1}$ instead of $P_{k-1} P_{k-1} \gfkk$. This means that we require to control the term $\norm*{P_{k-1} P_{k-1}^\top \nabla g_{k-1}}$ in terms of $c\norm*{\gfkk}$ and the noise.

Through direct computation, we will obtain that
\[
    \mathbb{E} \left[ \inner*{\vkh, \gfk}^2 \right]
        \geq \gamma_k \mathbb{E} \left[ \norm*{\gfk}^2 \right] - \eta_k \sigma_m^2,
\]
where the variance term $\sigma_m$ will be added directly to the descent equation \eqref{eqn:convergence_initial_eqn}. Then, all that remains is to analyse the recurrence relation of $\gamma_k$ like before. This time however, because of the decaying step size, $\gamma_k$ decays at an increasingly slower rate as $k$ increases, meaning we are allowed to use the same alignment vector for even more steps, motivating the use of $r_{i+1} - r_i = i^\beta$ for some $\beta \in (0,1)$. The recurrence is then solved with this assumption to find that it suffices to have
\[
    i^{(\beta-1)/2} \lesssim \frac{\log(\gamma_0/2\delta)}{5},
\]
for every refresh step $r_i$.

Finally, equation \eqref{eqn:convergence_initial_eqn} is summed from $k=1$ to $N$, with careful analysis of the variance contributions $\eta_k \sigma_m^2$ and also the number of updates to the alignment vector needed in $N$ steps, i.e. $s$ such that $r_s = N$, which turns out to be $O(N^{1/(1+\beta)})$.

\subsection{Local Regime}
The proof establishes that the gradient iterates $g_k = \nabla f(x_k)$ aligns progressively with the leading eigenvector $u_1$ of the Hessian $H = \nabla^2 f(x^\ast)$.  We measure alignment through the quantity
\begin{equation*}\label{eq:tk-def}
  t_k \;:=\; \frac{\norm{\Proj_{u_1^\perp} g_k}}{\norm{\Proj_{u_1} g_k}},
\end{equation*}
which is the tangent of the angle between $g_k$ and $u_1$. Showing $t_k \to 0$ is equivalent to showing that $g_k$ converges in direction to $u_1$. The argument proceeds in two phases.

\paragraph{Intuition and Key Recursion.}
The starting point is a gradient recursion obtained by Taylor-expanding the gradient of $f$ around $x^\ast$ and substituting the projected-gradient update \eqref{eqn:update_step}. A short calculation gives
\begin{equation}\label{eq:overview-rec}
  g_{k+1} \;=\; M g_k \;+\; \alpha H E_k \;+\; r_k,
  \qquad M \;=\; I - \alpha H,
\end{equation}
where $E_k = (I-P_kP_k^\top)g_k$ is the error due to the random projection of the algorithm and $r_k = r(e_{k+1}) - r(e_k)$ collects the Taylor remainders from the Hessian Lipschitz condition.  Because $\norm{r_k} = O(\norm{g_k}^2)$, the remainder is second order in the gradient norm and becomes negligible as the iterates approach $x^\ast$. The directional behaviour of $g_k$
is therefore governed by the first two terms of \eqref{eq:overview-rec}.

\paragraph{Phase 1: Gradient Alignment via a Perturbed Power Iteration.}
Observe that in the absence of projection errors ($E_k = 0$) and Taylor error terms, the recursion \eqref{eq:overview-rec} reduces to $g_{k+1} = Mg_k$. By assumption, $M = I - \alpha H$ has eigenvalues $\mu_i = 1-\alpha\lambda_i$ satisfying $\mu_1 > \mu_2 > \cdots > \mu_n > 0$, and iterating $g_{k+1} = Mg_k$ is precisely a power iteration on $M$. The component of $g_k$ along $u_1$ stays relatively constant, while every orthogonal component decays at a strictly smaller rate $\frac{\mu_i}{\mu_1}$. Consequently $t_k$ contracts at a rate $\frac{\mu_2}{\mu_1} < 1$, and $g_k$ aligns to $u_1$ geometrically fast.

In the presence of the projection error $\alpha H E_k$, one has to bound the numerator and denominator of $t_{k+1}$ separately. We utilise the alignment assumption $\langle \hat v_k, g_k\rangle^2 \ge \delta\norm{g_k}^2$ to show that this error is of order $\sqrt{1-\delta}$. Specifically, we have that
\begin{itemize}
  \item \emph{Upper bound on the numerator.}  Because $M$ and
        $\Proj_{u_1^\perp}$ share the same eigenbasis, they commute and
        \[
          \norm{\Proj_{u_1^\perp} M g_k} \;\leq\; \mu_2\,\norm{\Proj_{u_1^\perp} g_k}.
        \]
        The projection-error term $\alpha\norm{\Proj_{u_1^\perp} H E_k}$ contributes an additive noise of order $\sqrt{1-\delta}\,\norm{g_k\vphantom{k^\perp}}$, controlled by the alignment parameter $\delta$. We can bound the contribution from $H$ by $\lambda_n$ which is constant in terms of the free parameter $\delta$.
  \item \emph{Lower bound on the denominator.} The leading contribution is $\mu_1\norm{\Proj_{u_1} g_k}$, reduced by a perturbation of the same order $\sqrt{1-\delta}\,\norm{g_k\vphantom{k^\perp}}$. The noise in the denominator is similar controlled as in the numerator case by utilising the alignment property we assumed.
\end{itemize}
The residual terms are controlled by using a pre-established exponential decay estimate for $\norm*{g_k}$, whose own probability guarantee will show in the final bound. These two bounds can then be combined and expressing $\norm{g_k}$ in terms of $t_k$ via $\norm{g_k} = \norm{\Proj_{u_1}g_k}\sqrt{1+t_k^2}$ yields a scalar recursion of the form
\begin{equation*}
  t_{k+1} \;\lesssim\; \frac{\tilde \mu^{(k)}_2}{\mu_1}\,t_k \;+\; C(\delta),
\end{equation*}
where $C(\delta) \to 0$ as $\delta \to 1$ and $\tilde \mu^{(k)}_2$ is a perturbation of $\mu_2$ which also depends on $\delta$. We then find the condition on $\delta$ such that we ensure that the leading factor is $<1$ so that the above is a \emph{perturbed contraction}. The standard fixed-point iteration argument shows that $t_k$ converges to a neighbourhood of zero with radius $O\bigl(C(\delta)/(1-\max_k \tilde \mu_2^{(k)}/\mu_1)\bigr)$. By choosing $\delta$ sufficiently close to $1$, in other words, by requiring the adaptive direction $\hat v_k$ to be sufficiently well aligned with $g_k$ one can make this neighbourhood arbitrarily small. In particular, there exists a finite iterate $k_1 < K_\varepsilon$ at which $t_{k_1} = O(1-\delta)$, or equivalently,
\begin{equation*}
  \langle u_1, g_{k_1}\rangle^2 \;\ge\; \delta_0^2 \norm*{g_{k_1}}^2,
\end{equation*}
for a parameter $\delta_0$ that can be made arbitrarily small by increasing $\delta$.  All concentration events required for this phase hold simultaneously with probability at least $1 - k_0p_{decay}(\hat \tau, \delta, d)$ by a union bound, which we can improve to $1 - e^{-\frac{1}{8}k_1(1 - p_{decay}(\hat \tau, \delta, d))}$ by an application of Bernstein Inequality. Next, we set $g_{k_1} = v_{k_1}$ and for phase 2, have that for $k \geq k_1$,
\[
    \langle u_1, \hat v_{k}\rangle^2 \;\ge\; \delta_0^2.
\]

\paragraph{Phase 2: Frozen Direction and Continued Alignment.}
Having shown that $g_{k_1}$ (and hence the alignment vector $v_{k_1}:= g_{k_1}$ used from this point on) is $\delta_0$-aligned with $u_1$, we now ask whether the alignment persists once $v_k$ is frozen at $v_{k_1}$ for $k > k_1$, rather than re-estimated at every few steps. This is the regime relevant to the "no-refresh" claim: the guiding direction is fixed, so any further error induced by the algorithm comes from the fresh random matrix $\tilde P_k$ alone.

Unlike Phase~1, there is no standing alignment hypothesis to fall back on here, so the numerator and denominator bounds on the ratio $t_i^{(k)} := |u_i^\top g_k|/|u_1^\top g_k|$ (for each $i \geq 2$) must be derived directly from a Johnson-Linderstrauss-type concentration bound. With probability $1 - 2e^{-cd\tau_i^2}$, the random projection $\tilde P_k$ approximately preserves the inner products $\inner*{\tilde P_k^\top u_i, \tilde P_k^\top g_k}$ up to a slack $\tau_i$ relative to their projections onto $v_{k_1}^\perp$. More importantly, the same result holds for the re-scaled random matrix $\hat P_k$ which we use instead in this phase. Feeding this into the same expansion in \eqref{eq:overview-rec} yields, for every $i \geq 2$, a scalar recursion of the same shape as in Phase~1,
\[
    t_i^{k+1} \lesssim \frac{\tilde \mu_i^{(k)}}{\mu_1} t_i^{(k)} + \tilde \epsilon_i^{(k)},
\]
where now the slack term $\tilde \epsilon_i^{(k)}$ is driven by the chosen concentration parameter $\tau_i$ (in place of $\sqrt{1-\delta}$ as in Phase~1) together with the same negligible exponentially-decaying remainder. As before, requiring $\tilde \mu_i^{(k)} < \mu_1$ gives a contraction condition, which translates into an explicit upper bound on $\tau_i$, of order $(\lambda_i - \lambda_1)/\lambda_i \,\cdot\, \sqrt{1-\delta_0}$. Eigen-gaps that are small relative to $\lambda_1$ force tighter concentration, i.e. a higher-probability (larger $d$) requirement.

One-subtlety remains, because of the re-scaling of the random matrix $\tilde P_k \rightarrow \hat P_k$, the effective projection operator onto the random subspace $\hat P_k \hat P_k^\top$ has a slightly worse operator-norm, of order $(n-1)/d$ instead of 1. This inflates the remainder term by the same factor, and mildly perturbs the step-size required to ensure that we still have a exponential rate of decay, but does not change the qualitative picture that the dominant contractions are still goverened by the spectral gaps $\lambda_i - \lambda_1$. One only needs $k_1 = O(\log(n/d))$ steps of "burn-in" before this extra factor becomes immaterial.

Running the same perturbed-fixed-point argument componentwise, each $t_i^{(k)}$ is shown to be non-increasing (given the desired condition on $\tau_i$), so the aggregate quantity $\sum_{i\geq2} (t_i^{(k)})^2$ or equivalently $\norm{\Proj_{u_1^\perp} g_k}/|u_1^\top g_k|$ stays controlled for the enture horizon $k \in (k_1, K_\epsilon)$. 

Finally, combining the events across all $i \in [n]$ and all $k \in (k_1, K_\epsilon)$ via a union bound (using a successive conditioning argument for each consecutive event, together with the Phase~1 probability for reaching the initial alignment $\delta_0$), we obtain the stated overall success probability in \eqref{eqn:local_regime_prob}. Translating the resulting bound on $\tau_i$ back into a bound on $\norm*{\Proj_{u_1^\perp}g_k}/|u_1^\top g_k|$ gives the relationship between the target terminal correlation floor $\delta_1$ and the required initial threshold $\delta_0$. A tighter final alignment guarantee requires pushing $\delta_0$ closer to 1, i.e. a longer or more tightly concentrated Phase~1.
\section{Numerical Experiments}\label{sec:experiments}

In this section, we demonstrate our algorithms numerically for both the sparse and mini-batch variant. These experiments were run on an HPC cluster with $12$ CPU cores and $96$GB ram, with 2,000 concurrent directional derivatives computed at each iteration. For a Haar-distributed matrix with the lower dimension $d \ll n$ (such as $d = \log(n)$), it is orthogonal to any fixed vector with high probability, which allowed us to forgo the need to orthogonalise the matrix to lie in $v_k^\perp$ in the experiments. Each experiment is run with 5 seeds, with the min-max error bars plotted.

\paragraph{Sparse Setting.}
To illustrate the theoretical improvements of the computational complexity under the sparse setting, we conducted numerical experiments on a Rosenbrock function. A general Rosenbrock function has the form
\begin{equation}
    f(x) = \sum_{i=1}^M a_i(x_i - b_i)^2 + c_i(x_{i+1} - d_ix_i^2)^2,
\end{equation}
where $a_i, b_i, c_i, d_i$ are constants which define the structure of the function, and the objective is to minimise this function. To sparsify this function, we simply add a prefactor $m_i \in \{0, 1\}$ in front of each summand. In this experiment, we choose $M = 50,000$ with a sparsity factor of $250$. i.e. only $125$ $m_i's$ are $1$ (which are guaranteed to be non-consecutive). 

We run the standard gradient descent algorithm, Kozak's original SSD algorithm, and the 2 variants of our proposed algorithm on this test function. The step sizes for all algorithms are taken to be the same at $0.01$. For our algorithm where the alignment vector is estimated using IHT, we choose to run the algorithm with an upper bound on the sparsity, $s \leq s^\ast$. Using this $s^\ast$, we then select $k' = s^\ast \log(n)$ as an estimate on the number of rows of the random matrix to draw to compute $x$ in Equation \eqref{eqn:iht_init} and for the subsequent IHT steps. We test 2 values of $s^\ast$ satisfying $p := k'/n = 0.1, 0.2$.

For the value of $d$ in both Kozak's SSD algorithm and our algorithm, we opted to use $d=10$. Numerically, the convergence rates did not show much difference when $d$ varied within a small range around $10$, hence we only showed it with $d=10$. All results are plotting with $f(x)$ against time $s$.

Our proposed algorithm used 10GB of memory for the one with $p = 0.2$ and 5.4GB for $p = 0.1$, GD used 56.4GB and SSD used 0.6GB.

\paragraph{Minibatch Setting.}
For the minibatch setting, we consider a regularised multi-layer neural network (MLP) on the mnist dataset. We consider a function of the form
\begin{equation}
    f(x) = \frac{1}{N} \sum_{i=1}^N \ell(f_\theta(x_i) - y_i) + \lambda \norm*{\theta}_2^2,
\end{equation}
where the cross-entropy loss is used as the loss function, and the regularising coefficient is taken to be $0.005$. The architecture we choose is a 3-layer MLP, where the 2 hidden layers are dimensions $64$ each. The output layer is a 10 dimension vector which is then given into the loss function through the cross-entropy loss. The weights of the network are initialised with the He initialisation. 

We benchmark our algorithm against the vanilla SGD and Kozak's SSD, where we run a variant of Kozak's SSD by projecting the minibatch gradient instead of the full gradient at each iterate. The batch sizes $m$ in both the kozak's algorithm and SGD were taken to be $128$, while this same batch size was used to update the new alignment vector in our proposed algorithm, with the consequent iterations done with batch sizes $16$. To keep the noise from the gradient estimator small, the batch $B_{k}$ was taken to be $B_k \subset B_{r_i}'$ for $r_i \leq k < r_{i+1}$, and the step size follows $\alpha_{k} = \frac{\alpha_0}{\sqrt{r_i}}$. The step sizes for the minibatch variant of Kozak's method and SGD uses $\alpha_k = \frac{\alpha_0}{\sqrt{k}}$. Since Kozak's algorithm has a much higher iteration complexity, the step size decays to 0 much faster than the other algorithms, hence we decided to work with $\alpha_0 = 10$ for the minibatch variant of Kozak's algorithm while $\alpha_0 = 0.1$ for our proposed algorithm and SGD. All random subspaces are taken with $d = 10$. 

Our algorithm and SGD utilised 72GB of memory (because the refresh size used in our algorithm is the same as the batch size used in SGD), while the minibatch variant of Kozak's algorithm used 4GB. Although our algorithm used much more memory as compared to Kozak's algorithm, the number of iterations where such a capacity is needed is much less than the total iterations; it is only needed during the steps where the alignment vector is updated.

The results of both experiments can be observed in the figures below.

\begin{figure}[htbp]
    \centering

    \begin{subfigure}{0.48\textwidth}
        \centering
        \includegraphics[width=\linewidth]{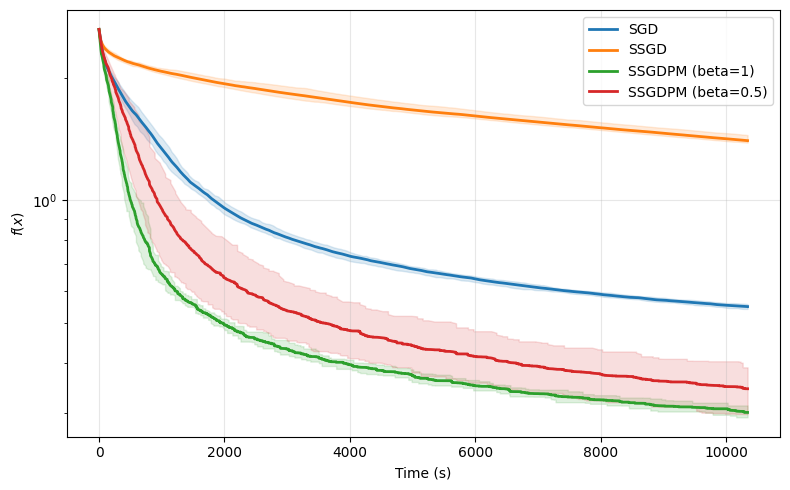}
        \label{fig:mlp}
    \end{subfigure}
    \hfill
    \begin{subfigure}{0.48\textwidth}
        \centering
        \includegraphics[width=\linewidth]{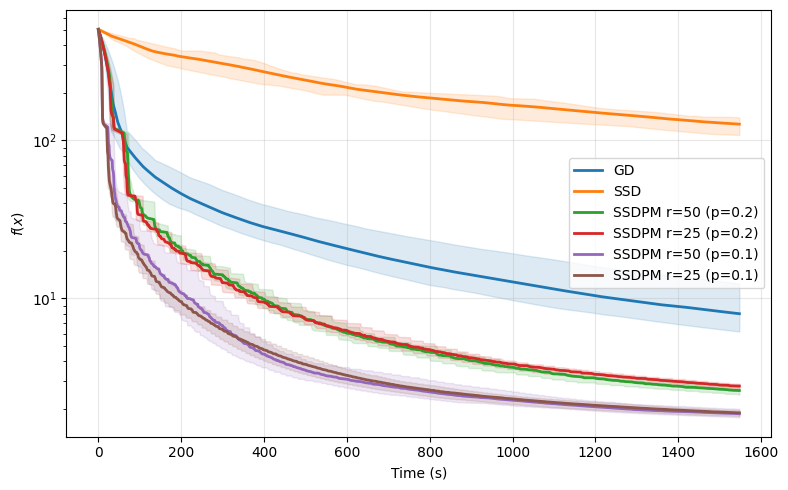}
        \label{fig:rosenbrock}
    \end{subfigure}

    \caption{The left figure is the minibatch setting with mnist dataset, while the right figure is the sparse setting with the Rosenbrock function.}
    \label{fig:combined}
\end{figure}
\section{Acknowledgements}
The authors would like to thank Alexandre d'Aspremont for illuminating discussions. AT was supported by the JSPS Grant-in-Aid for Scientific Research(B) JP23K28041. SG was supported in part by the NUS Dean’s Chair Associate Professorship E-146-00-0037-01 and the Singapore MOE grants A-8002014-00-00 and A-8003802-00-00.

\newpage
\bibliographystyle{abbrv}
\bibliography{refs}

\newpage 

\appendix

\section{Mathematical Tools Needed}
We study several properties of the random matrix $\tilde P_k$ which will be of use in the later proofs.
\begin{prop}[\cite{vershynin2018highdimensional} Lemma 5.3.2]
    Let $Q$ be a projection from $\mathbb{R}^n$ to a random $d$-dimensional subspace uniformly distributed in $G_{n, d}$. For a fixed vector $z \in \mathbb{R}^n$, we have
    \begin{align}
        \mathbb{E} \left[\norm*{Qz}^2\right] 
            &= \frac{d}{n} \norm{z}^2, \\
        (1-\tau) \frac{d}{n} \norm{z}^2 
            \leq \norm*{Qz}^2 
            &\leq (1 + \tau) \frac{d}{n} \norm{z}^2
    \end{align}
    with probability at least $1 - 2\exp{(-c\tau^2 d)}$. 
\end{prop}

\begin{lem}\label{lem:Pk_highprob}
    Let $u, v \in \mathbb{R}^n$ be fixed vectors and let $P \in \mathbb{R}^{n \times d}$ be Haar distributed on the subspace orthogonal to $u_1$, i.e., $P$ has orthonormal columns chosen uniformly from $u_1^\perp$. Then
    \begin{equation}
        \mathbb{P} \Big[ \big|\inner{P^\top u, P^\top v} - \frac{d}{n-1} \inner{u_\perp, v_\perp} \big| \leq \tau \frac{d}{n-1} \norm{u_\perp}\norm{v_\perp} \Big] 
        \geq 1- 2\exp(-c\tau^2 d),
    \end{equation}
    where $u_\perp = (I - u_1u_1^\top)u$ and $v_\perp = (I - u_1u_1^\top)v$ are the components in $u_1^\perp$.
\end{lem}

\begin{proof}
    Let $\Pi = I - u_1 u_1^\top$ denote the projection onto $u_1^\perp$. Then $P$ is Haar distributed on $\Pi$, i.e., its columns form an orthonormal basis of a uniformly random $d$-dimensional subspace of $\Pi$.  

    Define $u_\perp = \Pi u$ and $v_\perp = \Pi v$. Then
    \[
        \inner{P^\top u, P^\top v} = \inner{P^\top u_\perp, P^\top v_\perp}.
    \]

    Now, $P$ restricted to $u_1^\perp$ is equivalent to saying $PP^\top$ is a uniform random projection in $\mathbb{R}^{n-1}$. That is, if we identify $u_1^\perp \simeq \mathbb{R}^{n-1}$, then $P P^\top$ is uniform in $G_{n-1, d}$. Applying Lemma 5.3.2 in $\mathbb{R}^{n-1}$ and noting that $\norm{PP^\top u} = \norm{P^\top u}$, we find that
    \[
        \mathbb{E}[\norm{P^\top u_\perp}^2] = \frac{d}{n-1} \norm{u_\perp}^2, \quad
        \mathbb{E}[\norm{P^\top v_\perp}^2] = \frac{d}{n-1} \norm{v_\perp}^2,
    \]
    and the standard concentration inequality yields
    \[
        \norm{P^\top u_\perp}^2 \approx (1 \pm \tau) \frac{d}{n-1} \norm{u_\perp}^2 \quad \text{with probability at least } 1-2\exp(-c\tau^2 d),
    \]
    and similarly for $v_\perp$. Next, using the polarization identity
    \[
        \inner{P^\top u_\perp, P^\top v_\perp} = \frac{1}{4} \Big( \norm{P^\top (u_\perp + v_\perp)}^2 - \norm{P^\top (u_\perp - v_\perp)}^2 \Big),
    \]
    we obtain
    \begin{align*}
        \inner{P^\top \tilde u_\perp, P^\top \tilde v_\perp} 
            &\leq \frac{1}{4} \frac{d}{n-1}\Big( (1+\tau)\norm{\tilde u_\perp + \tilde v_\perp}^2 - (1-\tau)\norm{\tilde u_\perp - \tilde v_\perp}^2 \Big) \\
            &= \frac{1}{4} \frac{d}{n-1}\Big( \norm{\tilde u_\perp + \tilde v_\perp}^2 - \norm{\tilde u_\perp - \tilde v_\perp}^2 + \tau (\norm{\tilde u_\perp + \tilde v_\perp}^2 + \norm{\tilde u_\perp - \tilde v_\perp^2})\Big) \\
            &= \frac{d}{n-1} \inner{\tilde u_\perp, \tilde v_\perp} + \frac{\tau}{2} \frac{d}{n-1} (\norm{\tilde u_\perp}^2 + \norm{\tilde v_\perp}^2) \\
            &= \frac{d}{n-1} \inner{\tilde u_\perp, \tilde v_\perp} + \tau \frac{d}{n-1},
    \end{align*}
    where $\tilde u = u/\norm{u}$. Consequently, this implies that
    \begin{equation}
        \inner{P^\top u_\perp, P^\top v_\perp} 
        \leq \frac{d}{n-1} \left( \inner{ u_\perp, v_\perp} + \tau \norm{u_\perp}\norm{v_\perp} \right)
    \end{equation}
    in the same event. The lower bound is similarly obtained, completing the proof.
\end{proof}

\section{Proof for General Variant}\label{sec:proof_general}
In this section, we will provide the details to the proof for the results in the general algorithm, namely the iteration complexity in Theorem \ref{thm:conv_iter} and the time complexity in Theorem \ref{thm:cost_vanilla}. The main bulk of the argument is done in a consecutive conditioning manner, so we will first define the $\sigma-$algebras 
\begin{equation}\label{cond:sig_alg_1}
    \mathcal{F}_k = \sigma(P_0, P_1, \cdots, P_{k-1}).
\end{equation}

\subsection{Proof for Iteration Complexity (Theorem \ref{thm:conv_iter})}
To solve for Equation \eqref{eqn:conv_iter}, we will employ a 3-stage approach.
\begin{enumerate}
    \item Obtain a lower bound on the value of $\gamma_k$ given the previous $\gamma_{k-1}$.
    \item Find an upper bound on $k$ such that $\gamma_k > \delta$ is still satisfied.
    \item Sum the values of $\gamma_k$ from $1$ to $N$ and consequently find a lower bound on $N$ which sufficiently gives us $\mathbb{E} \left[ \norm{\nabla f(x_N)}^2 \right] < \epsilon^2$.
\end{enumerate}

\paragraph{Obtain a Recurrence Relation on $\gamma_k$.}
To find the recurrence relationship between $\gamma_k$ and $\gamma_{k-1}$, we first prove the following lemma that tells us how much the norm of the gradients changes after 1 iteration. WLOG, we will assume that the refresh step is at $x_0$, which is the initial point. Since the algorithm is sequential, we can always condition on the corresponding refresh steps for the iterations after it.

\begin{lem}\label{lem:grad_consec}
    Suppose $f$ is L-smooth, then in expectation, the squared norm of the gradient of $f$ at consecutive steps can be controlled by the inequality
    \begin{equation}
        \EE \left[\norm*{\nabla f(x_{k})}^2 \right] \leq \left( 1 + \alpha L \right)^2 \EE \left[\norm*{\nabla f(x_{k-1})}^2\right].
    \end{equation}
    
    \begin{proof}
        By the Lipschitz continuity of the gradient, we have from Equation \ref{eqn:l-smooth-1}
        \begin{equation*}
            \| \nabla f(x_{k}) \| - \| \nabla f(x_{k-1}) \| \leq L \|x_k - x_{k-1}\|.
        \end{equation*}
        Taking conditional expectation over $\mathcal{F}_{k-1}$ and substituting the update step \ref{eqn:update_step}, we can get
        \begin{equation}\label{eqn:grad_consec_1}
            \mathbb{E} \left[ \norm*{\nabla f(x_{k})} \mid \mathcal{F}_{k-1} \right] \leq \norm*{\nabla f(x_{k-1})} + \alpha L \mathbb{E} \left[ \norm*{P_{k-1} P_{k-1}^\top\nabla f(x_{k-1})} \mid \mathcal{F}_{k-1} \right].
        \end{equation}
        We now require a bound on the projected gradient. Observe that by the orthogonality of the columns of the matrix $P_{k-1}$, we have
        \begin{align*}
            \|P_{k-1}P_{k-1}^\top \nabla  f(x_{k-1})\|^2
                &= \nabla  f(x_{k-1})^\top P_{k-1}P_{k-1}^\top P_{k-1}P_{k-1}^\top \nabla  f(x_{k-1}) \\
                &=\nabla  f(x_{k-1})^\top P_{k-1}P_{k-1}^\top \nabla  f(x_{k-1}) \\
                &= \|P_{k-1}^\top \nabla f(x_{k-1})\|^2.
        \end{align*}
        Taking conditional expectation over $\mathcal{F}_{k-1}$, we have
        \begin{align}
            \EE[\|P_{k-1}P_{k-1}^\top \nabla  f(x_{k-1})\|^2 \mid \mathcal{F}_{k-1}]
                &= \left\langle \nabla f(x_{k-1}), \EE \left[P_{k-1}P_{k-1}^\top \right] \nabla f(x_{k-1}) \right\rangle \nonumber\\
                &= \left\langle \nabla f(x_{k-1}), \left( \left(1-\frac{d}{n-1}\right) \hat{v}_{k-1}\hat{v}_{k-1}^\top + \frac{d}{n-1}I_n \right) \nabla f(x_{k-1}) \right\rangle \nonumber\\
                &= \left(1-\frac{d}{n-1}\right)\left\langle \hat{v}_{k-1}, \nabla f(x_{k-1}) \right\rangle^2 + \frac{d}{n-1} \|\nabla f(x_{k-1})\|^2 \label{eqn:grad_consec_2}.
        \end{align}
        By Cauchy-Schwartz and Jensen's Inequality, we have
        \begin{align*}
            \EE \left[ \norm*{P_{k-1}P_{k-1}^\top \nabla  f(x_{k-1})} \mid \mathcal{F}_{k-1} \right]
                &\leq \left( \EE \left[ \norm*{P_{k-1}P_{k-1}^\top \nabla  f(x_{k-1})}^2 \mid \mathcal{F}_{k-1} \right] \right)^{1/2} \\
                &\leq \left(\|\nabla f(x_{k-1})\|^2\right)^{1/2} \\
                &= \| \nabla f(x_{k-1})\|.
        \end{align*}
        Substituting this back into \eqref{eqn:grad_consec_1}, we can get 
        \begin{align*}
            \EE[\| \nabla f(x_{k}) \|^2 \mid \mathcal{F}_{k-1}]
                &\leq \| \nabla f(x_{k-1}) \|^2 + 2\alpha L \| \nabla f(x_{k-1}) \| \EE\left[ \|P_{k-1} P_{k-1}^\top\nabla f(x_{k-1})\| \mid \mathcal{F}_{k-1} \right] \\&\qquad+ \alpha^2 L^2 \EE\left[\|P_{k-1} P_{k-1}^\top\nabla f(x_{k-1})\|^2 \mid \mathcal{F}_{k-1} \right] \\
                &\leq \left( 1 + \alpha L \right)^2 \|\nabla f(x_{k-1}) \|^2,
        \end{align*}
        and the inequality follows by taking expectation on both sides conditional on the $\sigma-$algebra of the refresh step and using the tower property of conditional expectation.
\end{proof}
\end{lem}
This gives us a control on how much the gradient norm changes in 1 update step. The change can be controlled arbitrarily small using the step size, as these changes are local. With this, we can obtain the recurrence relationship between consecutive $\gamma$'s as shown by the next lemma.
\begin{lem}\label{lem:consec_gamma}
    Suppose $\gamma_{i}$ satisfies $\EE[\langle \nabla f(x_{i}), \hat{v}_{i} \rangle^2] \geq \gamma_{i} \EE[\|\nabla f(x_{i})\|^2]$ for $i = k-1$, then with $v_k = v_{k-1}$, the same inequality is satisfied for $i=k$ with
    \begin{equation}
    \gamma_k = \frac{\left(1 - 2 \alpha L \sqrt{\left(1-\frac{d}{n-1}\right)} \right)\gamma_{k-1} - 2 \alpha L \sqrt{\frac{d}{n-1}}}{\left( 1 + \alpha L \right)^2}.
    \end{equation}

    Note that this value can be gauaranteed to be positive with control of the step size $\alpha$, which will be explained in the next section.
    
    \begin{proof}
        By rewriting $\gfk = \nabla f(x_{k-1}) + (\gfk - \nabla f(x_{k-1}))$, we have
        \begin{align*}
            \langle \hat{v}_{k}, \nabla f(x_{k}) \rangle^2
                &= \langle \hat{v}_{k-1}, \nabla f(x_{k-1}) \rangle^2 + 2\langle \hat{v}_{k-1}, \gfk - \nabla f(x_{k-1}) \rangle \langle \hat{v}_{k-1}, \nabla f(x_{k-1}) \rangle \\
                &\qquad + \langle \hat{v}_{k-1}, \gfk - \nabla f(x_{k-1}) \rangle^2 \\
                &\geq \langle \hat{v}_{k-1}, \nabla f(x_{k-1}) \rangle^2 + 2\langle \hat{v}_{k-1}, \gfk - \nabla f(x_{k-1}) \rangle \langle \hat{v}_{k-1}, \nabla f(x_{k-1}) \rangle \\
                &\geq \langle \hat{v}_{k-1}, \nabla f(x_{k-1}) \rangle^2 - 2 \|\gfk - \nabla f(x_{k-1})\| |\langle \hat{v}_{k-1}, \nabla f(x_{k-1}) \rangle|,
        \end{align*}
        where the last line is by Cauchy-Schwartz Inequality. Using the assumption that the function is $L$-smooth, we have
        \begin{equation}\label{eqn:original_rec_lbd}
            \langle \hat{v}_{k}, \nabla f(x_{k}) \rangle^2
                \geq \langle \hat{v}_{k-1}, \nabla f(x_{k-1}) \rangle^2 - 2 L\|x_k - x_{k-1}\| |\langle \hat{v}_{k-1}, \nabla f(x_{k-1}) \rangle|
        \end{equation}
        and
        \begin{equation*}
            \EE[\langle \hat{v}_{k}, \nabla f(x_{k}) \rangle^2 \mid \mathcal{F}_{k-1}]
                \geq \langle \hat{v}_{k-1}, \nabla f(x_{k-1}) \rangle^2 - 2L |\langle \hat{v}_{k-1}, \nabla f(x_{k-1}) \rangle| \times \EE[\|x_k - x_{k-1}\| \mid \mathcal{F}_{k-1}].
        \end{equation*}
        To upper-bound the last term, we first use Cauchy-Schwartz and then Equation \eqref{eqn:grad_consec_2} to get
        \begin{align*}
            \EE[\|x_k - x_{k-1}\| \mid \mathcal{F}_{k-1}]
                &= \alpha \EE\left[\|P_{k-1}P_{k-1}^\top \nabla f(x_{k-1})\| \mid \mathcal{F}_{k-1} \right] \\
                &\leq \alpha \EE\left[\|P_{k-1}P_{k-1}^\top \nabla f(x_{k-1})\|^2 \mid \mathcal{F}_{k-1} \right]^{1/2} \\
                &\leq \alpha \left( \left(1-\frac{d}{n-1}\right)\left\langle \hat{v}_{k-1}, \nabla f(x_{k-1}) \right\rangle^2 + \frac{d}{n-1} \|\nabla f(x_{k-1})\|^2\right)^{1/2}.
        \end{align*}
        Using the inequality $\sqrt{a+b} \leq \sqrt{a} + \sqrt{b}$, we have
        \begin{equation}\label{eqn:original_rec_ubd}
             \EE[\|x_k - x_{k-1}\| \mid \mathcal{F}_{k-1}] \leq \alpha \left[ \sqrt{\left(1-\frac{d}{n-1}\right)} |\left\langle \nabla \hat{v}_{k-1}, f(x_{k-1}) \right\rangle| + \sqrt{\frac{d}{n-1}} \|\nabla f(x_{k-1})\|\right].
        \end{equation}
        Substituting this back into the equation above, we have
        \begin{align*}
            \EE[\langle \hat{v}_{k}, \nabla f(x_{k}) \rangle^2 \mid \mathcal{F}_{k-1}]
                &\geq \langle \hat{v}_{k-1}, \nabla f(x_{k-1}) \rangle^2 - 2 \alpha L |\langle \hat{v}_{k-1}, \nabla f(x_{k-1}) \rangle| \times \\
                &\qquad \left[ \sqrt{\left(1-\frac{d}{n-1}\right)} |\left\langle \hat{v}_{k-1}, \nabla f(x_{k-1}) \right\rangle| + \sqrt{\frac{d}{n-1}} \|\nabla f(x_{k-1})\|\right]  \\
                &= \left(1 - 2 \alpha L \sqrt{\left(1-\frac{d}{n-1}\right)} \right)\langle \hat{v}_{k-1}, \nabla f(x_{k-1}) \rangle^2 \\
                &\qquad - 2 \alpha L \left[\sqrt{\frac{d}{n-1}} \|\nabla f(x_{k-1})\|\right] |\langle \hat{v}_{k-1}, \nabla f(x_{k-1}) \rangle| \\
                &\geq \left(1 - 2 \alpha L \sqrt{\left(1-\frac{d}{n-1}\right)} \right)\langle \hat{v}_{k-1}, \nabla f(x_{k-1}) \rangle^2 - 2 \alpha L \sqrt{\frac{d}{n-1}} \|\nabla f(x_{k-1})\|^2.
        \end{align*}
        Taking expectation and using the induction hypothesis on $\langle \hat{v}_{k-1}, \nabla f(x_{k-1}) \rangle^2$, we have
        \begin{equation*}
            \EE[\langle \hat{v}_{k}, \nabla f(x_{k}) \rangle^2]
                \geq \left(\left(1 - 2 \alpha L \sqrt{\left(1-\frac{d}{n-1}\right)} \right)\gamma_{k-1} - 2 \alpha L \sqrt{\frac{d}{n-1}} \right) \EE[ \| \nabla f(x_{k-1}) \|^2 ].
        \end{equation*}
        Finally, we use Lemma \eqref{lem:grad_consec} to translate the gradient norm on the RHS back to $\|\nabla f(x_k)\|$ to get
        \begin{equation}
            \gamma_k = \frac{\left(1 - 2 \alpha L \sqrt{\left(1-\frac{d}{n-1}\right)} \right)\gamma_{k-1} - 2 \alpha L \sqrt{\frac{d}{n-1}}}{\left( 1 + \alpha L \right)^2}.
        \end{equation}
    \end{proof}
\end{lem}

\paragraph{Finding an Upper Bound for Number of Reuse.}
As a direct consequence, we can find the number of steps $r$ in which we the same alignment vector $v_0$ still retains $\delta$ amount of correlation with the gradient. i.e., $r = \max\{k \in \mathbb{N} \mid \gamma_k > \delta \}$, where $\gamma_k = \mathbb{E}[\langle \hat{v}_k, \nabla f(x_k) \rangle^2]/\mathbb{E}[\|\nabla f(x_k)\|^2]$. 
\begin{cor}\label{cor:ref_bound}
    Suppose $\gamma_0$ is the initial alignment parameter. Then, the number of steps $r$ in which we can reuse the same alignment vector $v_0$ is upper bounded by
    \begin{equation}
        r \leq \frac{\log\left( \frac{\gamma_0}{\delta + \sqrt{\frac{d}{n-1}}} \right)}{\log\left( \frac{(1+\alpha L)^2}{1-2\alpha L} \right)},
    \end{equation}
    where $\alpha < 1/2L$.
    \begin{proof}
        From Lemma \eqref{lem:consec_gamma}, we have the recurrence relation
        \begin{equation*}
            \gamma_k = \frac{\left(1 - 2 \alpha L \sqrt{\left(1-\frac{d}{n-1}\right)} \right)\gamma_{k-1} - 2 \alpha L \sqrt{\frac{d}{n-1}}}{\left( 1 + \alpha L \right)^2}.
        \end{equation*}
        Solving the recurrence by induction, we obtain
        \begin{align*}
            \gamma_k 
                &=\left(\frac{1-2\alpha L \sqrt{1-\frac{d}{n-1}}}{(1+\alpha L)^2}\right)^k \gamma_0 - \frac{2\alpha L \sqrt{\frac{d}{n-1}}-2\alpha L \sqrt{\frac{d}{n-1}}\left(\frac{1-2\alpha L \sqrt{1-\frac{d}{n-1}}}{(1+\alpha L)^2}\right)^k}{(1+\alpha L)^2 - \left(1 - 2\alpha L \sqrt{1-\frac{d}{n-1}}\right)} \\
                &\geq \left(\frac{1-2\alpha L \sqrt{1-\frac{d}{n-1}}}{(1+\alpha L)^2}\right)^k \gamma_0 - \frac{2\alpha L \sqrt{\frac{d}{n-1}}}{(1+\alpha L)^2 - \left(1 - 2\alpha L \sqrt{1-\frac{d}{n-1}}\right)} \\
                &\geq \left(\frac{1-2\alpha L \sqrt{1-\frac{d}{n-1}}}{(1+\alpha L)^2}\right)^k \gamma_0 - \frac{2}{2 + \alpha L}\sqrt{\frac{d}{n-1}} \\
                &\geq \left(\frac{1-2\alpha L}{(1+\alpha L)^2}\right)^k \gamma_0 - \sqrt{\frac{d}{n-1}}.
        \end{align*}
        Setting this to be at least $\delta$, we have
        \begin{equation*}
            \left(\frac{1-2\alpha L}{(1+\alpha L)^2}\right)^r \geq \frac{\delta + \sqrt{\frac{d}{n-1}}}{\gamma_0}.
        \end{equation*}
        Taking logarithm on both sides, we have
        \begin{equation}
            r \leq \frac{\log\left( \frac{\gamma_0}{\delta + \sqrt{\frac{d}{n-1}}} \right)}{\log\left( \frac{(1+\alpha L)^2}{1-2\alpha L} \right)}.
        \end{equation}
    \end{proof}
\end{cor}

\begin{rem}
    For the denominator of the form above, using $\log(1+x) \leq x$ for $x > -1$ and $-\log(1-y) \leq \frac{y}{1-y}$ for $y \in [0, 1)$, we have 
    \[
        \log\left( \frac{(1+\alpha L)^2}{1 - 2\alpha L} \right) = 2 \log(1+\alpha L) - \log(1-2\alpha L) \leq 2\alpha L + \frac{2\alpha L}{1-2\alpha L} = \frac{4 \alpha L (1-\alpha L)}{1 - 2\alpha L}.
    \]
    So, $r \leq \frac{1-2\alpha L}{4\alpha L (1-\alpha L)} \log \left( \frac{\gamma_0}{\delta + \sqrt{\frac{d}{n-1}}} \right)$ is sufficient.
\end{rem}

\paragraph{Summing $\gamma_k$ and Combining Everything.}
Now, we can finally prove the main iteration complexity result in Theorem \ref{thm:conv_iter}.
\begin{proof}
    By $L$-smoothness of $f$, we have from Equation \eqref{eqn:lip_ineq} that
    \begin{align*}
        f(x_k) - f(x_{k-1}) \leq -\alpha \|P_{k-1}^\top \nabla f(x_{k-1}) \|^2 + \frac{\alpha^2 L}{2} \|P_{k-1} P_{k-1}^\top \nabla f(x_{k-1}) \|^2,
    \end{align*}
    where we used the update step as defined in \eqref{eqn:update_step}. Taking conditional expectation using the $\sigma$-algebras defined previously in \eqref{cond:sig_alg_1}, we have
    \begin{align*}
        \EE[f(x_{k-1}) - f(x_{k}) \mid \mathcal{F}_{k-1}]
            &\geq \alpha \EE\left[\norm{P_{k-1}^\top \nabla f(x_{k-1})}^2 \mid \mathcal{F}_{k-1}\right] - \frac{\alpha^2 L}{2} \EE\left[\norm{P_{k-1} P_{k-1}^\top \nabla f(x_{k-1})}^2 \mid \mathcal{F}_{k-1}\right] \\
            &= \alpha\brac{1 - \frac{\alpha L}{2}} \EE\left[\|P_{k-1}^\top \nabla f(x_{k-1}) \|^2 \mid \mathcal{F}_{k-1}\right] \\
            &= \alpha \brac{1 - \frac{\alpha L}{2}} \sqbrac{ \brac{1 - \frac{d}{n-1}}\EE \left[ \langle \hat{v}_{k-1}, \nabla f(x_{k-1}) \rangle^2 \mid \mathcal{F}_{k-1} \right] + \frac{d}{n-1}\|\nabla f(x_{k-1})\|^2}.
    \end{align*}
    Assuming we choose an $\alpha < \frac{1}{2L}$, we have $\alpha \left( 1 - \frac{\alpha L}{2} \right) > 0$. Next, taking expectation and with Equation \eqref{eqn:alignment_parameter}, we can get
    \begin{equation*}
        \EE[f(x_{k-1}) - f(x_{k})] \geq \alpha \left( 1 - \frac{\alpha L}{2} \right)\left(\brac{1 - \frac{d}{n-1}}\gamma_{k-1} + \frac{d}{n-1}\right)\EE\left[\|\nabla f(x_{k-1})\|^2\right].
    \end{equation*}
    Taking sum over the $N$ steps, we get
    \begin{equation}\label{eqn:for_PL}
        \alpha \left( 1 - \frac{\alpha L}{2} \right) \sum_{i=1}^{N} \left(\brac{1-\frac{d}{n-1}}\gamma_{i-1} + \frac{d}{n-1}\right)\EE\left[\|\nabla f(x_{i-1})\|^2\right] \leq \EE[f(x_{0}) - f(x_{N-1})] \leq f(x_{0}) - f^\ast,
    \end{equation}
    since $f^\ast \leq f(x)$. Rewriting the inequality, we have
    \begin{equation}
        \min_i \mathbb{E} \left[ \norm*{\nabla f(x_i)}^2 \right] \leq \frac{\alpha^{-1} \left( 1 - \frac{\alpha L}{2} \right)^{-1} (f(x_{0}) - f^\ast)}{\sum_{i=1}^{N} \left(\left(1 - \frac{d}{n-1} \right)\gamma_{i-1} + \frac{d}{n-1}\right)}.
    \end{equation}
    Now, all that is left is to simplify the LHS term. From Lemma \eqref{lem:consec_gamma}, we have
    \begin{align*}
        \gamma_i &= \frac{\left(1 - 2 \alpha L \sqrt{1-\frac{d}{n-1}} \right)\gamma_{i-1} - 2 \alpha L \sqrt{\frac{d}{n-1}}}{\left( 1 + \alpha L \right)^2} \\
        &= a\gamma_{i-1} - b \\
        &= a^i \gamma_0 - b\left(\frac{1-a^i}{1-a} \right),
    \end{align*}
    where
    \begin{align*}
        a &= \frac{1 - 2 \alpha L \sqrt{1-\frac{d}{n-1}}}{\left( 1 + \alpha L \right)^2} \\
        b &= \frac{2 \alpha L \sqrt{\frac{d}{n-1}}}{\left( 1 + \alpha L \right)^2}.
    \end{align*}
    Suppose at the $k$\textsuperscript{th} refresh, the initial alignment is $\gamma_0^{(k)}$ and the alignment threshold is $\delta_k$. Denote the number of steps before the next refresh as $r_k$, i.e. $r_k < \frac{1-2\alpha L}{4\alpha L (1-\alpha L)} \log \left( \frac{\gamma_0^{(k)}}{\delta_k + \sqrt{\frac{d}{n-1}}} \right)$ by Corollary \eqref{cor:ref_bound}. For simplicity, assume we have $\kappa$ number of refreshes, i.e. $N = \sum_{i=1}^\kappa r_i$. Then, the sum over $\gamma_i$ can be simplified to be
    \begin{align*}
        \sum_{i=0}^{N-1} \gamma_i
            &= \sum_{k=1}^\kappa \sum_{i=0}^{r_{k} - 1} \gamma_i^{(k)} \\
            &= \sum_{k=1}^\kappa \sum_{i=0}^{r_{k}-1} \left( a^i \gamma_0^{(k)} - b\left(\frac{1-a^i}{1-a} \right) \right)\\
            &= \sum_{k=1}^\kappa \left( \gamma_0^{(k)}\sum_{i=0}^{r_{k}-1} a^i - b \sum_{i=1}^{r_{k}-1} \left(\frac{1-a^i}{1-a} \right) \right) \\
            &= \sum_{k=1}^\kappa \left[\gamma_0^{(k)} \frac{1-a^{r_{k}}}{1-a} - \frac{b}{1-a} \left( r_{k}-1 - \frac{a(1-a^{r_{k}-1})}{1-a} \right)\right] \\
            &= \frac{1}{1-a} \sum_{k=1}^\kappa \left[\gamma_0^{(k)} (1-a^{r_{k}}) - b \left( r_{k}-1 - \frac{a}{1-a} + \frac{a^{r_{k}}}{1-a} \right)\right].
    \end{align*}
    Since $b > 0$ and $\gamma_0^{(k)} > 0$, we have $\gamma_0^{(k)} (1-a^{r_{k}}) > \gamma_0^{(k)} (1-a)$ and $-a^{r_{k}}b > -ab$, giving us
    \begin{align*}
        \sum_{i=0}^{N-1} \gamma_i
            &\geq \frac{1}{1-a} \sum_{k=1}^\kappa \left[\gamma_0^{(k)} (1-a) - b \left( r_{k}-1 - \frac{a}{1-a} + \frac{a}{1-a} \right)\right]  \\
            &= \sum_{k=1}^\kappa \left[\gamma_0^{(k)} - b \left( \frac{r_{k}-1}{1-a} \right)\right]  \\
            &= \sum_{k=1}^\kappa \gamma_0^{(k)} - \frac{b}{1-a} \left( \sum_{k=1}^\kappa (r_{k}-1) \right)  \\
            &= \sum_{k=1}^\kappa \gamma_0^{(k)} - \frac{b}{1-a} (N - \kappa).
    \end{align*}
    For the 2nd term, we can simplify it to be
    \begin{align*}
        \frac{b}{1-a} 
            &=  \left(\frac{2\alpha L \sqrt{\frac{d}{n-1}}}{(1+\alpha L)^2} \right) \left( 1 - \frac{1-2\alpha L \sqrt{1-\frac{d}{n-1}}}{(1+\alpha L)^2} \right)^{-1} \\
            &=  \frac{2\alpha L}{1 + 2\alpha L + \alpha^2 L^2 - \left(1-2\alpha L \sqrt{1-\frac{d}{n-1}}\right)} \sqrt{\frac{d}{n-1}} \\
            &=  \frac{1}{\frac{\alpha L}{2} + \left(1+\sqrt{1-\frac{d}{n-1}}\right) } \sqrt{\frac{d}{n-1}} \\
            &= t_0 \sqrt{\frac{d}{n-1}},
    \end{align*}
    where
    \begin{equation}
        t_0 = \frac{1}{\frac{\alpha L}{2} + \left(1+\sqrt{1-\frac{d}{n-1}}\right) } \in \left(\frac{4}{9}, \frac{1}{2} + O \left(\frac{d}{n} \right)\right).
    \end{equation}
    Consequently, we have
    \begin{align*}
        \sum_{i=0}^{N-1} \left( \brac{1-\frac{d}{n-1}}\gamma_i + \frac{d}{n-1} \right) 
        &\geq \brac{1-\frac{d}{n-1}} \sum_{k=1}^\kappa \gamma_0^{(k)} - t_0 \brac{1-\frac{d}{n-1}}\sqrt{\frac{d}{n-1}} (N - \kappa) + N\frac{d}{n-1}.
    \end{align*}
    To find $\kappa$, we can sum the $r_i's$ to get
    \begin{equation*}
        N = \sum_{i=i}^\kappa r_i \leq \frac{1-2\alpha L}{4\alpha L (1-\alpha L)} \sum_{i=1}^\kappa \log \left( \frac{\gamma_0^{(i)}}{\delta_i + \sqrt{\frac{d}{n-1}}} \right).
    \end{equation*}
    Equivalently, we have
    \begin{equation*}
        \frac{4\alpha L (1 -\alpha L)}{1 -2 \alpha L} N \leq \log\left( \prod_{i=1}^\kappa \frac{\gamma_0^{(i)}}{\delta_i + \sqrt{\frac{d}{n-1}}}\right) = \sum_{i=1}^\kappa \log \left( \frac{\gamma_0^{(i)}}{\delta_i + \sqrt{\frac{d}{n-1}}} \right).
    \end{equation*}
    Since we have $\delta_k = \delta$ and $\gamma^{(k)}_0 = \gamma_0$, we have
    \begin{equation*}
        \kappa \geq \frac{4\alpha L (1-\alpha L) N}{(1 - 2\alpha L)\log(\gamma_0/\tilde \delta)},
    \end{equation*}
    where $\tilde \delta = \delta + \sqrt{\frac{d}{n-1}}$. For $\alpha$ satisfying
    \[
        \alpha \geq \frac{1}{L} \cdot \frac{K}{2 + K + \sqrt{4 + K^2}} , \quad K = \log\left(\frac{\gamma_0}{\tilde\delta}\right)\cdot
    \frac{t_0\left(1-\frac{d}{n-1}\right)\sqrt{\frac{d}{n-1}}-\frac{d}{n-1}}{\left(1-\frac{d}{n-1}\right)\left(\gamma_0+t_0\sqrt{\frac{d}{n-1}}\right)},
    \]
    we have 
    \[
        \frac{d}{n-1} + \frac{4\alpha L (1-\alpha L)}{(1-2\alpha L)\log(\gamma_0/\tilde \delta)} \brac{1-\frac{d}{n-1}} \gamma_0 - \tilde t_0 > 0,
    \]
    where $\tilde t_0 = t_0 \left(1 - \frac{d}{n-1} \right) \sqrt{\frac{d}{n-1}}\left(1 - \frac{4\alpha L(1-\alpha L)}{(1-2\alpha L)\log(\gamma_0/\tilde \delta)}\right)$. Substituting this back into the equation above, we have
    \begin{equation}
        N \min_{k \in [N]} \mathbb{E} \left[ \norm*{\nabla f(x_k)}^2 \right] \leq \frac{\alpha^{-1} \left(1- \frac{\alpha L}{2}\right)^{-1} (f(x_0) - f^\ast)}{\frac{d}{n-1} + \frac{4\alpha L (1-\alpha L)}{(1-2\alpha L)\log(\gamma_0/\tilde \delta)} \brac{1-\frac{d}{n-1}} \gamma_0 - \tilde t_0},
    \end{equation}
     It then suffices to have
    \begin{equation}
        N \geq \frac{\alpha^{-1} \left(1- \frac{\alpha L}{2}\right)^{-1}(f(x_0) - f^\ast)}{\frac{4\alpha L (1-\alpha L)}{(1-2\alpha L)\log(\gamma_0/\tilde \delta)} \left(1 - \frac{d}{n-1}\right) \gamma_0 + \frac{d}{n-1} - \tilde t_0}\epsilon^{-2}
    \end{equation}
    to obtain $\min_{k \in [N]} \mathbb{E} \left[ \norm*{\nabla f(x_k)}^2 \right] < \epsilon^2$.
\end{proof}

\subsection{Proof for Computational Cost (Theorem \ref{thm:cost_vanilla})}
To compute the total cost of the algorithm, we let $r$ be the number of steps before we refresh the correlation vector $v_k = \gfk$. For each step where $v_k = v_{k-1}$, the computational cost is $O(nd + d\xi)$, which is the cost of computing $d$ directional derivatives and the cost of the matrix-vector multiplication. For $v_k = \gfk$, the cost is $O(nd + \nu)$ instead (the matrix-vector multiplication still costs the same). 

The total computational cost can then be computed as
    \begin{align*}
        &\underbrace{N \times (nd + d\xi)}_{v_k = v_{k-1}} + \underbrace{\frac{N}{r} \times (nd + \nu)}_{v_k = \gfk} \nonumber\\
            &= \frac{\alpha^{-1} \left(1- \frac{\alpha L}{2}\right)^{-1}\epsilon^{-2}}{\left(1 - \frac{d}{n-1}\right)\gamma_0/r + \frac{d}{n-1} - t_0\left(1 - \frac{d}{n-1}\right)\sqrt{\frac{d}{n-1}} \left(1 - 1/r\right)}(f(x_0) - f^\ast) \left[\left( d\xi + nd \right) + \frac{1}{r} \left( nd + \nu \right)\right] \\
            &= \frac{\alpha^{-1} \left(1- \frac{\alpha L}{2}\right)^{-1}\epsilon^{-2}}{\left(1 - \frac{d}{n-1}\right)\gamma_0 + r\frac{d}{n-1} - t_0\left(1 - \frac{d}{n-1}\right)\sqrt{\frac{d}{n-1}} \left(r - 1\right)}(f(x_0) - f^\ast) \left[r\left( d\xi + nd \right) + \left( nd + \nu \right)\right] \\
            &= O\left(\epsilon^{-2} \gamma_0\left[rd\left(n + \xi \right) + \nu \right]\right) \\
            &= O\left(\epsilon^{-2} \gamma_0^{-1}\left[d \log \left( \frac{\gamma_0}{\delta} \right) \left(\xi + n \right) + \nu \right]\right).
    \end{align*}
\section{Proof for Sparse Variant}
This section wil first show derivations for the computational complexity of our proposed algorithm in the sparse setting, followed by the analysis of the classical SSD algorithm in the same setting.

\subsection{Proof for Refresh Cost in Sparse Case (Proposition \ref{prop:sparse_refresh_cost})}
In this section, we will analyse the computational cost of the IHT algorithm applied to our problem. Recall that the algorithm is
\[
    y_{t+1} = H_s\left[y_t + \mu \Psi^\top (z - \Psi y_t)\right],
\]
where $z = \Psi \nabla f(x)$, $\Psi \in \mathbb{R}^{k' \times n}$, $y_0 \in \mathbb{R}^n$ and $H_s$ is the hard thresholding function. The initial cost to compute $z$ is $O(k' \xi)$, while each iteration is simply a matrix vector multiplication costing $O(k'n)$. The function $H_s$ at most requires $O(n\log(n))$ with sorting (which could be improved to $O(n)$ with partitioning, but the cost is dominated by $O(k'n)$ either way as $k' \gtrsim s\log(n/s)$). This means that 1 iteration of IHT has a complexity of $O(k'n + n\log(n)) = O(k'\xi)$. Consequently, $T$ iterations has a complexity of $O(Tk'n)$.

For a fixed initial alignment $\gamma_0$, $\gamma_0 = (1-\rho^2) \times p$ as shown in Equation \eqref{eqn:sparse_gamma_0}, where $\mathbb{P}[\mathcal{A}] \geq p$ for an appropriately defined event $\mathcal{A}$. By Corollary 1 of \cite{davies2009iht}, $\mathcal{A} = \{\delta_{3s} \leq \frac{1}{15}\}$, and with Gaussian matrices $\Psi \in \mathbb{R}^{k' \times n}$ where $k' \gtrsim n \log(n/s)$ and $\Psi_{ij} \sim N(0, 1/k')$, we have by \cite{baraniuk2008ripgauss} that $\mathbb{P}[\mathcal{A}] \geq 1 - Ce^{-ck'}$, where $C, c>0$ constants. So, we have
\[
    \rho(\gamma_0) = \sqrt{1 - \frac{\gamma_0}{1 - Ce^{-ck'}}} = \sqrt{\frac{1 - Ce^{-ck'} - \gamma_0}{1 - Ce^{-ck'}}},
\]
coupled with $T = \lceil \log(1/p)/\log(2) \rceil$ allows us to simplify the cost to
\[
    \nu(\rho(\gamma_0)) + O(k'\xi) = O \left( \log\left(\frac{1 - Ce^{-ck'}}{1 - Ce^{-ck'} - \gamma_0}\right) k' n + k'\xi \right),
\]
Note that when $Ce^{-ck'} \ll 1$, we have the logarithm term to be approximately $-\log(1-\gamma_0)$, which will not be too large whenever $\gamma_0$ is not too close to $1$.

\subsection{Proof for Computational Cost in Sparse Case (Theorem \ref{thm:cost_sparse})}

To compute the computational cost of the algorithm in the case where the function is linearly sparse, we do similar calculations as in the proof for Theorem \ref{thm:cost_vanilla}.
    Let $r$ be the number of steps before $v_k = y_T$, where $y_T$ is the estimation of $\nabla f(x_k)$ obtained from $T$ iterations of IHT. For each step where $v_k = v_{k-1}$, the computational cost is $O(nd + d\xi)$. In the update step, the computational cost is $O \left( v(\gamma_0) k' n + k'\xi + nd + d\xi \right)$, where $v(\gamma_0) =  \log\left(\frac{1 - Ce^{-ck'}}{1 - Ce^{-ck'} - \gamma_0}\right)$. The total computational cost will be given as
    \begin{align*}
        &\underbrace{\left( N - \frac{N}{r} \right) \times (nd + d\xi)}_{v_k = v_{k-1}} + \underbrace{\frac{N}{r} \times \left( v(\gamma_0) k' n + k'\xi + nd + d\xi \right)}_{v_k = \gfk} \\
            &= N \left(nd + d\xi + \frac{v(\gamma_0)k'n + k'\xi}{r}\right) \\
            &= \frac{\alpha^{-1} \left(1- \frac{\alpha L}{2}\right)^{-1}\epsilon^{-2}(f(x_0) - f^\ast)}{\left(1 - \frac{d}{n-1}\right)\gamma_0 + r\frac{d}{n-1} - t_0\left(1 - \frac{d}{n-1}\right)\sqrt{\frac{d}{n-1}} \left(r - 1\right)} \left( r(nd + d\xi) + v(\gamma_0)k'n + k'\xi\right).
    \end{align*}
    Suppose we assume $\gamma_0/r \gg d/(n-1)$, then we have
    \begin{align*}
        \mbox{Cost } 
            &\approx \frac{\alpha^{-1} (1-\alpha L /2)^{-1} \epsilon^{-2}}{\gamma_0 + (1-t_0)rd/(n-1)}(f(x_0) - f^\ast) \left[ r( nd + d\xi ) + v(\gamma_0) k'n + k'\xi\right] \\
            &= O \left( \epsilon^{-2} \gamma_0^{-1} \left[ r (nd + d\xi) + v(\gamma_0) k'n + k'\xi \right] \right) \\
            &= O\left( \epsilon^{-2} \gamma_0^{-1} \left[ (rd + v(\gamma_0) k')n + (rd + k')\xi \right] \right) \\
            &= O\left( \epsilon^{-2} \gamma_0^{-1} \left[ (d\log(\gamma_0/\delta) + v(\gamma_0) k') \times n + (d\log(\gamma_0/\delta) + k') \times \xi \right] \right),
    \end{align*}
    where $r = O(\log(\gamma_0/\delta))$ as in Corollary \eqref{cor:ref_bound}. Substituting $v(\gamma_0)$ back in and assuming $Ce^{-ck'} \ll 1$, we get
    \begin{equation}
        O\left( \frac{\epsilon^{-2}}{\gamma_0} \left[ \left( d\log(\gamma_0/\delta) + k' \log \left(\frac{1}{1-\gamma_0}\right) \right) \times n + \left( d\log(\gamma_0/\delta) + k' \right) \times \xi \right] \right),
    \end{equation}
    where $k' \gtrsim s\log(n/s)$.

\subsection{Analysis of Classical SSD in Sparse Setting (Theorem \ref{thm:sparse_classic})}

\begin{thm}\label{thm:conv2-explicit-simple}
Assume that \(f\) has intrinsic dimension \(s<n\), so that
\[
f(x)=g(Rx),\qquad \Pi=R^\top R,
\]
where \(R\in\mathbb R^{s\times n}\) has orthonormal rows. Let
\[
x_{k}=x_k-\alpha P_{k-1}P_{k-1}^\top \nabla f(x_{k-1}),
\]
where \(P_{k-1}\) are i.i.d. random matrices uniformly distributed on
\(\mathrm{St}(n,d)\), and define \(y_k:=Rx_k\). Assume that \(g\) is
\(L\)-smooth. Suppose that
\begin{equation}\label{eq:size-cond-final}
    \max\!\left\{1,2\log\!\left(\frac{2n^2}{9s}\right)\right\}\le d\le \frac{s}{16}, \qquad \alpha=\frac{n}{18sL},
\end{equation}
then
\begin{equation}\label{eq:main-bound-final}
    \min_{1 \le k\le N}\mathbb E\!\left[\|\nabla f(x_{k-1})\|^2\right]
        \le
    \frac{36Ls}{Nd}\,\bigl(f(x_0)-f^\ast\bigr).
\end{equation}
Consequently, if
\[
    N\ge \frac{36Ls}{d\,\epsilon^2}\bigl(f(x_0)-f^\ast\bigr),
\]
then
\[
    \min_{1\le k\le N}\mathbb E\!\left[\|\nabla f(x_{k-1})\|^2\right] \le \epsilon^2.
\]
\end{thm}

\begin{proof}
Since \(f(x)=g(Rx)\), the chain rule gives
\begin{equation}\label{eq:grad-relation-final}
    \nabla f(x)=R^\top \nabla g(Rx).
\end{equation}
Hence
\[
    \|\nabla f(x_k)\|=\|\nabla g(y_k)\|,
    \qquad
    f(x_k)=g(y_k).
\]

Let
\[
    u_k:=\nabla g(y_k),
    \qquad
    B_k:=RP_k(RP_k)^\top = R P_k P_k^\top R^\top \in \mathbb R^{s\times s}.
\]
Then
\[
y_{k}=Rx_{k}
=Rx_{k-1}-\alpha R P_{k-1}P_{k-1}^\top \nabla f(x_{k-1})
= y_{k-1}-\alpha B_{k-1} u_{k-1},
\]
where we used \eqref{eq:grad-relation-final}. By \(L\)-smoothness of \(g\),
\begin{equation}\label{eq:smoothness-step-final}
    g(y_{k})-g(y_{k-1})
        \le -\alpha\,u_{k-1}^\top B_{k-1} u_{k-1} + \frac{\alpha^2 L}{2} \norm*{B_{k-1} u_{k-1}}^2.
\end{equation}

We now estimate the two terms on the right-hand side in conditional expectation. Let
\[
    \mathcal F_k:=\sigma(P_0,\dots,P_{k-1}).
\]
Since \(P_{k-1}\) is independent of \(\mathcal F_{k-1}\), conditioning on \(\mathcal F_{k-1}\) allows us to regard \(u_{k-1}\) as fixed.

\medskip

\noindent
\paragraph{Step 1: The Linear Term.}
Since \(P_{k-1}\) is Haar-distributed on \(\mathrm{St}(n,d)\), rotational invariance gives
\[
    \mathbb E\left[P_{k-1} P_{k-1}^\top\right]=\frac{d}{n}I_n.
\]
Therefore
\[
    \mathbb E\left[ B_{k-1} \mid \mathcal F_{k-1} \right]
        = R\,\mathbb E \left[ P_{k-1} P_{k-1} ^\top \right]\,R^\top
        = \frac{d}{n}RR^\top
        = \frac{d}{n}I_s,
\]
and thus
\begin{equation}\label{eq:linear-term-final}
    \mathbb E\!\left[u_{k-1}^\top B_{k-1} u_{k-1}\mid \mathcal F_{k-1}\right]
        = u_{k-1}^\top \mathbb E \left[ B_{k-1}\mid \mathcal F_{k-1} \right]u_{k-1}
        = \frac{d}{n}\|u_{k-1}\|^2.
\end{equation}

\paragraph{Step 2: The Quadratic Term.}
Write
\[
    P_{k-1} = G_{k-1} \left(G_{k-1}^\top G_{k-1} \right)^{-1/2},
\]
where \(G_k\in\mathbb R^{n\times d}\) is a standard Gaussian matrix. Then
\[
    B_{k-1} = R G_{k-1} \left(G_{k-1}^\top G_{k-1} \right)^{-1} G_{k-1}^\top R^\top.
\]
Since \(R\) has orthonormal rows, \(R G_{k-1}\in\mathbb R^{s\times d}\) is also a standard Gaussian matrix. A direct computation gives
\[
    u_{k-1}^\top B_{k-1} u_{k-1} 
        = \norm*{(G_{k-1}^\top G_{k-1})^{-1/2}(R G_{k-1})^\top u_{k-1}}^2.
\]
Moreover,
\begin{align}
    \|B_{k-1} u_{k-1}\|^2
        &= \norm*{R G_{k-1} (G_{k-1}^\top G_{k-1})^{-1}(R G_{k-1})^\top u_{k-1}}^2 \nonumber\\
        &\le \lambda_{\max}\!\Big( (G_{k-1}^\top G_{k-1})^{-1/2}(R G_{k-1})^\top(R G_{k-1})(G_{k-1}^\top G_{k-1})^{-1/2} \Big) \times \nonumber \\
        &\qquad \norm*{(G_{k-1}^\top G_{k-1})^{-1/2}(R G_{k-1})^\top u_{k-1}}^2. \label{eq:hp-pointwise-final}
\end{align}

Now fix \(t=\sqrt d\), and define
\begin{equation}\label{eq:beta-final}
    \beta := \frac{(\sqrt s+2\sqrt d)^2}{(\sqrt n-2\sqrt d)^2}.
\end{equation}
Let
\[
    \mathcal E_{k-1}:= \left\{ \lambda_{\max}\!
    \Big( (G_{k-1}^\top G_{k-1})^{-1/2}(R G_{k-1})^\top(R G_{k-1})(G_{k-1}^\top G_{k-1})^{-1/2} \Big) 
    \le \beta \right\},
\]
and
\[
    p_{k-1} := \mathbb P \left[\mathcal E_{k-1}^c \mid \mathcal F_{k-1} \right].
\]

On \(\mathcal E_{k-1}\), \eqref{eq:hp-pointwise-final} yields
\[
    \norm*{B_{k-1} u_{k-1}}^2\le \beta\, u_{k-1}^\top B_{k-1} u_{k-1}.
\]
On \(\mathcal E_{k-1}^c\), we only need to use that \(B_{k-1}\) is a positive contraction. Indeed, for every \(v\in\mathbb R^s\),
\[
    v^\top B_{k-1} v
        = \|P_{k-1}^\top R^\top v\|^2
        \le \|R^\top v\|^2
        = v^\top RR^\top v
        = \|v\|^2,
\]
because \(P_{k-1} P_{k-1}^\top \preceq I_n\) and \(RR^\top = I_s\). Therefore
\[
    0\preceq B_{k-1}\preceq I_s,
        \qquad\text{hence}\qquad
    \|B_{k-1} u_{k-1}\|\le \|u_{k-1}\|.
\]
Splitting according to \(\mathcal E_k\), we obtain
\begin{align}
\mathbb E\!\left[\|B_{k-1} u_{k-1}\|^2\mid \mathcal F_{k-1}\right]
    &= \mathbb E\!\left[\|B_{k-1} u_{k-1}\|^2\mathbf 1_{\mathcal E_{k-1}}\mid \mathcal F_{k-1}\right] + \mathbb E\!\left[\|B_{k-1} u_{k-1}\|^2\mathbf 1_{\mathcal E_{k-1}^c}\mid \mathcal F_{k-1}\right] \nonumber\\
    &\le \beta\,\mathbb E\!\left[u_{k-1}^\top B_{k-1} u_{k-1}\mid \mathcal F_{k-1}\right] + p_{k-1}\|u_{k-1}\|^2 \nonumber\\
    &= \left(\beta\frac{d}{n}+p_{k-1}\right)\|u_{k-1}\|^2, \label{eq:quadratic-cond-final}
\end{align}
where we used \eqref{eq:linear-term-final}.

\paragraph{Step 3: Bound on \(p_k\).}
Since
\[
    \lambda_{\max}\!
    \Big((G_{k-1}^\top G_{k-1})^{-1/2}(R G_{k-1})^\top(R G_{k-1})(G_{k-1}^\top G_{k-1})^{-1/2}\Big)
        \le \frac{\lambda_{\max}((R G_{k-1})^\top(R G_{k-1}))}{\lambda_{\min}(G_{k-1}^\top G_{k-1})},
\]
the event \(\mathcal E_{k-1}^c\) is contained in
\[
    \left\{\sigma_{\max}(R G_{k-1})>\sqrt s+2\sqrt d \right\} 
    \cup
    \left\{\sigma_{\min}(G_{k-1})<\sqrt n-2\sqrt d \right\}.
\]
By the standard Gaussian singular value bounds,
\[
    \mathbb P\!\left(\sigma_{\max}(R G_{k-1})>\sqrt s+2\sqrt d\right)\le e^{-d/2},
\]
and
\[
    \mathbb P\!\left(\sigma_{\min}(G_{k-1})<\sqrt n-2\sqrt d\right)\le e^{-d/2},
\]
where the second estimate is valid because \(2\sqrt d<\sqrt n\) by \eqref{eq:size-cond-final}. Therefore, by a union bound,
\begin{equation}\label{eq:pk-final}
    p_{k-1} 
        = \mathbb P \left[ \mathcal E_{k-1}^c \mid \mathcal F_{k-1} \right]
        = \mathbb P \left[ \mathcal E_{k-1}^c \right]
        \le 2e^{-d/2}.
\end{equation}

By the lower bound on \(d\) in \eqref{eq:size-cond-final},
\[
    2e^{-d/2}\le \frac{9s}{n^2}.
\]
Since \(d\ge 1\), this implies that
\begin{equation}\label{eq:pk-simplified-final}
    p_{k-1} \le \frac{9sd}{n^2}.
\end{equation}

\paragraph{Step 4: Simplify \(\beta\).}
From \(d\le s/16\), we have \(2\sqrt d\le \sqrt s/2\), so
\[
\sqrt s+2\sqrt d\le \frac32\sqrt s.
\]
From \(d\le n/16\), we have \(2\sqrt d\le \sqrt n/2\), so
\[
\sqrt n-2\sqrt d\ge \frac12\sqrt n.
\]
Therefore
\begin{equation}\label{eq:beta-upper-final}
\beta
=
\frac{(\sqrt s+2\sqrt d)^2}{(\sqrt n-2\sqrt d)^2}
\le
\frac{(\frac32\sqrt s)^2}{(\frac12\sqrt n)^2}
=
9\,\frac{s}{n}.
\end{equation}
Combining \eqref{eq:pk-simplified-final} and \eqref{eq:beta-upper-final}, we get
\begin{equation}\label{eq:beta-p-final}
    \beta\frac{d}{n}+p_{k-1}
        \le 9\,\frac{sd}{n^2}+9\,\frac{sd}{n^2}
        = 18\,\frac{sd}{n^2}.
\end{equation}

\paragraph{Step 5: One-step Descent in Expectation.}
Taking conditional expectation in \eqref{eq:smoothness-step-final} and using
\eqref{eq:linear-term-final} and \eqref{eq:quadratic-cond-final}, we obtain
\begin{align*}
    \mathbb E\!\left[g(y_{k-1})-g(y_{k})\mid \mathcal F_{k-1}\right]
        &\ge \alpha\,\frac{d}{n}\|u_{k-1}\|^2 -\frac{\alpha^2 L}{2} \left(\beta\frac{d}{n}+p_{k-1}\right)\|u_{k-1}\|^2 \\
        &\ge \left( \alpha\frac{d}{n} - \frac{\alpha^2 L}{2} \cdot 18\,\frac{sd}{n^2} \right)\|u_{k-1}\|^2.
\end{align*}
With the choice \(\alpha=\frac{n}{18sL}\), this becomes
\begin{align*}
    \mathbb E\!\left[g(y_{k-1})-g(y_{k})\mid \mathcal F_{k-1}\right]
        &\ge \left(\frac{d}{18sL} - \frac{L}{2}\cdot \frac{n^2}{(18sL)^2}\cdot 18\,\frac{sd}{n^2} \right)\|u_{k-1}\|^2 \\
        &= \left( \frac{d}{18sL} - \frac{d}{36sL} \right)\|u_{k-1}\|^2 \\
        &= \frac{d}{36sL}\|u_{k-1}\|^2.
\end{align*}
Using \(\|u_{k-1}\|=\|\nabla f(x_{k-1})\|\), we obtain
\begin{equation}\label{eq:one-step-final}
    \mathbb E\!\left[f(x_{k-1})-f(x_{k})\mid \mathcal F_{k-1}\right]
        \ge \frac{d}{36sL}\,\|\nabla f(x_{k-1})\|^2.
\end{equation}
Taking expectation again yields
\[
    \mathbb E\!\left[f(x_{k-1})-f(x_{k})\right]
        \ge \frac{d}{36sL}\, \mathbb E\!\left[\|\nabla f(x_{k-1})\|^2\right].
\]

\paragraph{Step 6: Telescoping.}
Summing over \(k=0,\dots,N-1\), we obtain
\[
    \frac{d}{36sL} \sum_{k=1}^{N} \mathbb E\!\left[\|\nabla f(x_{k-1})\|^2\right]
    \le \sum_{k=1}^{N}\mathbb E\!\left[f(x_{k-1})-f(x_{k})\right]
    = f(x_0)-\mathbb E[f(x_N)]
    \le f(x_0)-f^\ast.
\]
Therefore
\[
    \frac1N \sum_{k=1}^{N}\mathbb E\!\left[\|\nabla f(x_{k-1})\|^2\right]
        \le \frac{36Ls}{Nd}\,\bigl(f(x_0)-f^\ast\bigr).
\]
Finally,
\[
    \min_{1\le k\le N}\mathbb E\!\left[\|\nabla f(x_{k-1})\|^2\right]
    \le
    \frac1N\sum_{k=1}^{N}\mathbb E\!\left[\|\nabla f(x_{k-1})\|^2\right],
\]
which proves \eqref{eq:main-bound-final}.
\end{proof}
\section{Proof for Minibatch Variant}\label{sec:mb_convergence}
This section is dedicated to proving both the iteration and computational complexity of the mini-batch variant of the algorithm. The proof is split into 2 parts: first obtain an expression for the lower bound of the alignment in expectation, then subsequently simplifying the iteration convergence proof with a tweak to the proof used in Section \ref{sec:proof_general}.

\subsection{Proof for Alignment in Minibatch Setting (Proposition \ref{prop:minibatch-refresh})}
The persistence of the alignment tells us that the same guidance vector can still be used in the subsequent gradient descent step. To show this, we split the proof into 2 main parts.
\begin{enumerate}
    \item Find a recurrence for $\mathbb{E} [ \langle \hat{v}_{k}, \gfk \rangle^2 ] \geq \gamma_k \mathbb{E} [\|\gfk\|^2] - \eta_k \sigma^2$, where $\eta_k$ can even potentially be $0$.
    \item Show that $\gamma_k > \delta$ is satisfied under the assumption $r_{i+1} - r_i = i^\beta$.
\end{enumerate}

\paragraph{Part 1a: Obtain a recursion with respect to the previous gradient.}
Define the $\sigma$-algebras $\mathcal{F}_k = \sigma(\nabla g_0, \widetilde P_0, \nabla g_1, \widetilde P_1, \cdots, \nabla g_{k-1}, \widetilde P_{k-1})$ and $\mathcal{G}_k = \sigma(\mathcal{F}_k, \nabla g_k)$. Then, by re-writing $\gfk = \gfk - \gfkk + \gfkk$ like before, we can get the following lower bound:
\begin{align*}
    \langle \hat{v}_{k}, \gfk \rangle^2 
        &\geq \langle \hat{v}_{k}, \gfkk \rangle^2 - 2 L |\langle \hat{v}_{k}, \gfkk \rangle | \|x_k - x_{k-1}\|.
\end{align*}
For the second term on the RHS, the descent direction is the projected gradient of the gradient estimate, instead of the projected gradient of the true gradient:
\[
    x_k - x_{k-1} = - \alpha_{k-1} P_{k-1} P_{k-1}^\top \nabla g_{k-1}.
\]
Instead of taking conditional expectation under $\mathcal{F}_{k-1}$, we first take it under $\mathcal{G}_{k-1}$ to get
\begin{align*}
    \mathbb{E} \left[ \langle \hat{v}_{k}, \gfk \rangle^2 \mid \mathcal{G}_{k-1} \right]
        &\geq \langle \hat{v}_{k}, \gfkk \rangle^2 - 2 \alpha_{k-1} L \left|\left\langle \hat{v}_{k}, \gfkk \right\rangle\right| \mathbb{E} \left[ \|P_{k-1} P_{k-1}^\top \nabla g_{k-1}\| \mid \mathcal{G}_{k-1} \right],
\end{align*}
where the expectation term on the RHS can be bounded similarly as before to obtain
\begin{align*}
    \mathbb{E} \left[ \|P_{k-1} P_{k-1}^\top \nabla g_{k-1}\| \mid \mathcal{G}_{k-1} \right]
        &\leq \sqrt{\left(1-\frac{d}{n-1}\right)} \left|\left\langle \vkkh, \nabla g_{k-1} \right\rangle\right| + \sqrt{\frac{d}{n-1}} \|\nabla g_{k-1}\|,
\end{align*}
finally giving us
\begin{align*}
    \mathbb{E} \left[ \langle \hat{v}_{k}, \gfk \rangle^2 \mid \mathcal{G}_{k-1} \right]
        &\geq \langle \hat{v}_{k-1}, \gfkk \rangle^2 - 2 \alpha_{k-1} L |\langle \hat{v}_{k-1}, \gfkk \rangle | \times\\
        &\quad \left( \sqrt{ \left(1-\frac{d}{n-1}\right)} |\langle \vkkh, \nabla g_{k-1} \rangle| + \sqrt{\frac{d}{n-1}} \|\nabla g_{k-1}\| \right).
\end{align*}
Compared to earlier, where we only had $\gfkk$ throughout, here we have $\nabla g_{k-1}$, so we have to work around differently. Using the identity $2ab \leq a^2 + b^2$, we can bound the second term by
\begin{equation*}
    |\langle \hat{v}_{k-1}, \gfkk \rangle |  |\langle \vkkh, \nabla g_{k-1} \rangle| \leq \frac{1}{2} \langle \hat{v}_{k-1}, \gfkk \rangle^2 + \frac{1}{2}  \langle \vkkh, \nabla g_{k-1} \rangle^2.
\end{equation*}
As for the third term, taking expectation under $\mathcal{F}_{k-1}$, we can obtain
\begin{align*}
    \mathbb{E} \left[  \left|\langle \hat{v}_{k-1}, \gfkk \rangle\right| \|\nabla g_{k-1}\| \mid \mathcal{F}_{k-1} \right]
        &\leq \mathbb{E} \left[ \|\nabla g_{k-1}\| \mid \mathcal{F}_{k-1} \right] \cdot \left\|\gfkk\right\| \\
        &\leq \|\gfkk\|\sqrt{\sigma_m^2 + \|\gfkk\|^2} \\
        &\leq \frac{1}{2}\|\gfkk\|^2 + \frac{1}{2} \left(\sigma_m^2 + \|\gfkk\|^2 \right) \\
        &= \|\gfkk\|^2 + \frac{1}{2}\sigma_m^2,
\end{align*}
where we used the same inequality as before and the upper bound of the variance of the mini-batch. Substituting everything back into the conditional expectation of the correlation, we obtain the lower bound
\begin{align*}
    \mathbb{E} \left[ \langle \hat{v}_{k}, \gfk \rangle^2 \mid \mathcal{G}_{k-1} \right]
        &\geq \langle \hat{v}_{k-1}, \gfkk \rangle^2 -2\alpha_{k-1}L \sqrt{\frac{d}{n-1}} \left( \|\gfkk\|^2 + \frac{1}{2} \sigma_m^2 \right) \\
        &\qquad - \alpha_{k-1} L \sqrt{1-\frac{d}{n-1}} \Big( \langle \hat{v}_{k-1}, \gfkk \rangle^2 + \mathbb{E} [\langle \vkkh, \nabla g_{k-1} \rangle^2 \mid \mathcal{F}_{k-1}] \Big) \\
        &= \left( 1- \alpha_{k-1} L \sqrt{1 - \frac{d}{n-1}} \right) \langle \vkkh, \gfkk \rangle^2 - \alpha_{k-1}L \sqrt{\frac{d}{n-1}} \sigma_m^2\\
        &\qquad - \alpha_{k-1} L \sqrt{1-\frac{d}{n-1}} \mathbb{E} \left[\langle \vkkh, \nabla g_{k-1} \rangle^2 \mid \mathcal{F}_{k-1} \right] - 2\alpha_{k-1}L\sqrt{\frac{d}{n-1}} \|\gfkk\|^2 .
\end{align*}
For the second term, we can similarly express $\nabla g_{k-1} = \nabla g_{k-1} - \gfkk + \gfkk$ to get
\begin{align*}
    \mathbb{E} \left[\langle \vkkh, \nabla g_{k-1} \rangle^2 \mid \mathcal{F}_{k-1} \right]
        &= \mathbb{E} \left[ \left(\langle \vkkh, \nabla g_{k-1} -\gfkk \rangle + \langle \vkkh, \gfkk \rangle \right)^2  \mid \mathcal{F}_{k-1} \right] \\
        &= \mathbb{E} \left[\langle \vkkh, \nabla g_{k-1} -\gfkk \rangle^2 + \langle \vkkh, \gfkk \rangle^2 \mid \mathcal{F}_{k-1} \right]
\end{align*}
where the cross term cancels out due to $\gfkk = \mathbb{E}[\nabla g_{k-1}]$. 
We can then bound the first term with Cauchy-Schwartz and the variance assumption to get
\begin{align*}
    \mathbb{E} \left[\langle \vkkh, \nabla g_{k-1} \rangle^2 \mid \mathcal{F}_{k-1} \right]
        &\leq \sigma_m^2 + \langle \vkkh, \gfkk \rangle^2.
\end{align*}
Substituting this bound and simplifying everything, we obtain
\begin{align*}
    \mathbb{E} \left[ \langle \hat{v}_{k}, \gfk \rangle^2 \mid \mathcal{G}_{k-1} \right]
        &\geq \left( 1- 2\alpha_{k-1} L \sqrt{1 - \frac{d}{n-1}} \right) \langle \vkkh, \gfkk \rangle^2 \\
        &\qquad - 2\alpha_{k-1}L\sqrt{\frac{d}{n-1}} \|\gfkk\|^2 - \alpha_{k-1}L \left(\sqrt{\frac{d}{n-1}} + \sqrt{1-\frac{d}{n-1}} \right) \sigma_m^2 \\
        &\geq \left( 1- 2\alpha_{k-1} L \sqrt{1 - \frac{d}{n-1}} \right) \langle \vkkh, \gfkk \rangle^2 \\&\qquad- 2\alpha_{k-1}L\sqrt{\frac{d}{n-1}} \|\gfkk\|^2 - 2\alpha_{k-1}L \sigma_m^2.
\end{align*}
Using the assumption that $\mathbb{E} \left[ \langle \hat{v}_{k-1}, \gfkk \rangle^2 \right] \geq \gamma_{k-1} \mathbb{E} [\|\gfkk\|^2] - \eta_{k-1} \sigma_m^2$, we can get
\begin{align*}
    \mathbb{E} [ \langle \hat{v}_{k}, \gfk \rangle^2 ]
        &\geq \left[\left( 1- 2\alpha_{k-1} L \sqrt{1 - \frac{d}{n-1}} \right)\gamma_{k-1} - 2\alpha_{k-1}L\sqrt{\frac{d}{n-1}} \right] \mathbb{E} [\|\gfkk\|^2] \\
        &\qquad - \left[\left( 1- 2\alpha_{k-1} L \sqrt{1 - \frac{d}{n-1}} \right)\eta_{k-1} + 2\alpha_{k-1}L \right] \sigma_m^2.
\end{align*}
\paragraph{Part 1b: Obtain a relationship between norms of consecutive gradient.}
This is done similarly to the vanilla case and the technique is the same; we just have to note that there is a variance bound on the mini-batch gradient instead of directly working with the variance of the true gradient. Using the decomposition of $\gfk = \gfk - \gfkk + \gfkk$, we get
\begin{align*}
    \|\gfk\| &\leq \|\gfk - \gfkk\| + \|\gfkk\| \leq L \|x_k - x_{k-1}\| + \|\gfkk\|.
\end{align*}
For the first term, we have
\begin{align*}
    \mathbb{E} \left[ \|x_k - x_{k-1} \|^2 \mid \mathcal{F}_{k-1}\right]
        &= \alpha_{k-1}^2 \mathbb{E} \left[ \|P_{k-1} P_{k-1}^\top \nabla g_{k-1} \|^2 \mid \mathcal{F}_{k-1}\right] \\
        &\leq \alpha_{k-1}^2 \mathbb{E} \left[ \|\nabla g_{k-1} \|^2 \mid \mathcal{F}_{k-1}\right] \\
        &\leq \alpha_{k-1}^2 \left( \|\gfkk\|^2 + \sigma_m^2 \right).
\end{align*}
Consequently, we can get
\begin{align*}
    \mathbb{E} \left[ \|x_k - x_{k-1} \| \mid \mathcal{F}_{k-1}\right]
        &\leq \alpha_{k-1} \left( \|\gfkk\| + \sigma_m \right).
\end{align*}
Then, we have
\begin{align*}
    \mathbb{E}\left[ \| \gfk \|^2 \mid \mathcal{F}_{k-1} \right]
        &\leq \|\gfkk\|^2 + 2L \|\gfkk\| \mathbb{E} \left[  \|x_k - x_{k-1}\| \mid \mathcal{F}_{k-1} \right] \\&\qquad+ L^2  \mathbb{E} \left[  \|x_k - x_{k-1}\|^2 \mid \mathcal{F}_{k-1} \right] \\
        &\leq \|\gfkk\|^2 + 2\alpha_{k-1} L \|\gfkk\| ( \sigma_m + \|\gfkk\| ) \\
        &\qquad + \alpha_{k-1}^2 L^2 (\sigma_m^2 + \|\gfkk\|^2 ) \\
        &= (1 + \alpha_{k-1} L)^2 \|\gfkk\|^2 + 2\alpha_{k-1} L \|\gfkk\|\sigma_m + \alpha_{k-1}^2 L^2 \sigma_m^2.
\end{align*}
The cross-term can once again be bounded by $2\|\gfkk\|\sigma_m \leq \|\gfkk\|^2 + \sigma_m^2$ to get
\begin{align*}
    \mathbb{E} \left[ \| \gfk \|^2 \mid \mathcal{F}_{k-1} \right]
    &\leq (1 + \alpha_{k-1} L)^2 \|\gfkk\|^2 + \alpha_{k-1} L ( \|\gfkk\|^2 +\sigma_m^2 ) + \alpha_{k-1}^2 L^2 \sigma_m^2 \\
    &= \left[ (1 + \alpha_{k-1} L)^2 + \alpha_{k-1} L \right] \|\gfkk\|^2 + \left[ \alpha_{k-1} L + \alpha_{k-1}^2 L^2 \right] \sigma_m^2.
\end{align*}
Rewriting this inequality, we have
\begin{equation}
    \mathbb{E} \left[\|\gfkk\|^2 \right] 
        \geq \frac{1}{(1+\alpha_{k-1} L)^2 + \alpha_{k-1} L} \mathbb{E} \left[\|\gfk\|^2\right] - \frac{(1+\alpha_{k-1} L) \alpha_{k-1} L}{(1+\alpha_{k-1} L)^2 + \alpha_{k-1} L} \sigma_m^2.
\end{equation}
With this bound on the norms of consecutive gradients, we obtain the expression
\begin{align*}
    \mathbb{E} [ \langle \hat{v}_{k}, \gfk \rangle^2 ]
        &= \frac{\left( 1- 2\alpha_{k-1} L \sqrt{1 - \frac{d}{n-1}} \right)\gamma_{k-1} - 2\alpha_{k-1}L\sqrt{\frac{d}{n-1}} }{(1+\alpha_{k-1} L)^2 + \alpha_{k-1} L} \mathbb{E} [\|\gfk\|^2] \\
        &\quad - \left[ 2\alpha_{k-1} L +  \frac{\left[\left( 1- 2\alpha_{k-1} L \sqrt{1 - \frac{d}{n-1}} \right)\gamma_{k-1} - 2\alpha_{k-1}L\sqrt{\frac{d}{n-1}} \right] \left[ 1+\alpha_{k-1} L \right] \alpha_{k-1} L }{(1+\alpha_{k-1} L)^2 + \alpha_{k-1} L} \right. \\
        &\qquad \qquad \left. + \left( 1- 2\alpha_{k-1} L \sqrt{1 - \frac{d}{n-1}} \right) \eta_{k-1} \right] \sigma_m^2 \\
        &= \gamma_k \mathbb{E} [\|\gfk\|^2] - \eta_k \sigma_m^2,
\end{align*}
where we have the recurrence relation on $\gamma_k$ and $\eta_k$ as
\begin{align}
    \gamma_k 
        &= \frac{\left( 1- 2\alpha_{k-1} L \sqrt{1 - \frac{d}{n-1}} \right)\gamma_{k-1} - 2\alpha_{k-1}L\sqrt{\frac{d}{n-1}} }{(1+\alpha_{k-1} L)^2 + \alpha_{k-1} L} \label{rec:gamk_mb}\\
    \eta_k 
        &= 2\alpha_{k-1} L +  \frac{\left[\left( 1- 2\alpha_{k-1} L \sqrt{1 - \frac{d}{n-1}} \right)\gamma_{k-1} - 2\alpha_{k-1}L\sqrt{\frac{d}{n-1}} \right] \left[ 1+\alpha_{k-1} L \right] \alpha_{k-1} L }{(1+\alpha_{k-1} L)^2 + \alpha_{k-1} L} \nonumber \\
        &\qquad + \left( 1- 2\alpha_{k-1} L \sqrt{1 - \frac{d}{n-1}} \right) \eta_{k-1}. \label{rec:etak_mb}
\end{align}

\paragraph{Part 2: Solving the recurrence under step schedule assumption.}
First, we would want to find out $\gamma_0, \eta_0$, i.e. the value of $\gamma, \eta$ at the step where we update the alignment vector $v_0$.
Given that have $v_0 = \frac{1}{m'} \sum_{i \in B'} \nabla f_i(x_0)$ for a mini-batch $B'$, where we choose $|B'| = m' > m$, we can obtain
\begin{align*}
    \mathbb{E} \left[ \langle v_0, \nabla f(x_0) \rangle \right]^2 
        &= \mathbb{E} \left[ \langle \hat{v}_0, \nabla f(x_0) \rangle \|v_0\|\right]^2 \\
        &\leq \mathbb{E} \left[ \langle \hat{v}_0, \nabla f(x_0) \rangle^2 \right] \mathbb{E} \left[ \|v_0\|^2 \right] \\
        &\leq \mathbb{E} \left[ \langle \hat{v}_0, \nabla f(x_0) \rangle^2 \right] \left( \sigma_{m'}^2 + \| \nabla f(x_0) \|^2 \right).
\end{align*}
On the other hand, since $v_0$ is an unbiased estimator of the gradient, we also have
\begin{align*}
    \mathbb{E} \left[ \langle v_0, \nabla f(x_0) \rangle \right]^2  = \| \nabla f(x_0) \|^4.
\end{align*}
Combining these 2, we get
\begin{equation*}
    \mathbb{E} \left[ \langle \hat{v}_0, \nabla f(x_0) \rangle^2 \right] \geq \frac{\| \nabla f(x_0) \|^4}{\sigma_{m'}^2 + \| \nabla f(x_0) \|^2} = \left(1 - \frac{\sigma_{m'}^2}{\sigma_{m'}^2 + \| \nabla f(x_0) \|^2} \right) \| \nabla f(x_0) \|^2,
\end{equation*}
giving us
\begin{align}
    \gamma_0 &= 1 - \frac{\sigma_{m'}^2}{\sigma_{m'}^2 + \| \nabla f(x_0) \|^2} \label{eqn:mb_gam0} \\
    \eta_0 &= 0. \label{eqn:mb_eta0}
\end{align}

Assuming $r_i < k < r_{i+1}$, where $r_i$ is the $i^{th}$ refresh step, we can use the recurence relation of $\gamma_k$ to obtain
\begin{align*}
    \gamma_k 
        &\geq \frac{\left( 1- 2\alpha_{k-1} L \right)\gamma_{k-1} - 2\alpha_{k-1}L\sqrt{\frac{d}{n-1}} }{(1+\alpha_{k-1} L)^2 + \alpha_{k-1} L} \\
        &\geq \gamma_{r_i} \prod_{i=r_i}^{k-1} a_i - \sum_{i=r_i}^{k-1} b_i \prod_{j=i+1}^{k-1} a_j,
\end{align*}
where the coefficients $a_j$ and $b_j$ can be simplified to
\begin{align*}
    a_j &= \frac{1- 2\alpha_{j-1} L}{(1+\alpha_{j-1} L)^2 + \alpha_{j-1} L} = \frac{1 - \frac{2}{\sqrt{j}}}{\left( 1 + \frac{1}{\sqrt{j}} \right)^2 + \frac{1}{\sqrt{j}}} = \frac{(\sqrt{j}-2)(\sqrt{j})}{j + 1 + 3\sqrt{j}}\\
    b_j &= \frac{2\alpha_{j-1}L\sqrt{\frac{d}{n-1}} }{(1+\alpha_{j-1} L)^2 + \alpha_{j-1} L} = \frac{ 2\sqrt{j} }{j + 1 + 3\sqrt{j}} \sqrt{\frac{d}{n-1}},
\end{align*}
under the assumption that the step-size is $\alpha_k = \frac{1}{L\sqrt{k+1}}$.
Equivalently, we can obtain that the correlation right before the next update has to satisfy the equation
\begin{equation}\label{eqn:rec_nec}
    \gamma_{r_{i+1} - 1} 
        \geq \gamma_{r_i} \prod_{k=r_i}^{r_{i+1}-2} a_k - \sum_{k=r_i}^{r_{i+1}-2} b_k \prod_{j=k+1}^{r_{i+1}-2} a_j > \delta.
\end{equation}
to ensure that it is still at least $\delta$. To make calculations slightly easier, we will simplify the expressions for $a_j$ and $b_j$ to get
\begin{align*}
    a_j 
    &= \frac{(\sqrt{j}-2)(\sqrt{j})}{j + 1 + 3\sqrt{j}}
    = \frac{j-2\sqrt{j}}{j + 1 + 3\sqrt{j}}
    = 1 - \frac{5\sqrt{j} + 1}{j + 1 + 3\sqrt{j}} \\
    &\geq 1 - \frac{5\sqrt{j}}{j + 3\sqrt{j}} 
    = 1 - \frac{5}{3 + \sqrt{j}}, \\
    b_j &= \frac{ 2\sqrt{j} }{j + 1 + 3\sqrt{j}} \sqrt{\frac{d}{n-1}} \leq \frac{2}{\sqrt{j}}\sqrt{\frac{d}{n-1}}.
\end{align*}
With these simplified expressions, we can simplify the expression in \ref{eqn:rec_nec} to get
\begin{align*}
    \gamma_{r_{i+1}} 
        &\geq \gamma_{r_i} \prod_{k=r_i+1}^{r_{i+1}-2} a_k - \sum_{k=r_i}^{r_{i+1}-2} b_k \prod_{j=k+1}^{r_{i+1}-2} a_j \\
        &\geq \gamma_{r_i} a_{r_i+1}^{r_{i+1} - r_i - 1} - \sum_{k=r_i}^{r_{i+1}-2} b_k \\
        &\geq \gamma_{r_i} \left( 1 - \frac{5}{3 + \sqrt{r_i+1}} \right)^{r_{i+1} - r_i - 1} - 2\sqrt{\frac{d}{n-1}} \sum_{k=r_i}^{r_{i+1}-2}  \frac{ 1 }{\sqrt{j }}.
\end{align*}
For the second term, we use the approximation $\sum_{i=m}^{n} \frac{1}{\sqrt{i}} \leq 2(\sqrt{n} - \sqrt{m}) + m^{-1/2}$ to get that $\sum_{k=r_i}^{r_{i+1}-2}  \frac{ 1 }{\sqrt{j }} \leq 2 (\sqrt{r_{i+1}} - \sqrt{r_i}) + (r_i)^{-1/2}$ for sufficiently large $r_i$. Letting $r_{i+1} - r_i = \epsilon_i$ and assuming $\epsilon_i \ll r_i$, we can expand the square root by Taylor's formula to get the approximation
\begin{equation}
    \sqrt{r_i + \epsilon_i} - \sqrt{r_i} = \sqrt{r_i} \left( \sqrt{1 + \frac{\epsilon_i}{r_i}} - 1 \right) = \sqrt{r_i} \left( \frac{\epsilon_i}{2r_i} + o\left(\frac{\epsilon_i}{r_i}\right) \right).
\end{equation}
Consequently, we obtain the bound
\begin{align*}
    \gamma_{r_{i+1}} 
        &\geq \gamma_{r_i} \left( 1 - \frac{5}{3 + \sqrt{r_i+1}} \right)^{r_{i+1} - r_i} - 2\sqrt{\frac{d}{n-1}} \left( \frac{r_{i+1} - r_i}{\sqrt{r_i}}  + o\left(\frac{\sqrt{r_{i+1}} - \sqrt{r_i}}{\sqrt{r_i}}\right)\right) \\
        &\geq \gamma_{r_i} \left( 1 - \frac{5}{3 + \sqrt{r_i}} \right)^{r_{i+1} - r_i} - 4\sqrt{\frac{d}{n-1}} \left( \frac{r_{i+1} - r_i}{\sqrt{r_i}}\right)
\end{align*}
for sufficiently large $r_i$. Using the assumption that $r_{i+1} - r_i = i^\beta$, we can find that the order of $r_i$ is approximately bounded by
\begin{equation}\label{eqn:ri_lorder}
    r_i = \sum_{k=1}^i (r_k - r_{k-1}) = \sum_{k=1}^i k^\beta \geq \int_{0}^i k^\beta dk = \frac{i^{\beta + 1} }{\beta + 1}.
\end{equation}
and
\begin{equation*}
    \sqrt{r_i} \geq \sqrt{\frac{i^{\beta+1}}{\beta + 1}} \geq \frac{i^{(\beta + 1)/2}}{\sqrt{\beta + 1}}
\end{equation*}
for sufficiently large $i$. We can now simplify the bound to get
\begin{align*}
    \gamma_{r_{i+1}} 
        &\geq \gamma_{r_i} \left( 1 - \frac{5\sqrt{\beta + 1}}{i^{(\beta + 1)/2}} \right)^{i^\beta} - 4\sqrt{\frac{d}{n-1}} \left( \sqrt{\beta + 1} \frac{i^\beta}{i^{(\beta + 1)/2}} \right) \\
        &= \gamma_{r_i} \left( 1 - \frac{5\sqrt{\beta + 1}}{i^{(\beta + 1)/2}} \right)^{i^\beta} - 4\sqrt{\frac{d(\beta + 1)}{n-1}} \left( i^{(\beta - 1)/2} \right).
\end{align*}
Setting the RHS to be $> \delta$ and simplifying the inequality, we obtain
\begin{align*}
    \gamma_{r_i} \left( 1 - \frac{5\sqrt{\beta+1}}{i^{(\beta + 1)/2}} \right)^{i^\beta} 
        &> \delta +  4\sqrt{\frac{d(\beta+1)}{n-1}} i^{(\beta-1)/2}.
\end{align*}
Taking logarithm on both sides and moving the terms around, we get
\begin{equation*}
    i^{\beta}
        < \frac{\log\left( \frac{\delta +  4\sqrt{\frac{d(\beta+1)}{n-1}} i^{(\beta-1)/2}}{\gamma_{r_i}} \right)}{\log \left( 1 - \frac{5\sqrt{\beta+1}}{i^{(\beta + 1)/2}} \right)}
        = \frac{\log\left( \frac{\gamma_{r_i}}{\delta +  4\sqrt{\frac{d(\beta+1)}{n-1}} i^{(\beta-1)/2}} \right)}{\log \left( 1 + \frac{5\sqrt{\beta+1}}{i^{(\beta + 1)/2}-5\sqrt{\beta+1}} \right)}.
\end{equation*}
For sufficiently large $i$, we can bound the numerator by $\log(\gamma_{r_i}/2\delta)$ while the denominator can be expanded using Taylor's expansion to get
\begin{align*}
    i^\beta
        &\leq \frac{\log\left( \frac{\gamma_{r_i}}{2\delta} \right)}{ \frac{5\sqrt{\beta+1}}{i^{(\beta + 1)/2}-5\sqrt{\beta+1}} (1 + o(1))} \\
        &\leq \frac{\log\left( \frac{\gamma_{r_i}}{2\delta} \right)}{ 5\sqrt{\beta+1}(1 + o(1))} i^{(\beta + 1)/2}\\
        &\lesssim \frac{\log\left( \frac{\gamma_{r_i}}{2\delta} \right)}{ 5\sqrt{\beta+1}} i^{(\beta + 1)/2}.
\end{align*}
Consequently, we have an upper-bound on the equation satisfied by $\beta$:
\begin{equation}
    i^{(\beta-1)/2} 
        \lesssim  \frac{1}{5} \log(\gamma_{r_i}/2\delta),
\end{equation}
where $\gamma_{r_i} = 1 - \frac{\sigma_{m'}^2}{\sigma_{m'}^2 + \| \nabla f(x_{r_i}) \|^2}$. For sufficiently large $i$ and large $\|\nabla f(x_{r_i})\|$ (where we are still far from convergence), this inequality is eventually satisfied. We simply translate the starting iteration with $k_0$ satisfying the inequality, with $\alpha_k = \frac{1}{\sqrt{k+k_0}}$, where 
\begin{equation}\label{eqn:mb_k0}
    k_0 = r_{i^\ast}. \qquad i^\ast \geq \left(\frac{5}{\log(1/2\delta)}\right)^{2/(1-\beta)}.
\end{equation}

\subsection{Proof for Iteration Complexity in Minibatch Setting (Theorem \ref{thm:iter_mb_general})}
The proof for the convergence is similar to the vanilla case, just that now we have the mini-batch gradient and the variance term in the lower-bound for the correlation. By L-smoothness, we have
\begin{align*}
    f(x_{k}) - f(x_{k-1}) 
        &= -\alpha_{k-1} \langle \nabla f(x_{k-1}), P_{k-1} P_{k-1}^\top \nabla g_{k-1} \rangle + \frac{\alpha_{k-1}^2 L}{2} \|P_{k-1} P_{k-1}^\top \nabla g_{k-1}\|^2.
\end{align*}
Using the same definitions for $\mathcal{F}_k$ and $\mathcal{G}_k$, we have
\begin{align*}
    \mathbb{E} [f(x_{k}) - f(x_{k-1}) \mid \mathcal{G}_{k-1}]
        &= -\alpha_{k-1} \left\langle \nabla f(x_{k-1}), \left[ \left(1-\frac{d}{n-1}\right)\hat{v}_{k-1} \hat{v}_{k-1}^\top + \frac{d}{n-1} I \right] \nabla g_{k-1} \right\rangle \\
        &\qquad + \frac{\alpha_{k-1}^2 L}{2} \left\langle \nabla g_{k-1}, \left[ \left(1-\frac{d}{n-1} \right)\hat{v}_{k-1} \hat{v}_{k-1}^\top + \frac{d}{n-1} I\right] \nabla g_{k-1} \right\rangle \\
        &= -\alpha_{k-1} \left(1-\frac{d}{n-1}\right) \langle \nabla f(x_{k-1}), \hat{v}_{k-1} \rangle \langle \hat{v}_{k-1}, \nabla g_{k-1} \rangle - \alpha_{k-1} \frac{d}{n-1} \langle \nabla f(x_{k-1}), \nabla g_{k-1} \rangle \\
        &\qquad + \frac{\alpha_{k-1}^2 L}{2} \left(1-\frac{d}{n-1}\right) \langle \nabla g_{k-1}, \hat{v}_{k-1} \rangle^2 + \frac{\alpha_{k-1}^2 L}{2}\frac{d}{n-1} \|\nabla g_{k-1}\|^2 \\
        &\leq -\alpha_{k-1} \left(1-\frac{d}{n-1}\right) \langle \nabla f(x_{k-1}), \hat{v}_{k-1} \rangle \langle \hat{v}_{k-1}, \nabla g_{k-1} \rangle \\
        &\qquad - \alpha_{k-1} \frac{d}{n-1} \langle \nabla f(x_{k-1}), \nabla g_{k-1} \rangle  + \frac{\alpha_{k-1}^2 L}{2}\|\nabla g_{k-1}\|^2.
\end{align*}
Taking expectation with respect to $\mathcal{F}_{k-1}$ and using the tower property of conditional expectation, we obtain
\begin{align*}
    \mathbb{E} \left[ f(x_{k}) - f(x_{k-1}) \mid \mathcal{F}_{k-1} \right]
        &= \mathbb{E} \left[ \mathbb{E} \left[ f(x_{k}) - f(x_{k-1}) \mid \mathcal{G}_{k-1} \right] \mid \mathcal{F}_{k-1} \right] \\
        &\leq -\alpha_{k-1} \left(1-\frac{d}{n-1}\right) \langle \nabla f(x_{k-1}), \hat{v}_{k-1} \rangle^2 - \alpha_{k-1} \frac{d}{n-1} \|\nabla f(x_{k-1}) \|^2 \\
        &\qquad+ \frac{\alpha_{k-1}^2 L}{2} \mathbb{E} \left[ \|\nabla g_{k-1}\|^2 \mid \mathcal{F}_{k-1} \right],
\end{align*}
since $\nabla g_{k-1}$ is an unbiased estimator of the gradient. i.e. $\mathbb{E}[\nabla g_{k-1}] = \nabla f(x_{k-1})$. Equivalently, we have
\begin{align}
    \mathbb{E} \left[ f(x_{k}) - f(x_{k-1}) \right]
        &\leq -\alpha_{k-1} \left(1-\frac{d}{n-1}\right) \mathbb{E} \left[ \langle \nabla f(x_{k-1}), \hat{v}_{k-1} \rangle^2 \right] - \alpha_{k-1} \frac{d}{n-1} \mathbb{E} \left[ \|\nabla f(x_{k-1}) \|^2 \right] \nonumber\\ 
        &\qquad+ \frac{\alpha_{k-1}^2 L}{2} \mathbb{E} \left[ \|\nabla g_{k-1}\|^2 \right] \label{eqn:mb_general_iterationdecay}.
\end{align}
Using the assumption on the variance of the mini-batch gradient and the lower bound of the correlation, we obtain
\begin{align*}
    \mathbb{E} \left[ f(x_{k}) - f(x_{k-1}) \right]
        &\leq -\alpha_{k-1} \left(1-\frac{d}{n-1}\right) \left(\gamma_{k-1}\mathbb{E} \left[ \|\gfkk\|^2 \right] - \eta_{k-1} \sigma_m^2 \right) \\
        &\qquad - \alpha_{k-1} \frac{d}{n-1} \mathbb{E} \left[ \|\nabla f(x_{k-1}) \|^2 \right] + \frac{\alpha_{k-1}^2 L}{2} \left( \sigma_m^2 + \mathbb{E} \left[ \|\gfkk\|^2 \right] \right) \\
        &= -\left( \alpha_{k-1} \left(1 - \frac{d}{n-1} \right) \gamma_{k-1} + \alpha_{k-1} \frac{d}{n-1} - \frac{\alpha_{k-1}^2 L}{2} \right) \mathbb{E} \left[ \|\nabla f(x_{k-1}) \|^2 \right] \\
        &\qquad + \left( \alpha_{k-1} \left(1 - \frac{d}{n-1} \right) \eta_{k-1} + \frac{\alpha_{k-1}^2 L}{2} \right) \sigma_m^2.
\end{align*}
By assumption, we have $\gamma_{k-1} > \delta$, allowing us to simplify the above to
\begin{align*}
    \mathbb{E} \left[ f(x_{k}) - f(x_{k-1}) \right]
        &\leq -\left( \alpha_{k-1} \left(1 - \frac{d}{n-1} \right) \delta + \alpha_{k-1} \frac{d}{n-1} - \frac{\alpha_{k-1}^2 L}{2} \right) \mathbb{E} \left[ \|\nabla f(x_{k-1}) \|^2 \right] \\
        &\qquad + \left( \alpha_{k-1} \left(1 - \frac{d}{n-1} \right) \eta_{k-1} + \frac{\alpha_{k-1}^2 L}{2} \right) \sigma_m^2.
\end{align*}
for $\alpha_{k-1}$ such that $\alpha_{k-1} \left(1 - \frac{d}{n-1} \right) \delta + \alpha_{k-1} \frac{d}{n-1} - \frac{\alpha_{k-1}^2 L}{2} > 0$, which can be satisfied for $\alpha_{k-1} = \frac{1}{L\sqrt{k}} \in \left(0, \frac{2}{L}\left(\left(1-\frac{d}{n-1}\right)\delta + \frac{d}{n-1}\right)\right)$. Taking sum up to $N$ steps, we have
\begin{align*}
    \mathbb{E} \left[ f(x_{N}) - f(x_{0}) \right]
        &\leq -\sum_{k=1}^N \left( \alpha_{k-1} \left(1 - \frac{d}{n-1} \right) \delta + \alpha_{k-1} \frac{d}{n-1} - \frac{\alpha_{k-1}^2 L}{2} \right) \mathbb{E} \left[ \|\nabla f(x_{k-1}) \|^2 \right] \\
        &\qquad + \sum_{k=1}^N\left( \alpha_{k-1} \left(1 - \frac{d}{n-1} \right) \eta_{k-1} + \frac{\alpha_{k-1}^2 L}{2} \right) \sigma_m^2.
\end{align*}
Rewriting the inequality, we get
\begin{align*}
    \mathbb{E} \left[ \|\nabla f(x_{\tau}) \|^2 \right] \sum_{k=1}^N \left( \alpha_{k-1} \left(1 - \frac{d}{n-1} \right) \delta + \alpha_{k-1} \frac{d}{n-1} - \frac{\alpha_{k-1}^2 L}{2} \right) \\
    \leq \sum_{k=1}^N\left( \alpha_{k-1} \left(1 - \frac{d}{n-1} \right) \eta_{k-1} + \frac{\alpha_{k-1}^2 L}{2} \right) \sigma_m^2 + (f(x_0) - f*),
\end{align*}
where $\tau = \arg \min_k \mathbb{E} \left[ \| \nabla f(x_k) \| \right]$ is the index that gives the minimum gradient. All that remains is to check the order of $\sum_{k=1}^N \alpha_{k-1} \eta_{k-1}$. Similar to the lower bound for $r_s$ which we obtained previously, we can use the fact that $x^{\beta}$ is an increasing function to obtain an upper bound for $r_s$, giving us
\begin{equation}\label{eqn:ri_uorder}
    r_s \leq \frac{(s+1)^{\beta + 1} - 1}{\beta + 1}.
\end{equation}
Setting the upper-bound to be equal to $N$, we find that $s$ is equal to
\begin{equation}\label{eqn:mb_num_ref}
    \frac{(s+1)^{\beta + 1} - 1}{\beta + 1} = N \implies s = [(\beta+1)N + 1]^{\frac{1}{\beta+1}} -1.
\end{equation}
For $\eta_k$ where $r_i < k < r_{i+1}$, we have from Equation \eqref{rec:etak_mb}
\begin{align*}
    \eta_k &\leq 3\alpha_{k-1} L  + \eta_{k-1},
\end{align*}
which gives us the upper bound
\begin{align*}
    \eta_k \leq 3L \sum_{j=r_i}^{k-1} \alpha_j + \eta_{r_i}.
\end{align*}
Since we have $\eta_{r_i}=0$ from Equation \eqref{eqn:mb_eta0}, we find that $\eta_k \leq 3 \sum_{j={r_i}}^{k-1} (j+1)^{-1/2} \leq 6(\sqrt{k} - \sqrt{r_i+1}) + 3r_i^{-1/2}$.
Consequently, we have
\begin{equation*}
    \sum_{k=r_i}^{r_{i+1}} \eta_k 
        \leq 6 \left[\sum_{k=r_i+1}^{r_{i+1}} (\sqrt{k} - \sqrt{r_i+1})\right] + 3 \left(\frac{r_{i+1}-r_i}{\sqrt{r_i}}\right).
\end{equation*}
For the first sum, we have
\begin{equation*}
    \sum_{k=r_i+1}^{r_{i+1}} \left( \sqrt{k} - \sqrt{r_i+1} \right) \leq (r_{i+1}-r_i) \left( \sqrt{r_{i+1}} - \sqrt{r_i} \right) = (r_{i+1}-r_i) \left( \sqrt{r_i + \epsilon_i} - \sqrt{r_i} \right).
\end{equation*}
Using the same approximation for the second component as before, $\sqrt{r_i + \epsilon_i} - \sqrt{r_i} = \sqrt{r_i} \left( \frac{\epsilon_i}{2r_i} + o\left(\frac{\epsilon_i}{r_i}\right) \right) \leq \frac{\epsilon_i}{\sqrt{r_i}}$ for sufficiently large $i$, we have
\begin{equation*}
    \sum_{k=r_i}^{r_{i+1}} \eta_k \leq 6\sum_{k=r_i+1}^{r_{i+1}} (\sqrt{k-1} - \sqrt{r_i}) + 3\frac{r_{i+1}-r_i}{\sqrt{r_i}} \leq \frac{6\epsilon_i^2 + 3\epsilon_i}{\sqrt{r_i}}.
\end{equation*}
Consequently, the variance contribution is approximately
\begin{equation}\label{eqn:mb_var_sum_2}
    \sum_{k=k_0}^N \eta_k \alpha_k 
        = \sum_{i=i^\ast}^s \sum_{j=r_i}^{r_{i+1}-1}\alpha_{j} \eta_j 
        \leq \sum_{i=i^\ast}^s \alpha_{r_i}\sum_{j=r_i}^{r_{i+1}-1} \eta_j 
        \leq \sum_{i=i^\ast}^s \left(\frac{1}{L\sqrt{r_i+1}}\right) \frac{6\epsilon_i^2 + 3\epsilon_i}{\sqrt{r_i}} 
        \leq \frac{1}{L} \sum_{i=i^\ast}^s \frac{6\epsilon_i^2 + 3\epsilon_i}{r_i}.
\end{equation}
Once again, using the fact that $\epsilon_i = i^\beta$ and a lower bound $r_i \geq \frac{i^{\beta+1}}{\beta+1}$, the above has an upper bound of
\begin{equation*}
     \sum_{k=k_0}^N \eta_k \alpha_k \leq \frac{\beta+1}{L} \sum_{i=1}^s \frac{6i^{2\beta} + 3i^{\beta}}{i^{\beta+1}} = \frac{\beta+1}{L} \sum_{i=i^\ast}^{s} \left(6i^{\beta - 1} + 3i^{-1}\right).
\end{equation*}
The individual sums in the equation can be similarly bounded by integrals, giving us
\begin{align*}
    \sum_{i=i^\ast}^s i^{\beta-1} &\leq \frac{s^{\beta} - (i^\ast - 1)^\beta}{\beta} \\
    \sum_{i=i^\ast}^s i^{-1} &\leq \ln(s) - \ln(i^\ast - 1).
\end{align*}
Following that, we have 
\begin{align*}
    \sum_{k=1}^N \alpha_k \eta_k 
        &\leq \frac{\beta+1}{L} \left( \frac{6}{\beta} \left(s^{\beta} - (i^\ast - 1)^\beta\right) + 3\ln \frac{s}{i^\ast - 1} \right) \\
        &\lesssim \frac{\beta+1}{L} \left( \frac{6}{\beta}(\beta+1)^{\frac{\beta}{1+\beta}} \left( N^{\frac{\beta}{1+\beta}} - k_0^{\frac{\beta}{1 + \beta}}\right) + \frac{3}{\beta+1} \ln \left( \frac{N}{k_0} \right) \right)\\
        &= \frac{1}{L} \left( 6\beta^{-1}(\beta+1)^{\frac{\beta-1}{1+\beta}} \left( N^{\frac{\beta}{1+\beta}} - k_0^{\frac{\beta}{1 + \beta}}\right) + 3\ln((\beta+1)N/k_0)\right).
\end{align*}
where we used the approximation $s \approx \left[(\beta+1)N\right]^{\frac{1}{\beta+1}}$ derived from Equation \eqref{eqn:mb_num_ref}. For the factor with the gradient norm, we have
\begin{align*}
    \sum_{k=k_0}^N \left( \alpha_{k-1} \left(1 - \frac{d}{n-1} \right) \delta + \alpha_{k-1} \frac{d}{n-1} - \frac{\alpha_{k-1}^2 L}{2} \right) 
    &= \sum_{k=k_0}^N \frac{ \left(1 - \frac{d}{n-1} \right) \delta + \frac{d}{n-1} - \frac{1}{2\sqrt{k}} }{L\sqrt{k}} \\
    &\approx \frac{1}{L} (\sqrt{N} - \sqrt{k_0}).
\end{align*}
This means we get
\begin{align*}
    \mathbb{E} \left[ \|\nabla f(x_{\tau}) \|^2 \right] (\sqrt{N} - \sqrt{k_0}) 
        &\lesssim \beta^{-1}(\beta+1)^{\frac{\beta-1}{1+\beta}} \left( N^{\frac{\beta}{1+\beta}} - k_0^{\frac{\beta}{1 + \beta}}\right) \sigma_m^2 + (f(x_0) - f*) \\
    \mathbb{E} \left[ \|\nabla f(x_{\tau}) \|^2 \right] 
        &\lesssim \frac{\beta^{-1}(\beta+1)^{\frac{\beta-1}{1+\beta}} \left( N^{\frac{\beta}{1+\beta}} - k_0^{\frac{\beta}{1 + \beta}}\right) \sigma_m^2 + (f(x_0) - f*)}{\sqrt{N} - \sqrt{k_0}},
\end{align*}
where the leading term is $\frac{N^{\frac{\beta}{1+\beta}} - k_0^{\frac{\beta}{1 + \beta}}}{\sqrt{N} - \sqrt{k_0}}$. It suffices that this term is $\leq \epsilon^2$ to ensure that the LHS is $\leq \epsilon^2$, giving us the condition
\begin{equation}
    \frac{N^{\frac{\beta}{1+\beta}} - k_0^{\frac{\beta}{1 + \beta}}}{\sqrt{N} - \sqrt{k_0}} \leq \epsilon^2 \iff \epsilon^{-2} \left( 1- \left(\frac{k_0}{N}\right)^{\frac{\beta}{1 + \beta}} \right) + \frac{\sqrt{k_0}}{N^{\frac{\beta}{1+\beta}}} \leq N^{\frac{1-\beta}{2(1+\beta)}},
\end{equation}
which simplifies to
\begin{equation}\label{eqn:mb_iteration}
    N \geq \left[ \epsilon^{-2} \left( 1- \left(\frac{k_0}{N}\right)^{\frac{\beta}{1 + \beta}} \right) + \frac{\sqrt{k_0}}{N^{\frac{\beta}{1+\beta}}} \right]^{\frac{2(1+\beta)}{1-\beta}} \gtrsim O\left(\epsilon^{-4\frac{1+\beta}{1-\beta}}\right),
\end{equation}
where we hide some constant and smaller order terms in $\gtrsim$ notation.

\subsection{Proof for Computational Cost in Minibatch Setting (Theorem \ref{thm:mb_cost})}
From Equation \eqref{eqn:mb_num_ref}, we have the number of times we need to update the alignment vector to be $s = N^{\frac{1}{\beta + 1}}$ for a choice of $\beta \in (0,1)$. Suppose the cost of computing an alignment vector is $O(\nu)$ and the cost of a directional derivative is $O(\xi)$. During the update step, the cost is $O(\nu + nd)$ while during a normal iteration it is $O(nd + md\xi)$. This gives us
\begin{align*}
    N \times (nd + md\xi) + s \times \nu 
        &=  N \times (nd + md\xi) + N^{\frac{1}{\beta + 1}} \times \nu \\
        &=  N^{\frac{1}{\beta+1}} \left( N^{\frac{\beta}{\beta+1}} (nd + md\xi) + \nu \right)\\
        &= O\left( \epsilon^{\frac{-4}{1-\beta}} \left( \epsilon^{\frac{-4\beta}{1-\beta}} (nd + md\xi) + \nu \right)\right).
\end{align*}

\section{Proof for the Local Regime (Theorem \ref{thm:local-norefresh})}
This section is dedicated to show, in detail, the analysis for the persistence of memory in the local regime. We will start off with defining the local regime setup, then show some lemmas which hold in the local regime. Finally, the 2-stage proof will be presented in the last part.

\subsection{Setting and Algorithm}
Let $f:\R^n\to\R$ be $C^3$ and let $x^\ast$ be its unique minimizer. Define
\begin{equation}\label{eq:opt}
    H = \nabla^2 f(x^\ast), \qquad \nabla f(x^\ast) = 0.
\end{equation}
By assumption of $L$-Lipschtiz and $\mu$-strongly convex, we have $\forall x \in \mathbb{R}^n$,
\begin{equation}\label{eq:sc-smooth}
    \mu I \preceq \nabla^2 f(x)\preceq L I.
\end{equation}
We also assume that the Hessian is locally Lipschitz near the unique optimal minimiser.
\begin{assump}[Locally Lipschitz Hessian]
    Assume $\exists L_H, r_0 > 0$ such that for all $x, y$ with $\norm{x-x^\ast},\norm{y-x^\ast}\le r_0$,
    \begin{equation}\label{eq:hess-lip}
        \norm{\nabla^2 f(x)-\nabla^2 f(y)}\le L_H\norm{x-y}.
    \end{equation}
\end{assump}

We define the notations below.
\begin{equation*}
    e_k:=x_k-x^\ast,\qquad g_k:=\nabla f(x_k),\qquad H:=\nabla^2 f(x^\ast).
\end{equation*}
Assume $H$ has distinct eigenvalues
\begin{equation*}
    0<\lambda_1<\lambda_2<\cdots<\lambda_n,
\end{equation*}
with orthonormal eigenvectors $(u_i)_{i=1}^n$, so $Hu_i=\lambda_i u_i$. Define
\begin{equation*}
    M:=I-\alpha H,\qquad \mu_i:=1-\alpha\lambda_i,
\end{equation*}
so that
\begin{equation}\label{eq:muorder}
    0<\mu_n<\cdots<\mu_2<\mu_1<1.
\end{equation}
Define the projection error
\begin{equation}\label{eq:Edef}
    E_k:=(I-P_kP_k^\top)g_k.
\end{equation}
This term will determine how far our iterates diverge from the standard gradient descent. The main bulk of the proof is dedicated to controlling the deviations induced by this term.

For a fixed tolerance $\varepsilon>0$, let
\begin{equation*}
    K_\varepsilon:=\min\{k\ge 0:\norm{g_k}\le \varepsilon\}.
\end{equation*}
All the events below will be enforced only for $k<K_\varepsilon$; this allows a union bound over finitely many steps.

\subsection{Build-up to Main Theorem}
We will introduce various lemmas that allow us to control the errors of the algorithm through Taylor's Expansion. We first show that when sufficiently close to the optimal point, the gradient is close to $He$, where $e = x - x^\ast$.

\begin{lem}[Local gradient expansion]\label{lem:taylor}
    There exist constants $C_1>0$ such that for all $e$ with $\norm{e}\le r_0$,
    \begin{equation}\label{eq:taylor}
        \nabla f(x^\ast+e) = He + r(e),
    \end{equation}
    where the remainder satisfies
    \begin{equation}\label{eq:rbound}
        \norm{r(e)}\le C_1\norm{e}^2.
    \end{equation}
    One may take $C_1= L_H / 2$.
\end{lem}
\begin{proof}
    For $e$ such that $\norm{e} < r_0$, the fundamental theorem of calculus gives
    \begin{equation*}
        \nabla f(x^\ast+e)-\nabla f(x^\ast)=\int_{0}^{1}\nabla^2 f(x^\ast+te)\,e\,dt.
    \end{equation*}
    Since $\nabla f(x^\ast)=0$, add and subtract $He$ on both sides to get
    \begin{equation*}
        \nabla f(x^\ast+e)=He+\int_{0}^{1}\big(\nabla^2 f(x^\ast+te)-H\big)e\,dt.
    \end{equation*}
    Define $r(e)$ as the integral term. By \eqref{eq:hess-lip}, we have
    \begin{align*}
        \norm{r(e)}
            &\le \int_0^1 \norm{\nabla^2 f(x^\ast+te)-H}\,\norm{e}\,dt \label{eq:taylor-proof3}\\
            &\le \int_0^1 L_H\,t\,\norm{e}\,\norm{e}\,dt \\
            &=\frac{L_H}{2}\norm{e}^2.
    \end{align*}
\end{proof}

 The lemma above allows us to characterise the relationship between consecutive gradients in terms of an iteration with some additional error terms. 
 
 Denote
 \[
    g_k = \nabla f(x_k),
 \]
 then the follow gives us the main recursion satisfied by subsequent gradients.
\begin{lem}[Modified gradient recursion]\label{lem:gradrec}
    Assume $\norm{e_k}\le r_0$ and $\norm{e_{k+1}}\le r_0$. Then there exists $C_2 > 0$ such that
    \begin{equation}\label{eq:gradrec}
        g_{k+1} = M g_k+\alpha H E_k + r_k,
    \end{equation}
    where 
    \begin{equation}
        r_k = r(e_{k+1})-r(e_k)
    \end{equation}
    satisfies $\norm{r_k} \leq C_2 \norm{g_k}^2$. One valid choice for $C_2 = 5C_1/\mu^2$.
\end{lem}

\begin{proof}
    By Lemma~\ref{lem:taylor}, at $x_k=x^\ast + e_k$, we find that
    \begin{equation*}
        g_k=H e_k+r(e_k).
    \end{equation*}
    From the update \eqref{eqn:update_step} and $E_k=(I-P_kP_k^\top)g_k$,
    \begin{align*}
        e_{k+1}
            &= e_k + (x_{k+1} - x_k) \\
            &= e_k - \alpha P_k P_k^\top g_k \\
            &= e_k - \alpha g_k + \alpha E_k. 
    \end{align*}
    Substituting the first equation into the above equation, we get
    \begin{equation}\label{eq:ek3}
        e_{k+1} = (I-\alpha H)e_k - \alpha r(e_k) + \alpha E_k.
    \end{equation}
    Apply Lemma~\ref{lem:taylor} at $x_{k+1}=x^\ast+e_{k+1}$
    \begin{equation*}
        g_{k+1}=H e_{k+1}+r(e_{k+1}).
    \end{equation*}
    Substitute \eqref{eq:ek3} into the RHS to get
    \begin{equation}\label{eq:gk1-sub}
        g_{k+1}=H(I-\alpha H)e_k-\alpha H r(e_k)+\alpha H E_k+r(e_{k+1}).
    \end{equation}
    On the other hand,
    \begin{align*}
        Mg_k
            &= (I-\alpha H)(H e_k+r(e_k)) \\
            &= (I-\alpha H)H e_k+(I-\alpha H)r(e_k) \\
            &= H(I-\alpha H) e_k+(I-\alpha H)r(e_k),
    \end{align*}
    since $H(I-\alpha H)=(I-\alpha H)H$. Then, combining both equations, we get
    \begin{align*}
        g_{k+1}
            &= (Mg_k - r(e_k))+\alpha H E_k+r(e_{k+1}) \\
            &= Mg_k + \alpha H E_k + r_k.
    \end{align*}
    For the size of $r_k$, by Lemma~\ref{lem:taylor} and the triangle inequality,
    \begin{equation*}
        \norm{r_k} \le \norm{r(e_{k+1})}+\norm{r(e_k)}
            \le C_1\norm{e_{k+1}}^2+C_1\norm{e_k}^2.
    \end{equation*}
    From the update $e_{k+1}=e_k-\alpha P_k P_k^\top g_k$ and $\|P_k P_k^\top\|_{\op}=1$,
    \begin{equation*}
        \norm{e_{k+1}}\le \norm{e_k}+\alpha\norm{g_k}.
    \end{equation*}
    By strong convexity and $\nabla f(x^\ast)=0$, we have $\mu\|e_k\|\le \|g_k\|$, hence
    \begin{equation*}
        \norm{e_k}\le \frac{1}{\mu}\norm{g_k}.
    \end{equation*}
    Also $\|g_k\|\le L\|e_k\|$ by Lipschitz continuity of the gradient, so we get
    $\|e_{k+1}\|\le (1+\alpha L)\|e_k\|\le 2\|e_k\|$ as $\alpha < 1/L$.
    Therefore
    \begin{equation*}
    \norm{r_k} \le C_1(4\|e_k\|^2+\|e_k\|^2)=5C_1\|e_k\|^2
    \le \frac{5C_1}{\mu^2}\|g_k\|^2.
    \end{equation*}
\end{proof}

Before we begin, we will first pay off the debt we had above by proving Proposition \ref{prop:sum_angle}.
\begin{proof}[Proof for Proposition \ref{prop:sum_angle}]
Let
\[
\alpha = \langle u,v\rangle, \quad \beta = \langle v,w\rangle, \quad x = \langle u,w\rangle.
\]
Since $u,v,w$ are unit vectors, the Gram matrix
\[
G = \begin{pmatrix}
    u^\top \\ v^\top \\ w^\top
\end{pmatrix} \begin{pmatrix}
    u & v & w
\end{pmatrix}
=
\begin{pmatrix}
1 & \alpha & x \\
\alpha & 1 & \beta \\
x & \beta & 1
\end{pmatrix}
\]
is positive semidefinite. Hence $\det(G) \ge 0$, which gives
\[
1 + 2\alpha \beta x - \alpha^2 - \beta^2 - x^2 \ge 0.
\]
Rewriting,
\[
x^2 - 2\alpha \beta x + (\alpha^2 + \beta^2 - 1) \le 0.
\]
Viewing this as a quadratic inequality in $x$, the roots are
\[
x = \alpha \beta \pm \sqrt{(1-\alpha^2)(1-\beta^2)},
\]
which means
\[
    |x| \geq |\alpha \beta| - \sqrt{(1-\alpha^2)(1-\beta^2)}.
\]
Using $\alpha^2 \ge \delta_1$ and $\beta^2 \ge {\delta_2}$ (which minimizes the lower bound), we get
\[
    |x| \ge \sqrt{\delta_1\delta_2} - \sqrt{(1-\delta_1)(1-\delta_2)}.
\]
If $\delta_1 + \delta_2 \ge 1$, then
\[
\sqrt{\delta_1\delta_2} \ge \sqrt{(1-\delta_1)(1-\delta_2)},
\]
so the right-hand side is non-negative and we may square both sides to obtain the desired result.
\end{proof}

Next, for the convergence to hold true, we require that the gradient decays exponentially, which we show below.
\begin{lem}[Exponential Decay of Gradient with High Probability]\label{lem:exp_decay}
    The gradient of the iterates decrease at a rate of
    \begin{equation}
        \norm*{g_k} \leq C\rho_{decay}^{k}
    \end{equation}
    for some constant $C >0$ with probability at least  $1 - e^{- \frac{1}{2} \sigma^2k(1 - p_{decay}(\tau, \delta, d))}$, where 
    \begin{align}
        p_{decay}(\tau, \delta, d) &= 2\exp{(-cd\tau^2)} + p_v(\delta) \\
        \rho_{decay} &= \left[1 - 2\mu\alpha \left( 1 - \frac{\alpha L}{2} \right) \left(\delta+\frac{d}{n-1}(1+\tau)\right) \right]^{(1-\sigma)(1 - p_{decay}(\tau, \delta, d))},
    \end{align}
    for some $\sigma \in (0,1)$ and $\alpha \in (0, 2/L)$.
\end{lem}

\begin{proof}
    From the same workings for the convergence, we have that conditioned on the past,
    \begin{equation*}
        f(x_{k-1}) - f(x_{k})
            \geq \alpha \left( 1 - \frac{\alpha L}{2} \right) \norm*{P_{k-1}^\top g_{k-1}}^2,
    \end{equation*}
    which implies that $f(x_{k-1}) \leq f(x)$ necessarily. From the same equation, we also get
    \begin{align*}
        f(x_{k-1}) - f(x_{k})
            &\geq \alpha \left( 1 - \frac{\alpha L}{2} \right) \left(\delta+\frac{d}{n-1}(1+\tau)\right)\norm*{g_{k-1}}^2
    \end{align*}
    with probability at least $1-2\exp{(-cd\tau^2)} - p_v(\delta) := 1 - p_{decay}(\tau, \delta, d)$ by Lemma \ref{lem:Pk_highprob}. Let $Y_{k-1}$ be the indicator event of the inequality above, then we have that $\mathbb{E}[Y_{k-1}] \geq 1 - p_{decay}(\tau, \delta, d)$. Furthermore, since the step size is chosen such that the RHS is positive, we have
    \begin{equation*}
        Y_{k-1} \norm*{g_{k-1}}^2 \leq \frac{f(x_{k-1}) - f(x_{k})}{\alpha \left( 1 - \frac{\alpha L}{2} \right) \left(\delta+\frac{d}{n-1}(1+\tau)\right)}.
    \end{equation*}
    Hence, we get
    \begin{equation*}
        \left( \min_{i \in [0, k-1]}  \norm*{g_i}^2\right) \frac{1}{k} \sum_{i=0}^{k-1} Y_i\leq \frac{f(x_{0}) - f(x_{k})}{k\alpha \left( 1 - \frac{\alpha L}{2} \right) \left(\delta+\frac{d}{n-1}(1+\tau)\right)}.
    \end{equation*}
    We have by a Chernoff bound (see \cite{vershynin2018highdimensional}), that for all $\sigma \in (0,1)$,
    \begin{equation}\label{eqn:decay_total_prob}
        \mathbb{P} \left[ \sum_{i=1}^{k-1} Y_i \geq (1-\sigma)(1 - p_{decay}(\tau, \delta, d))k \right] \geq 1 - \exp{\left( -\frac{\sigma^2}{2}(1 - p_{decay}(\tau, \delta, d))k \right)}.
    \end{equation}
    Substituting the PL-inequality, we obtain
    \begin{align*}
        f(x_k) - f(x^\ast) 
            &= f(x_k) - f(x_{k-1}) + f(x_{k-1}) - f(x^\ast) \\
            &\leq -\alpha \left( 1 - \frac{\alpha L}{2} \right) \left(\delta+\frac{d}{n-1}(1+\tau)\right)\norm*{g_{k-1}}^2Y_{k-1} + (f(x_{k-1}) - f(x^\ast)) \\
            &\leq \left[1 - 2\mu\alpha \left( 1 - \frac{\alpha L}{2} \right) \left(\delta+\frac{d}{n-1}(1+\tau)\right) {Y_{k-1}}\right](f(x_{k-1}) - f(x^\ast)) \\
            &= \left[1 - 2\mu\alpha \left( 1 - \frac{\alpha L}{2} \right) \left(\delta+\frac{d}{n-1}(1+\tau)\right) \right]^{Y_{k-1}} (f(x_{k-1}) - f(x^\ast)).
    \end{align*}
    Consequently, we obtain
    \begin{equation}
        f(x_k) - f(x^\ast)
            \leq \left[1 - 2\mu\alpha \left( 1 - \frac{\alpha L}{2} \right) \left(\delta+\frac{d}{n-1}(1+\tau)\right) \right]^{\sum_{i=0}^{k-1} Y_{i}} (f(x_{0}) - f(x^\ast)).
    \end{equation}
    Using Equation \eqref{eqn:decay_total_prob}, we have that
    \begin{equation}\label{eqn:nonincreasing_1}
        f(x_k) - f(x^\ast)
            \leq \left[1 - 2\mu\alpha \left( 1 - \frac{\alpha L}{2} \right) \left(\delta+\frac{d}{n-1}(1+\tau)\right) \right]^{(1-\sigma)(1 - p_{decay}(\tau, \delta, d))k} (f(x_{0}) - f(x^\ast)).
    \end{equation}
    Let 
    \[
        \rho = \left[1 - 2\mu\alpha \left( 1 - \frac{\alpha L}{2} \right) \left(\delta+\frac{d}{n-1}(1+\tau)\right) \right]^{(1-\sigma)(1 - p_{decay}(\tau, \delta, d))},
    \]
    then we have
    \begin{equation*}
        f(x_k) - f(x^\ast)
            \leq \rho^k (f(x_{0}) - f(x^\ast)).
    \end{equation*}
    From L-smoothness of the function, we have
    \begin{equation*}
        \norm*{g_k}^2 \leq 2L (f(x_k) - f^\ast) \leq 2L(f(x_0) - f(x^\ast))\rho^{k},
    \end{equation*}
    giving the desired equation with constant
    \[
        C = \sqrt{2L(f(x_0) - f(x^\ast))}\
    \]
    holding with probability at least $1 - e^{- \frac{1}{2} \sigma^2k(1 - p_{decay}(\tau, \delta, d))}$.
\end{proof}
\begin{rem}
    For simplicity, we can let $\sigma = 1/2$ such that the result holds with 
    $$\rho_{decay} = \left[1 - 2\mu\alpha \left( 1 - \frac{\alpha L}{2} \right) \left(\delta+\frac{d}{n-1}(1+\tau)\right) \right]^{\frac{1}{2}(1 - p_{decay}(\tau, \delta, d))}$$ with probability at least $1 - e^{-\frac{1}{8}k(1 - p_{decay}(\tau, \delta, d))}$
\end{rem}

\paragraph{Rescaling Matrix for Phase 2.}
For phase 2 of the argument, after $k_1$ has been fixed, we will re-scale the random part of the matrix by a factor of $\sqrt{\frac{n-1}{d}}$. Under this setup, we will obtain a similar decay rate (with a slightly different step size) as in the original matrix. Consider the rescaled version of the random matrix $\tilde P_k$ by a factor of $\sqrt{\frac{n-1}{d}}$, which we denote by $\hat P_k = \sqrt{\frac{n-1}{d}}$. Then, it satisfies the basic properties
\begin{equation}
    \hat P_k^\top \hat P_k = \frac{n-1}{d}I, \qquad \mathbb{E} \left[\hat P_k \hat P_k^\top\right] = I - \hat v_k \hat v_k^\top.
\end{equation}
Then, the matrix $P_k = \begin{pmatrix}\hat v_k & \hat P_k\end{pmatrix}$ has the properties
\begin{equation}
    P_k^\top P_k = \begin{pmatrix}
        1 & 0\\ 0 & \frac{n-1}{d} I_d
    \end{pmatrix}, \qquad \mathbb{E}[P_k P_k^\top] = I_n.
\end{equation}

\begin{cor}\label{cor:rescaled_Pk_highprob}
    As a corollary of Lemma \eqref{lem:Pk_highprob}, we have that
    \begin{equation}
        \mathbb{P} \left[ \abs*{\inner*{\hat P^\top u, \hat P^\top v} -\inner*{u_\perp, v_\perp}} \leq \tau \norm*{u_\perp}\norm*{v_\perp} \right] \geq 1 - 2\exp{(-cd\tau^2)}.
    \end{equation}
    \begin{proof}
        Notice that 
        \begin{align*}
            \abs*{\inner*{\hat P^\top u, \hat P^\top v} - \inner*{u_\perp, v_\perp}} 
            &= \frac{n-1}{d}\abs*{\inner*{\tilde P^\top u, \tilde P^\top v} - \frac{d}{n-1}\inner*{u_\perp, v_\perp}} \\
            &\leq \frac{n-1}{d} \left(\tau \frac{d}{n-1} \norm{u_\perp}\norm{v_\perp}\right) \\
            &= \tau \norm*{u_\perp}\norm*{v_\perp}.
        \end{align*}
    \end{proof}
\end{cor}

\begin{lem}[Decay Rate Under Rescaled $P_k$]\label{lem:decay_rescaled}
    Suppose $P_k = \begin{pmatrix}\hat v_k & \hat P_k\end{pmatrix}$, we have
    \begin{equation*}
        f(x_{k+1}) - f(x_{k}) 
         \leq \left[ -\alpha \tau + \frac{\alpha^2 L}{2} \left(1 - \frac{n-1}{d}(1+\tau)\right) \right] \inner*{\hat v_k, g_k}^2
          + \left[-\alpha (1-\tau) + \frac{\alpha^2 L}{2} \frac{n-1}{d}\right] \norm*{g_k}^2
    \end{equation*}
    with probability at least $1 - 2e^{-cd\tau^2}$. Consequently, if $\inner*{\hat v_k, g_k}^2 \geq \delta \norm*{g_k}^2$, we have
    \begin{equation}
        f(x_{k+1}) - f(x_{k}) 
         \leq \left(\left[ -\alpha \tau + \frac{\alpha^2 L}{2} \left(1 - \frac{n-1}{d}(1+\tau)\right) \right] \delta + \left[-\alpha (1-\tau) + \frac{\alpha^2 L}{2} \frac{n-1}{d}\right] \right)\norm*{g_k}^2.
    \end{equation}
\end{lem}

\begin{proof}
    By L-smoothness, we have
    \begin{equation}
        f(x_{k+1}) - f(x_{k}) \leq -\alpha \norm*{P_k^\top g_k}^2 + \frac{\alpha^2 L }{2} \norm*{P_k P_k^\top g_k}^2.
    \end{equation}
    Using Corollary \ref{cor:rescaled_Pk_highprob}, we have
    \begin{align*}
        \norm*{P_k^\top g_k}^2 
            &= \inner*{\hat v_k, g_k}^2 + \norm*{\tilde P_k^\top g_k}^2\\
            &\geq \inner*{\hat v_k, g_k}^2 + (1-\tau) \norm*{\Proj_{\hat v_k^\perp}g_k}^2 \\
            &= \inner*{\hat v_k, g_k}^2 + (1-\tau) \left( \norm*{g_k}^2 - \inner*{\hat v_k, g_k}^2\right) \\
            &= \tau \inner*{\hat v_k, g_k}^2 + (1-\tau) \norm*{g_k}^2.
    \end{align*}
    with probability $1 - 2\exp{(-cd\tau^2)}$. For the second term, we find that
    \begin{align*}
        g_k^\top P_k P_k^\top P_k P_k^\top g_k
            &= g_k^\top \left( \hat v_k \hat v_k^\top + \frac{n-1}{d} \tilde P_k \tilde P_k^\top \right) g_k \\
            &= \inner*{\hat v_k, g_k}^2 + \frac{n-1}{d}\norm*{\tilde P_k^\top g_k}^2 \\
            &\leq \inner*{\hat v_k, g_k}^2 + \frac{n-1}{d} (1+\tau) \norm*{\Proj_{\hat v_k^\perp}g_k}^2 \\
            &= \inner*{\hat v_k, g_k}^2 + \frac{n-1}{d} (1+\tau) \left( \norm*{g_k}^2 - \inner*{\hat v_k, g_k}^2\right) \\
            &= \frac{n-1}{d} \norm*{g_k}^2 + \left(1 - \frac{n-1}{d}(1+\tau)\right) \inner*{\hat v_k, g_k}^2
    \end{align*}
    under the same event. Substituting both back into the equation above, we obtain
    \begin{align}
         f(x_{k+1}) - f(x_{k}) 
         &\leq -\alpha \tau \inner*{\hat v_k, g_k}^2 - \alpha (1-\tau) \norm*{g_k}^2 \nonumber\\
         &\qquad + \frac{\alpha^2 L }{2} \left( \frac{n-1}{d} \norm*{g_k}^2 + \left(1 - \frac{n-1}{d}(1+\tau)\right) \inner*{\hat v_k, g_k}^2 \right) \nonumber\\
         &= \left[ -\alpha \tau + \frac{\alpha^2 L}{2} \left(1 - \frac{n-1}{d}(1+\tau)\right) \right] \inner*{\hat v_k, g_k}^2 \nonumber\\
         &\qquad + \left[-\alpha (1-\tau) + \frac{\alpha^2 L}{2} \frac{n-1}{d}\right] \norm*{g_k}^2. \label{eqn:nonincreasing_2}
    \end{align}
\end{proof}

\begin{rem}
    From Equations \eqref{eqn:nonincreasing_1} and \eqref{eqn:nonincreasing_2}, it is clear that with suitable step sizes in either phases, $f(x_{k+1}) \leq f(x_k)$. (The step sizes will be chosen to satisfy this, as is the case for all gradient-descent based methods.) This means that $f(x_k) \leq f(x_0)$, and by strong-convexity,
    \begin{equation}
        \frac{\mu}{2} \norm*{x_k - x^\ast} \leq f(x_k) - f^* \leq f(x_0) - f^* \leq \frac{\mu r_0^2}{2}.
    \end{equation}
    So, $\norm*{e_k} = \norm*{x_k - x^\ast} \leq r_0$ and the taylor approximations hold for all iterates.
\end{rem}

The main idea of the proof is that the iteration is a perturbed contraction, meaning that the main part contracts with some additional error added to it. This means that the iterates converge to a small neighbourhood around 0 of size determined by the perturbation. By ensuring the perturbation is small at the start, we can ensure that the iterates converge to a very small value. The lemma below proves the convergence of the perturbed fixed point iteration.
\begin{lem}[Perturbed Fixed Point]\label{lem:fp_perturb}
    Let $\mu_2 < \mu_1$ and $\epsilon > 0$ be a small term. Then the iteration
    \begin{equation}
        t_{k+1} \leq \frac{\mu_2}{\mu_1} t_k + \epsilon
    \end{equation}
    satisfies
    \begin{equation}
        \limsup t_k = O(\epsilon).
    \end{equation}
\end{lem}
\begin{proof}
    Consider the recurrence
    \begin{equation*}
        s_{k+1} = \frac{\mu_2}{\mu_1} s_k + \epsilon,
    \end{equation*}
    which has a solution
    \begin{equation*}
        s_k = \left(\frac{\mu_2}{\mu_1}\right)^k t_0 + \epsilon \sum_{i=0}^{k-1} \left(\frac{\mu_2}{\mu_1}\right)^i = \left(\frac{\mu_2}{\mu_1}\right)^k t_0 + \frac{\epsilon}{1 - \frac{\mu_2}{\mu_1}} \left[1 - \left(\frac{\mu_2}{\mu_1} \right)^k\right].
    \end{equation*}
    Taking $\limsup$, we have
    \begin{equation*}
        \limsup_k s_k = \epsilon\frac{\mu_1}{\mu_1 - \mu_2}.
    \end{equation*}
    Since $t_k < s_k$ by induction, we have $\limsup_k t_k \leq c\epsilon$, where $c = \frac{\mu_1}{\mu_1 - \mu_2}$.
\end{proof}

Lastly, the analysis will be done through a sequential conditioning argument. Each step will be independent of the past events and we will obtain a conditional recurrence. The lemma below will then allow us to apply this result to obtain a final probability of the events of interest.
\begin{lem}\label{lem:prod_prob}
    Consider the sequence of $\sigma-$algebras $\mathcal{F}_1 \subset \cdots \subset \mathcal{F}_k$ and the sequence of events $A_1, \cdots, A_k$ which are measurable in their respective $\sigma$-algebra. i.e. $\mathbb{E}[1_{A_i} \mid \mathcal{F}_i] = \mathbb{P}[A_i]$. If $\forall k, \,\mathbb{P}[A_k \mid \mathcal{F}_{k-1}] \geq 1 - p_k$, then
    \begin{equation}
        \mathbb{P}\left[\bigcap_{i=1}^k A_i\right] 
        \geq (1-p_k) \times \mathbb{P}\left[\bigcap_{i=1}^{k-1} A_i\right] 
        \geq \prod_{i=1}^k (1-p_i).
    \end{equation}
\end{lem}
\begin{proof}
    Since $\mathbb{P}[A] = \mathbb{E}[1_A]$, we have
    \begin{align*}
        \mathbb{E}\left[1_{\bigcap_{i=1}^k A_i}\right]
            &= \mathbb{E}\left[\mathbb{E} \left[ 1_{\bigcap_{i=1}^{k-1} A_i} 1_{A_k} \bigm| \mathcal{F}_{k-1}\right]\right] \\
            &= \mathbb{E}\left[1_{\bigcap_{i=1}^{k-1} A_i} \mathbb{E} \left[ 1_{A_k} \bigm| \mathcal{F}_{k-1}\right]\right]\\
            &\geq (1-p_k) \times \mathbb{E}\left[1_{\bigcap_{i=1}^{k-1} A_i}\right]
    \end{align*}
    and the result follows.        
\end{proof}

\subsection{Analysis for Persistent Alignment}

\begin{proof}[Proof of Theorem \ref{thm:local-norefresh}]
    The proof will mainly proceed in 2 steps.
    \begin{enumerate}
        \item We will first show that there is a $k_1 \in (0, K_\epsilon)$ such that $\langle g_{k_1}, u_1 \rangle^2 \geq \delta_0 \norm*{g_{k_1}}^2$. In other words, we find that eventually the gradient of the function will be close to the direction $u_1$.
        \item In the next part, we will use the above result to show that even with a fixed $v_k = v_{k_1}$ for all $k > k_1$, $g_k$ still remains closely correlated to the direction $u_1$. This is done by controlling the error $E_k$ of successive gradient descent steps due to the random projections, by the closeness of the vectors between $u_1$ and $v_{k_1} = g_{k_1}$.
    \end{enumerate}
    In the heart of this analysis, we will consider the events which we desire and obtain a one-step recursion and the probability of this one step, conditioned on the past. This allows us to only consider the randomness introduced in this step, which is simply $\tilde P_k$.

    \paragraph{Phase 1: Gradient eventually enters into cone around $u_1$.} Define the following variable
    \begin{equation}
        t_k = \frac{\norm*{\Proj_{u_1^\perp}g_k}}{\norm*{\Proj_{u_1} g_k}}.
    \end{equation}
    Notice that $t_k$ is essentially the tangent of the angle between $g_k$ and $u_1$. To find an upper bound of $t_{k+1}$, we can first find an upper bound for the numerator and a lower bound for the denominator.
    \paragraph{Part 1a: Upper bound for Numerator.}
    Using Lemma \eqref{lem:gradrec}, we have
    \begin{equation*}
        \norm*{\Proj_{u_1^\perp}g_{k+1}} \leq \norm*{\Proj_{u_1^\perp}Mg_k} + \alpha \norm*{\Proj_{u_1^\perp}HE_k} + \norm*{\Proj_{u_1^\perp}r_k},
    \end{equation*}
    where $H$ is the Hessian at the optimal point and $M = I-\alpha H$.
    For the first term, we use the fact that $M$ and $\Proj_{u_1}$ commute to get
    \begin{equation*}
        \norm*{\Proj_{u_1^\perp}Mg_k} 
        = \norm*{M\Proj_{u_1^\perp}g_k}
        = \norm*{M\Proj_{u_1^\perp}^2g_k}
        \leq \norm*{M\Proj_{u_1^\perp}}\norm*{\Proj_{u_1^\perp}g_k}
        = \mu_2\norm*{\Proj_{u_1^\perp}g_k}.
    \end{equation*}
    For the second term, we can simply bound it by
    \begin{equation}
        \norm*{\Proj_{u_1^\perp}HE_k} \leq \lambda_n \norm*{(I-\tilde P_k \tilde P_k^\top)\Proj_{\hat v_k^\perp} g_k},
    \end{equation}
    where $E_k = (I-\tilde P_k \tilde P_k^\top - \hat v_k \hat v_k^\top) g_k= (I-\tilde P_k \tilde P_k^\top)(I - \hat v_k \hat v_k^\top)g_k$. Using the fact that $I - \tilde P_k \tilde P_k^\top$ is a projection and the alignment assumption, we have
    \begin{equation}
        \norm*{\Proj_{\hat v_k^\perp} g_k} = \sqrt{\left(I - \hat v_k \hat v_k^\top \right) g_k} \leq \sqrt{1-\delta} \norm*{g_k}.
    \end{equation}
    For the third term, we use Lemma \eqref{lem:gradrec} to get
    \begin{equation*}
        \norm*{\Proj_{u_1^\perp}r_k} \leq C_2 \|g_k\|^2.
    \end{equation*}
    Combining everything, the numerator can be bounded by
    \begin{align*}
        \norm*{\Proj_{u_1^\perp}g_{k+1}} 
        &\leq \mu_2 \norm*{\Proj_{u_1^\perp}g_{k}} + \alpha \lambda_n \sqrt{1-\delta}\|g_k\| + C_2\|g_k\|^2.
    \end{align*}
    \paragraph{Part 1b: Lower bound for Denominator.}
    For the denominator (the projection onto $u_1$), we have
    \begin{align*}
        \norm*{\Proj_{u_1} g_{k+1}} 
        &= \norm*{\Proj_{u_1}\left[Mg_k + \alpha HE_k + r_k\right]} \\
        &= \left|\mu_1 u_1^\top g_k + \alpha \lambda_1 u_1^\top E_k + u_1^\top r_k\right| \\
        &\geq \mu_1 \norm*{\Proj_{u_1}g_k} - \alpha \lambda_1 |u_1^\top E_k| - \|r_k\|.
    \end{align*}
    The second and third term will use the same bounds as in the numerator, giving us
    \begin{equation}
        \norm*{\Proj_{u_1} g_{k+1}} 
            \geq \mu_1 \norm*{\Proj_{u_1} g_k} - \alpha \lambda_1 \sqrt{1-\delta} \norm{g_k} - C_2 \norm{g_k}^2.
    \end{equation}
    \paragraph{Part 1c: Combining both bounds.}
    Taking both bounds, we obtain
    \begin{align*}
        t_{k+1} 
            &\leq \frac{\mu_2 \norm*{\Proj_{u_1^\perp}g_{k}} + \alpha \lambda_n \sqrt{1-\delta}\|g_k\| + C_2\|g_k\|^2}{\mu_1 \norm*{\Proj_{u_1} g_k} - \alpha \lambda_1 \sqrt{1-\delta} \norm{g_k} - C_2 \norm{g_k}^2} \\
            &= \frac{\mu_2 t_k + \alpha \lambda_n \sqrt{1-\delta} \frac{\|g_k\|}{\norm*{\Proj_{u_1} g_k} } + C_2\frac{\|g_k\|^2}{\norm*{\Proj_{u_1} g_k} }}{\mu_1 - \alpha \lambda_1 \sqrt{1-\delta} \frac{\norm{g_k}}{\norm*{\Proj_{u_1} g_k} } - C_2 \frac{\norm{g_k}^2}{\norm*{\Proj_{u_1} g_k}}}.
    \end{align*}
    Define $\epsilon_\delta = \sqrt{1-\delta}$ and $\epsilon_k = \frac{\norm*{g_k}}{\norm*{\Proj_{u_1} g_k}}$. We have by Lemma \ref{lem:exp_decay} that the gradient norm decays exponentially with factor $\rho_{decay}$ (omitting the variables for simplicity). Since we are in the local regime, we can assume that $\norm*{g_k} \leq \rho_{decay}^{k_0+k}$ for some fixed $k_0$ where the iteration of this convergence begins. Then, we can express the recursion of $t_{k+1}$ as
    \begin{equation}
        t_{k+1} 
            \leq \frac{\mu_2 t_k + \alpha \lambda_n \epsilon_\delta \epsilon_k + C_2\epsilon_k \rho_{decay}^{k_0 + k}}{\mu_1 - \alpha \lambda_1 \epsilon_\delta \epsilon_k - C_2 \epsilon_k \rho_{decay}^{k+k_0}}.
    \end{equation}
    
    \paragraph{Part 1d: Finding the Conditions For a Perturbed Fixed Point.}
    Since $\epsilon_\delta$ can be controlled arbitrarily small by having $\delta \gg 0$ and $\rho^{k_0}$ is exponentially small, we can use Taylor's expansion for $\frac{1}{1-x} = 1 + x(1+o(1))$ to get
    \begin{align*}
        t_{k+1} 
            &\leq \frac{\mu_2 t_k + \alpha \lambda_n \epsilon_{\delta} \epsilon_k + C_2 \epsilon_k \rho_{decay}^{k+k_0}}{\mu_1\left( 1 - \alpha \lambda_1 \mu_1^{-1} \epsilon_{\delta} \epsilon_k - C_2 \mu_1^{-1} \epsilon_k \rho_{decay}^{k+k_0}\right)} \\
            &= \frac{\mu_2 t_k + \alpha \lambda_n \epsilon_{\delta} \epsilon_k + C_2 \epsilon_k \rho_{decay}^{k+k_0}}{\mu_1} \left( 1 + \left(\alpha \lambda_1 \mu_1^{-1} \epsilon_{\delta} \epsilon_k + C_2 \mu_1^{-1} \epsilon_k \rho_{decay}^{k+k_0} \right)(1+o(1))\right) \\
            &= \frac{\tilde \mu_2^{(k)}}{\mu_1} t_k + \tilde \epsilon_k,
    \end{align*}
    where
    \begin{align*}
        \tilde \mu_2^{(k)} 
            &= \mu_2 \left( 1 + \left(\alpha \lambda_1 \mu_1^{-1} \epsilon_{\delta} \epsilon_k + C_2 \mu_1^{-1} \epsilon_k \rho_{decay}^{k+k_0}\right)(1+o(1))\right) \\
        \tilde \epsilon_k 
            &= \mu_1^{-1}\left(\alpha \lambda_n \epsilon_{\delta} \epsilon_k + C_2 \epsilon_k \rho_{decay}^{k+k_0}\right)\left(1 + \left(\alpha \lambda_1 \mu_1^{-1} \epsilon_{\delta} \epsilon_k + C_2 \mu_1^{-1} \epsilon_k \rho_{decay}^{k+k_0} \right)(1+o(1)) \right)
    \end{align*}
    By assumption, we have $\abs*{\inner*{u_1, g_{k_0}}} = \omega \norm{g_{k_0}}$ for some $\omega > 0$, which is equivalent to $\epsilon_0 = \omega^{-1}$. Consequently, we have $\norm*{\Proj_{u_1^\perp} g_k} = \sqrt{1-\omega^2}\norm*{g_k}$. This tells us that
    \begin{equation}
        t_0 = \frac{\sqrt{1-\omega^2}}{\omega}.
    \end{equation}
    Consider the following decomposition of expressing $\|g_k\|$ in terms of $\norm*{\Proj_{u_1^\perp}g_k}$
    \begin{equation}\label{eqn:g_to_tk}
        \|g_k\| = \sqrt{\norm*{\Proj_{u_1}g_k}^2 + \norm*{\Proj_{u_1^\perp}g_k}^2} = \norm*{\Proj_{u_1}g_k} \sqrt{1 + t_k^2}.
    \end{equation}
    We find that
    \begin{equation}
        \epsilon_k = \sqrt{1 + t_k^2} \implies \epsilon_{k+1} = \sqrt{1 + t_{k+1}^2} \leq \sqrt{1 + \left(\frac{\tilde \mu_2^{(k)}}{\mu_1} t_k + \tilde \epsilon_k \right)^2}.
    \end{equation}
    Then, an upper bound for $\epsilon_{k+1}$ yields
    \begin{align*}
        \epsilon_{k+1}
            &\leq \sqrt{\frac{1 + \left(\frac{\tilde \mu_2^{(k)}}{\mu_1} t_k + \tilde \epsilon_k \right)^2}{1 + t_k^2}} \epsilon_k.
    \end{align*}
    To show that it does not grow, we just have to show that the factor in front of $\epsilon_k$ is $<1$. Equivalently, we have
    \begin{align*}
        &\sqrt{\frac{1 + \left(\frac{\tilde \mu_2^{(k)}}{\mu_1} t_k + \tilde \epsilon_k \right)^2}{1 + t_k^2}}  < 1 \\
        \iff& \frac{\tilde \mu_2^{(k)}}{\mu_1} t_k + \tilde \epsilon_k < t_k \\
        \iff& \mu_1 \tilde \epsilon_k < \left(\mu_1 - \tilde \mu_2^{(k)}\right)t_k.
    \end{align*}
    So, as long as $\tilde \epsilon_0 < \left(1 - \frac{\tilde \mu_2^{(0)}}{\mu_1}\right) t_0 = \left(1 - \frac{\tilde \mu_2^{(0)}}{\mu_1}\right) \frac{\sqrt{1-\omega^2}}{\omega}$, we can ensure that $\epsilon_{1} < \epsilon_0 = \omega^{-1}$, which also implies that $\tilde \epsilon_1 < \tilde \epsilon_0$. So, we require that
    \begin{align*}
        &\mu_1^{-1}\left(\alpha \lambda_n \epsilon_{\delta} \epsilon_0 + C_2 \epsilon_0 \rho_{decay}^{k_0}\right)\left(1 + \left(\alpha \lambda_1 \mu_1^{-1} \epsilon_{\delta} \epsilon_0 + C_2 \mu_1^{-1} \epsilon_0 \rho_{decay}^{k_0} \right)(1+o(1)) \right) \\
        <& \left(1 - \frac{\mu_2 \left( 1 + \left(\alpha \lambda_1 \mu_1^{-1} \epsilon_{\delta} \epsilon_0 + C_2 \mu_1^{-1} \epsilon_0 \rho_{decay}^{k_0}\right)(1+o(1))\right)}{\mu_1}\right) \left(\frac{\sqrt{1-\omega^2}}{\omega}\right).
    \end{align*}
    To simplify the terms further, we have by definition that $|o(1)| < \eta$ for some $\eta \in (0,1)$. Then, a sufficient condition for the above inequality is
    \begin{equation*}
        \left[\frac{\sqrt{1-\omega^2}}{\omega}\mu_2 + \alpha \lambda_n \epsilon_{\delta} \epsilon_0 + C_2 \epsilon_0 \rho_{decay}^{k_0}\right]\left[1 + \left(\alpha \lambda_1 \mu_1^{-1} \epsilon_{\delta} \epsilon_0 + C_2 \mu_1^{-1} \epsilon_0 \rho_{decay}^{k_0} \right)(1+\eta) \right] < \frac{\sqrt{1-\omega^2}}{\omega} \mu_1.
    \end{equation*}
    Expanding and expressing the equation in terms of $\epsilon_\delta$, we get
    \begin{equation}
        a \epsilon_{\delta}^2 + b \epsilon_\delta + c < 0,
    \end{equation}
    where
    \begin{align}
        a &=  (1+\eta) \left(\alpha \lambda_1 \mu_1^{-1} \epsilon_0 \right) \left( \alpha \mu_1 \lambda_n \epsilon_0 \right) \nonumber\\
        &= (1+\eta)\alpha^2\lambda_1\lambda_n \epsilon_0^2 \\
        b &=  \alpha \mu_1 \lambda_n \epsilon_0 \left[1 + C_2 \mu_1^{-1} \epsilon_0 \rho_{decay}^{k_0} (1+\eta) \right] +  \left[\frac{\sqrt{1-\omega^2}}{\omega}\mu_2 +  C_2 \epsilon_0 \rho_{decay}^{k_0}\right]\alpha \lambda_1 \mu_1^{-1} \epsilon_0 (1+\eta) \nonumber \\
        &= \alpha \epsilon_0 \left[(1+\eta) \frac{\lambda_1}{\mu_1} \left(\frac{\sqrt{1-\omega^2}}{\omega} \mu_2 + C_2 \epsilon_0 \rho_{decay}^{k_0}\right) + \lambda_n \left(\mu_1 + C_2 \epsilon_0 \rho_{decay}^{k_0} (1+\eta) \right) \right] \\
        c &= \left(\frac{\sqrt{1-\omega^2}}{\omega}\mu_2 + C_2 \epsilon_0 \rho_{decay}^{k_0}\right)\left(1 + C_2 \mu_1^{-1} \epsilon_0 \rho_{decay}^{k_0} (1+\eta) \right) - \frac{\sqrt{1-\omega^2}}{\omega} \mu_1 \nonumber \\
        &=C_2 \epsilon_0 \rho_{decay}^{k_0} + (1+\eta)C_2 \mu_1^{-1} \epsilon_0 \rho_{decay}^{k_0}\left(\frac{\sqrt{1-\omega^2}}{\omega}\mu_2 + C_2 \epsilon_0 \rho_{decay}^{k_0}\right) - \alpha \frac{\sqrt{1-\omega^2}}{\omega} (\lambda_2 - \lambda_1) \nonumber \\
        &= C_2 \epsilon_0 \rho_{decay}^{k_0} \left(1 + (1+\eta) \frac{\sqrt{1-\omega^2}}{\omega} \frac{\mu_2}{\mu_1} + (1+\eta) C_2 \mu_1^{-1}\epsilon_0 \rho_{decay}^{k_0} \right) - \alpha \frac{\sqrt{1-\omega^2}}{\omega} (\lambda_2 - \lambda_1).
    \end{align}
    Since $a > 0$, to obtain an admissible solution $0 \leq \epsilon_\delta < \epsilon_\delta^\ast$, we require that $c < 0$. Under this condition, coupled with $a > 0$, it suffices that $\epsilon_\delta^\ast = -\frac{c}{b}$.
    \begin{align*}
        \epsilon_\delta^\ast 
            &= \frac{\alpha \frac{\sqrt{1-\omega^2}}{\omega} (\lambda_2 - \lambda_1) -C_2 \epsilon_0 \rho_{decay}^{k_0} \left(1 + (1+\eta) \frac{\sqrt{1-\omega^2}}{\omega} \frac{\mu_2}{\mu_1} + (1+\eta) C_2 \mu_1^{-1}\epsilon_0 \rho_{decay}^{k_0} \right) }{\alpha \epsilon_0 \left[(1+\eta) \frac{\lambda_1}{\mu_1} \left(\frac{\sqrt{1-\omega^2}}{\omega} \mu_2 + C_2 \epsilon_0 \rho_{decay}^{k_0}\right) + \lambda_n \left(\mu_1 + C_2 \epsilon_0 \rho_{decay}^{k_0} (1+\eta) \right) \right]}.
    \end{align*}
    Assuming that $\rho_{decay}^{k_0}$ is sufficiently small in the local regime, we have
    \begin{equation}
        \epsilon_\delta^\ast 
            \lesssim \frac{ \frac{\sqrt{1-\omega^2}}{\omega} (\lambda_2 - \lambda_1)}{\epsilon_0 \left[(1+\eta) \frac{\sqrt{1-\omega^2}}{\omega} \frac{\mu_2}{\mu_1} \lambda_1 + \lambda_n \mu_1 \right]}
            < \frac{\lambda_2 - \lambda_1}{(1+\eta) \omega^{-1} \frac{\mu_2}{\mu_1} \lambda_1 + (1-\omega^2)^{-1/2} \lambda_n \mu_1},
    \end{equation}
    where we used $\omega^{-1} = \epsilon_0 \geq 1$. Since $\epsilon_{\delta}^\ast = \sqrt{1-\delta^\ast}$, we equivalently have
    \begin{equation}
        \delta^\ast
            \gtrsim 1 -  \left(\frac{\lambda_2 - \lambda_1}{(1+\eta) \omega^{-1} \frac{\mu_2}{\mu_1} \lambda_1 + (1-\omega^2)^{-1/2} \lambda_n \mu_1}\right)^2.
    \end{equation}
    Now, we want to ensure that this is indeed a contraction, to do so, we require $\mu_2^{(0)} < \mu_1$.
    \begin{align*}
        1 + \left(\alpha \lambda_1 \mu_1^{-1} \epsilon_{\delta} \epsilon_0 + C_2 \mu_1^{-1} \epsilon_0 \rho_{decay}^{k_0}\right)(1+\eta) 
        &< \frac{\mu_1}{\mu_2} \\
        \alpha \lambda_1 \mu_1^{-1} \epsilon_{\delta} \epsilon_0 + C_2 \mu_1^{-1} \epsilon_0 \rho_{decay}^{k_0}
        &< \frac{\mu_1 - \mu_2}{\mu_2 (1+\eta)} \\
        \lambda_1 \mu_1^{-1} \epsilon_{\delta} \epsilon_0 
        &< \frac{\lambda_2 - \lambda_1}{ \mu_2 (1+\eta)} - C_2 \alpha^{-1}\mu_1^{-1} \epsilon_0 \rho_{decay}^{k_0} \\
         \epsilon_{\delta} 
        &< \epsilon_0^{-1} \frac{\mu_1}{\mu_2(1+\eta)}\frac{\lambda_2 - \lambda_1}{ \lambda_1} - C_2 \frac{\rho_{decay}^{k_0}}{\alpha\lambda_1}.
    \end{align*}
    Equivalently, we have
    \begin{equation}
        \epsilon_{\delta}^\ast \lesssim \frac{\omega\mu_1}{\mu_2(1+\eta)}\frac{\lambda_2 - \lambda_1}{ \lambda_1} 
        \implies \delta^\ast \gtrsim 1 - \left(\omega\frac{\mu_1}{\mu_2(1+\eta)}\frac{\lambda_2 - \lambda_1}{ \lambda_1} \right)^2.
    \end{equation}
    For the inductive step, for $\epsilon_{2}$ to not expand, we require that $\mu_{1} \tilde \epsilon_1 < \left( \mu_1 - \tilde \mu_2^{(1)}\right)t_1$. Since $\tilde \epsilon_1 < \tilde \epsilon_0$ and $\mu_2^{(1)} < \mu_2^{(0)}$, we require $\mu_{1} \tilde \epsilon_0 < \left( \mu_1 - \tilde \mu_2^{(0)}\right)t_1$. Note that in Lemma \eqref{lem:fp_perturb}, we have that $\limsup t_k < \frac{\mu_1}{\mu_1 - \mu_2^{(0)}} \tilde \epsilon_0$. So either we have reached the stable point (i.e. $O(\tilde \epsilon_0)$) or $t_k$ is still larger than this radius. If the former, we are done, if the latter, we satisfy the inequality above implying $\epsilon_2 < \epsilon_1$. Hence, we have that either we have reached the neighbourhood around $\tilde \epsilon_0$ or $\epsilon_k$ will be non-increasing which implies the factors $\mu_2^{(k)}, \tilde \epsilon_{k}$ are non-increasing. 
    
    Then, for 
    $$\delta \gtrsim 1 - \min{ \left(\omega\frac{\mu_1}{\mu_2(1+\eta)}\frac{\lambda_2 - \lambda_1}{ \lambda_1},  \frac{\lambda_2 - \lambda_1}{(1+\eta) \omega^{-1} \frac{\mu_2}{\mu_1} \lambda_1 + (1-\omega^2)^{-1/2} \lambda_n \mu_1}\right)}^2,$$
    we have that $t_{k+1} \leq \frac{\tilde \mu_2^{(0)}}{\mu_1} t_k + \tilde \epsilon_0$ is a perturbed contracting, which according to Lemma \eqref{lem:fp_perturb}, $\limsup t_k = O(\epsilon_\delta) < C\epsilon_\delta$ for some fixed constant $C$. The contraction forces the iterates to stay within a $C\epsilon_\delta$ radius around 0. Then fix any target $\tilde \delta_0$, let $\epsilon_\delta = \frac{1}{C} \sqrt{\frac{1- \tilde \delta_0}{\tilde \delta_0}}$, there exists $k_1$ such that $t_{k_1} < C \epsilon_\delta$ (we assume $k_1 < K_\epsilon$, if not then we have already converged). Since $t_{k_1}$ is the tangent of the angle between $u_1$ and $g_{k_1}$, we have by the formula $\cos^2 \theta = \frac{1}{1 + \tan^2 \theta}$ that $\cos^2 \theta_{k_1} \ge \tilde \delta_0$ which is equivalent to $\inner*{u_1, g_{k_1}}^2 \geq \tilde \delta_0 \norm*{g_{k_1}}^2$. For simplicity, we let $\inner*{u_1, g_{k_1}}^2 = \delta_0 \norm*{g_{k_1}}^2$ where $\delta_0 \ge \tilde \delta_0$.
    
    \paragraph{Phase 2: Freezing $v_k$ still promises continued alignment with $u_1$.}
    For this phase, consider $k \ge k_1$. Consider the expansion
    \begin{equation}
        g_{k+1} = \left( I - \alpha H P_k P_k^\top \right) g_k + r_k.
    \end{equation}

    \paragraph{Part 2a: Upper Bounding Denominator.}
    Expanding $P_k P_k^\top = \hat v_{k_1} \hat v_{k_1}^\top + \tilde P_k \tilde P_k^\top$, we find that for $i \neq 1$,
    \begin{align*}
        u_i^\top g_{k+1} 
            &= u_i^\top \left( I - \alpha H \left(\hat v_{k_1} \hat v_{k_1}^\top + \tilde P_k \tilde P_k^\top \right) \right) g_k  + u_i^\top r_k \\
            &= u_i^\top g_k - \alpha \lambda_i \inner*{\tilde P_k^\top u_i, \tilde P_k^\top g_k} - \alpha \lambda_i u_i^\top \hat v_{k_1} \hat v_{k_1}^\top g_k + u_i^\top r_k.
    \end{align*}
    For the second term, we have
    \begin{align*}
        \inner*{\tilde P_k^\top u_i, \tilde P_k^\top g_k}
            &\geq \inner*{\left(I - \hat v_{k_1} \hat v_{k_1}^\top \right) u_i, \left(I - \hat v_{k_1} \hat v_{k_1}^\top \right) g_k} - \tau \norm*{\left(I - \hat v_{k_1} \hat v_{k_1}^\top \right) u_i} \norm*{\left(I - \hat v_{k_1} \hat v_{k_1}^\top \right) g_k} \\
            &= u_i^\top \left(I - \hat v_{k_1} \hat v_{k_1}^\top \right) g_k - \tau \norm*{\Proj_{\hat v_{k_1}^\perp}g_k}\norm*{\Proj_{\hat v_{k_1}^\perp}u_i}  \\
            &= u_i^\top g_k - u_i^\top \hat v_{k_1} \hat v_{k_1}^\top g_k - \tau \norm*{\Proj_{\hat v_{k_1}^\perp}g_k}\norm*{\Proj_{\hat v_{k_1}^\perp}u_i}
    \end{align*}
    with probability $1 - \exp{(-cd\tau^2)}$.
    Substituting back into the original equation, we obtain an upper bound
    \begin{align}
        u_i^\top g_{k+1} 
            &\leq u_i^\top g_k - \alpha \lambda_i \left( u_i^\top g_k - u_i^\top \hat v_{k_1} \hat v_{k_1}^\top g_k - \tau \norm*{\Proj_{\hat v_{k_1}^\perp}g_k}\norm*{\Proj_{\hat v_{k_1}^\perp}u_i} \right)  \nonumber \\
            &\qquad - \alpha \lambda_i u_i^\top \hat v_{k_1} \hat v_{k_1}^\top g_k + u_i^\top r_k \nonumber \\
            &= (1-\alpha \lambda_i) \inner*{u_i, g_k} + \alpha \lambda_i \tau \norm*{\Proj_{\hat v_{k_1}^\perp}g_k}\norm*{\Proj_{\hat v_{k_1}^\perp}u_i} + u_i^\top r_k. \label{eqn:phase2_uigk}
    \end{align}
    Taking absolute value, we obtain
    \begin{align}
        \abs*{u_i^\top g_{k+1}} 
            &\leq (1-\alpha \lambda_i) \abs*{\inner*{u_i, g_k}} + \alpha \lambda_i \tau \left( \norm*{\Proj_{\hat v_{k_1}^\perp}g_k}\norm*{\Proj_{\hat v_{k_1}^\perp}u_i} \right) + \norm*{r_k}
    \end{align}

    \paragraph{Part 2b: Lower Bounding Denominator.}
    On the other hand, for the projection onto $u_1$, we have
    \begin{align*}
        u_1^\top g_{k+1} 
            &= u_1^\top g_k - \alpha \lambda_1 \inner*{\tilde P_k^\top u_1, \tilde P_k^\top g_k} - \alpha \lambda_1 u_1^\top \hat v_{k_1} \hat v_{k_1}^\top g_k + u_1^\top r_k.
    \end{align*}
    Similarly, by JLL, the second term can be upper bounded by
    \begin{align*}
        \inner*{\tilde P_k^\top u_1, \tilde P_k^\top g_k}
            &\leq u_1^\top \left(I - \hat v_{k_1} \hat v_{k_1}^\top \right) g_k + \tau \norm*{\Proj_{\hat v_{k_1}^\perp}g_k}\norm*{\Proj_{\hat v_{k_1}^\perp}u_1} \\
            &= u_1^\top g_k - u_i^\top \hat v_{k_1} \hat v_{k_1}^\top g_k + \tau \norm*{\Proj_{\hat v_{k_1}^\perp}g_k}\norm*{\Proj_{\hat v_{k_1}^\perp}u_1}
    \end{align*}
    with probability $1 - \exp{(-cd\tau^2)}$. This give us
    \begin{align*}
        u_1^\top g_{k+1} 
            &\geq u_1^\top g_k - \alpha \lambda_1 \left( u_1^\top g_k - u_1^\top \hat v_{k_1} \hat v_{k_1}^\top g_k + \tau \norm*{\Proj_{\hat v_{k_1}^\perp}g_k}\norm*{\Proj_{\hat v_{k_1}^\perp}u_1} \right)  \\
            &\qquad - \alpha \lambda_1 u_1^\top \hat v_{k_1} \hat v_{k_1}^\top g_k + u_1^\top r_k \\
            &= (1-\alpha \lambda_1) \inner*{u_1, g_k} - \alpha \lambda_1 \tau \norm*{\Proj_{\hat v_{k_1}^\perp}g_k}\norm*{\Proj_{\hat v_{k_1}^\perp}u_1} + u_1^\top r_k.
    \end{align*}
    Using $\norm*{\Proj_{\hat v_{k_1}^\perp}u_1} \leq \sqrt{1-\delta_0} = \epsilon_{\delta_0}$ and triangle inequality, the absolute value of the correlation yields a lower bound
    \begin{align}
        \abs*{u_1^\top g_{k+1}}
            \geq (1-\alpha \lambda_1) \abs*{u_1^\top g_k} - \alpha \lambda_1 \tau \epsilon_{\delta_0} \norm*{\Proj_{\hat v_{k_1}^\perp}g_k} - \norm*{r_k}.
    \end{align}

    \paragraph{Part 2c: Combining Both Bounds.}
    If we define $t_{i}^{(k)} = \frac{\abs*{\inner*{u_i, g_k}}}{\abs*{\inner*{u_1, g_k}}}$, then we have the recursion
    \begin{equation}
        t^{(k+1)}_i \leq \frac{(1-\alpha \lambda_i) \abs*{\inner*{u_i, g_k}} + \alpha \lambda_i \tau_i \left( \norm*{\Proj_{\hat v_{k_1}^\perp}g_k}\norm*{\Proj_{\hat v_{k_1}^\perp}u_i} \right) + \norm*{r_k}}{(1-\alpha \lambda_1) \abs*{u_1^\top g_k} - \alpha \lambda_1 \tau_1 \epsilon_{\delta_0} \norm*{\Proj_{\hat v_{k_1}^\perp}g_k} - \norm*{r_k}}, \qquad \forall i \geq 2,
    \end{equation}
    which has a similar expression as in phase 1, holding with probability at least $1 - \sum_{i}\exp{(-cd\tau_i^2)}$. Simplifying the equation above, we get
    \begin{equation}
        t^{(k+1)}_i \leq \frac{\mu_i t^{(k)}_i + \alpha \lambda_i \tau_i \epsilon_k + C_2' \epsilon_k \norm*{g_k}}{\mu_1 - \alpha \lambda_1 \tau_1 \epsilon_{\delta_0} \epsilon_k - C_2' \epsilon_k \norm*{g_k}},
    \end{equation}
    where $\epsilon_k = \frac{\norm{g_k}}{\abs{u_1^\top g_k}}$ and $C_2' = \left(\frac{n-1}{d}\right)^2 C_2$. In Lemma \ref{lem:gradrec}, we used the fact that $\norm{P_k P_k^\top}_{op} = 1$, however in phase 2, we have that this is $\frac{n-1}{d}$ instead. Despite the case, the main contributor is still the exponentially decaying term, and it would only require $O(\log(n/d))$ steps for the extra factor to be negligible. The iteration starts from $k = k_1$, where from phase 1, we have $\epsilon_{k_1} = \delta_0^{-1/2}$. Doing the same trick as in phase 1, where we used the Taylor expansion for $\frac{1}{1-x}$, we obtain
    \begin{align*}
        t^{(k+1)}_i \leq \frac{\mu_i t^{(k)}_i + \alpha \lambda_i \tau_i \epsilon_k + C_2' \epsilon_k \norm*{g_k}}{\mu_1} \left[ 1 + (1+\eta)\left(\alpha \mu_1^{-1}\lambda_1 \tau_1 \epsilon_{\delta_0} \epsilon_k + C_2' \mu_1^{-1}\epsilon_k \norm*{g_k}\right) \right],
    \end{align*}
    for some $\eta \in (0,1)$. With a small perturbation to Lemma \eqref{lem:exp_decay}, we find that for $k > k_1$, we have $\norm*{g_{k} }^2 \leq C \tilde \rho^{k-k_1} \rho_{decay}^{k_1}$, where $\tilde \rho$ is the rate of decay in phase 2 with the rescaled matrix. Rewriting the equation yields
    \begin{align*}
       t^{(k+1)}_i 
            &\leq \frac{\mu_i t^{(k)}_i + \alpha \lambda_i \tau_i \epsilon_k + C_2' \epsilon_k \tilde \rho^{k-k_1}\rho_{decay}^{k_1}}{\mu_1} \left[ 1 + (1+\eta)\left(\alpha \mu_1^{-1}\lambda_1 \tau_1 \epsilon_{\delta_0} \epsilon_k + C_2' \mu_1^{-1} \epsilon_k \tilde \rho^{k-k_1} \rho_{decay}^{k_1}\right) \right] \\
            &= \frac{\tilde \mu_i^{(k)}}{\mu_1} t^{(k)}_i + \tilde \epsilon^{(k)}_i,
    \end{align*}
    where
    \begin{align*}
        \tilde \mu_i^{(k)} 
            &= \mu_i \left[ 1 + (1+\eta)\left(\alpha \mu_1^{-1}\lambda_1 \tau_1 \epsilon_{\delta_0} \epsilon_k + C_2' \mu_1^{-1} \epsilon_k \tilde \rho^{k-k_1} \rho_{decay}^{k_1}\right) \right] \\
        \tilde \epsilon_i^{(k)}
            &=  {\mu_1^{-1}}\left(\alpha \lambda_i \tau_i \epsilon_k + C_2' \epsilon_k \tilde \rho^{k-k_1} \rho_{decay}^{k_1}\right) \left[ 1 + (1+\eta)\left(\alpha \mu_1^{-1}\lambda_1 \tau_1 \epsilon_{\delta_0} \epsilon_k + C_2' \mu_1^{-1} \epsilon_k \tilde \rho^{k-k_1} \rho_{decay}^{k_1}\right) \right].
    \end{align*}

    \paragraph{Part 2d: Finding the Condition for Perturbed Fixed Point.}
    First, for the factor to be a contraction, we require $\tilde \mu_i^{(k)} < \mu_1$.
    \begin{align*}
        1 + (1+\eta)\left(\alpha \mu_1^{-1}\lambda_1 \tau_1 \epsilon_{\delta_0} \epsilon_k + C_2' \mu_1^{-1} \epsilon_k \tilde \rho^{k-k_1} \rho_{decay}^{k_1}\right)
        &< \frac{\mu_1}{\mu_i} \\
        \alpha \mu_1^{-1}\lambda_1 \tau_1 \epsilon_{\delta_0} \epsilon_k + C_2' \mu_1^{-1} \epsilon_k \tilde \rho^{k-k_1} \rho_{decay}^{k_1}
        &< \alpha \frac{\lambda_i - \lambda_1}{\mu_i} \\
        \lambda_1 \tau_1 \epsilon_{\delta_0} \epsilon_k 
        &< \frac{\mu_1}{\mu_i} (\lambda_i - \lambda_1) - C_2' \alpha^{-1} \epsilon_k \tilde \rho^{k-k_1} \rho_{decay}^{k_1} \\
        \tau_1 \epsilon_{\delta_0}
        &< \frac{\mu_1}{\mu_i} \frac{\lambda_i - \lambda_1}{\lambda_1} \epsilon_k^{-1} - C_2' \lambda_1^{-1} \alpha^{-1} \tilde \rho^{k-k_1} \rho_{decay}^{k_1}.
    \end{align*}
    For $k = k_1$, it suffices that 
    \begin{equation}\label{eqn:phase2_tau1}
        \tau_1 \epsilon_{\delta_0} < \frac{\mu_1}{\mu_2} \frac{\lambda_2 - \lambda_1}{\lambda_1} \sqrt{\delta_0} - C_2' \lambda_1^{-1} \alpha^{-1} \rho_{decay}^{k_1},
    \end{equation}
    where the minimum is taken over all $i \geq 2$ and satisfied when $i = 2$. Under this $\tau_1$ condition, the initial inequality at the top is satisfied for all $i \geq 2$. To control $\epsilon_k$, we will consider a slightly different decomposition of $\norm*{g_k}$ as before.
    \begin{equation}
        \norm*{g_k} = \sqrt{\sum_{i=1}^n \abs*{u_i^\top g_k}^2} = \abs*{u_1^\top g_k} \sqrt{1 + \sum_{i \geq 2} \left(t_i^{(k)}\right)^2}.
    \end{equation}
    Consequently, we have
    \begin{equation*}
        \epsilon_k = \sqrt{1 + \sum_{i \geq 2} \left(t_i^{(k)}\right)^2} 
            \implies
        \epsilon_{k+1} \leq \epsilon_k \sqrt{\frac{1 + \sum_{i \geq 2} \left( \frac{\tilde \mu_i^{(k)}}{\mu_1} t^{(k)}_i + \tilde \epsilon^{(k)}_i\right)^2}{1 + \sum_{i \geq 2} \left(t_i^{(k)}\right)^2} }.
    \end{equation*}
    For $\epsilon_k$ to be non-increasing in $k$, the factor has to be $\leq 1$. Equivalently, we require
    \begin{equation*}
        {\sum_{i \geq 2} \left( \frac{\tilde \mu_i^{(k)}}{\mu_1} t^{(k)}_i + \tilde \epsilon^{(k)}_i\right)^2} \leq {\sum_{i \geq 2} \left(t_i^{(k)}\right)^2}, 
    \end{equation*}
    and it suffices that
    \begin{equation}
        {\frac{\tilde \mu_i^{(k)}}{\mu_1} t^{(k)}_i + \tilde \epsilon^{(k)}_i} \leq {t_i^{(k)}} \qquad \forall i \geq 2.
    \end{equation}
    Consider $k = k_1$, we require
    \begin{align*}
        &\left(\alpha \lambda_i \tau_i \epsilon_{k_1} + C_2' \epsilon_{k_1} \rho_{decay}^{k_1}\right) \left[ 1 + (1+\eta)\left(\alpha \mu_1^{-1}\lambda_1 \tau_1 \epsilon_{\delta_0} \epsilon_{k_1} + C_2' \mu_1^{-1} \epsilon_{k_1}  \rho_{decay}^{k_1}\right) \right] \\
        \leq& \left( \mu_1 - \mu_i \left[ 1 + (1+\eta)\left(\alpha \mu_1^{-1}\lambda_1 \tau_1 \epsilon_{\delta_0} \epsilon_{k_1} + C_2' \mu_1^{-1} \epsilon_{k_1} \rho_{decay}^{k_1}\right) \right] \right) t_i^{(k_1)}.
    \end{align*}
    If $t_i^{(k_1)} < \frac{\mu_1}{\mu_1 - \tilde \mu_i^{(k_1)}} \tilde \epsilon_i^{(k_1)}$, we are done as that is indeed the best result we can obtain for the convergence. In the worst case, we have $t_i^{(k_1)} = \sqrt{\frac{1-\delta_0}{\delta_0}}$ and $\epsilon_{k_1} = \frac{1}{\sqrt{\delta_0}}$. For this, we simplify the above inequality by multiplying both sides with $\sqrt{\delta_0}$ to get
    \begin{align*}
        &\left(\alpha \lambda_i \tau_i + C_2'  \rho_{decay}^{k_1}\right) \left[ 1 + (1+\eta)\left(\alpha \mu_1^{-1}\lambda_1 \tau_1 \epsilon_{\delta_0} \epsilon_{k_1} + C_2' \mu_1^{-1} \epsilon_{k_1}  \rho_{decay}^{k_1}\right) \right] \\
        \leq& \left( \mu_1 - \mu_i \left[ 1 + (1+\eta)\left(\alpha \mu_1^{-1}\lambda_1 \tau_1 \epsilon_{\delta_0} \epsilon_{k_1} + C_2' \mu_1^{-1} \epsilon_{k_1} \rho_{decay}^{k_1}\right) \right] \right) \sqrt{1-\delta_0}.
    \end{align*}
    Combining the second term on the RHS with the LHS, we obtain
    \begin{align*}
        \left(\sqrt{1-\delta_0} \mu_i + \alpha \lambda_i \tau_i + C_2'  \rho_{decay}^{k_1}\right) \left[ 1 + (1+\eta)\left(\alpha \mu_1^{-1}\lambda_1 \tau_1 \epsilon_{\delta_0} \epsilon_{k_1} + C_2' \mu_1^{-1} \epsilon_{k_1} \rho_{decay}^{k_1}\right) \right]
        \leq \mu_1 \sqrt{1-\delta_0}.
    \end{align*}
    Considering this as a function of $\tau_i$, we get a linear function
    \begin{equation*}
        a \tau_i + b \leq 0,
    \end{equation*}
    where
    \begin{align*}
        a 
            &= \alpha \lambda_i \left[ 1 + (1+\eta)\left(\alpha \mu_1^{-1}\lambda_1 \tau_1 \epsilon_{\delta_0} \epsilon_{k_1} + C_2' \mu_1^{-1} \epsilon_{k_1} \rho_{decay}^{k_1}\right) \right]\\
        b
            &= \left(\sqrt{1-\delta_0} \mu_i + C_2'  \rho_{decay}^{k_1}\right) \left[ 1 + (1+\eta)\left(\alpha \mu_1^{-1}\lambda_1 \tau_1 \epsilon_{\delta_0} \epsilon_{k_1} + C_2' \mu_1^{-1} \epsilon_{k_1} \rho_{decay}^{k_1}\right) \right] -\mu_1 \sqrt{1-\delta_0}.
    \end{align*}
    Since $a > 0$ this is equivalent to $\tau_i < \frac{b}{a}$.
    \begin{align*}
        \lambda_i\tau_i 
            &< \frac{\mu_1 \sqrt{1-\delta_0} - \left(\sqrt{1-\delta_0} \mu_i + C_2'  \rho_{decay}^{k_1}\right) \left[ 1 + (1+\eta)\left(\alpha \mu_1^{-1}\lambda_1 \tau_1 \epsilon_{\delta_0} \epsilon_{k_1} + C_2' \mu_1^{-1} \epsilon_{k_1} \rho_{decay}^{k_1}\right) \right]}{\alpha \left[ 1 + (1+\eta)\left(\alpha \mu_1^{-1}\lambda_1 \tau_1 \epsilon_{\delta_0} \epsilon_{k_1} + C_2' \mu_1^{-1} \epsilon_{k_1} \rho_{decay}^{k_1}\right) \right]} \\
        \frac{\alpha \lambda_i \tau_i}{\sqrt{1-\delta_0}}
            &< \frac{\mu_1 }{\left[ 1 + (1+\eta)\left(\alpha \mu_1^{-1}\lambda_1 \tau_1 \epsilon_{\delta_0} \epsilon_{k_1} + C_2' \mu_1^{-1} \epsilon_{k_1} \rho_{decay}^{k_1}\right) \right]} - \left(\mu_i + C_2' (1-\delta_0)^{-1/2} \rho_{decay}^{k_1}\right).
    \end{align*}
    Similar to Phase 1, we assume the contribution by $\rho_{decay}^{k_1}$ is small in the local regime and the second term in the denominator on the RHS is controlled small by $\epsilon_{\delta_0}$ in comparison to 1, giving us
    \begin{equation*}
        \tau_i 
           \lesssim \frac{(1-\tilde \eta)\mu_1 - \mu_i}{\alpha \lambda_i} \sqrt{1-\delta_0} \lesssim \frac{\lambda_i - \lambda_1}{\lambda_i} \sqrt{1-\delta_0},
    \end{equation*}
    where $0 < \tilde\eta \ll 1$.
    
    \paragraph{Part 2e: Exponential Decay in the Inductive Step.}
    Under these conditions for $\tau_1, \cdots, \tau_n$, we have that 
    \[
        \epsilon_{k_1 + 1} < \epsilon_{k_1} \implies \inner*{u_1, g_{k_1+1}}^2 \geq \delta_0 \norm*{g_{k_1+1}}^2.
    \]
    By Proposition \eqref{prop:sum_angle}, we have $\inner*{v_{k_1}, g_{k_1+1}}^2 \geq (2\delta_0 - 1)^2$. Then, Lemma \eqref{lem:decay_rescaled} with $\delta = (2\delta_0 - 1)^2$, we obtain
    \begin{align*}
        f(x_{k_1+2}) - f(x_{k_1+1}) 
            &\leq \Bigg\{ \left[ -\alpha \tau + \frac{\alpha^2 L}{2} \left(1 - \frac{n-1}{d}(1+\tau)\right) \right] (2\delta_0 - 1)^2 \\&\quad+ \left[-\alpha (1-\tau) + \frac{\alpha^2 L}{2} \frac{n-1}{d}\right] \Bigg\}  \norm*{g_{k_1+1}}^2.
    \end{align*}
    The minimum of the polynomial $ax^2 - bx$ where $a,b > 0$ is attained at $x^\ast = \frac{b}{2a}$, meaning
    \begin{align}
        \alpha^\ast
        &= \frac{d}{L} \left(\frac{\tau (2 \delta_0- 1)^2 + (1-\tau)}{(n-1) - (2\delta_0-1)^2((n-1)(1+\tau) - d)}\right) \nonumber \\
        &= \frac{d}{L} \left(\frac{1 - 4\tau\delta_0(1-\delta_0)}{(n-1)(1 -(2\delta_0-1)^2(1+\tau)) + d(2\delta_0-1)^2)}\right) \nonumber \\
        &= \frac{d}{L} \left(\frac{1 - 4\tau\delta_0(1-\delta_0)}{(n-1)(-\tau + 4\delta_0(1-\delta_0)(1+\tau)) + d(2\delta_0-1)^2)}\right).
    \end{align}
    Substituting this back, and noting that $a{x^\ast}^2 - bx^\ast = -\frac{b^2}{4a} = -\frac{b}{2} x^\ast$, we find that
    \begin{align*}
        f(x_{k_1+2}) - f(x_{k_1+1}) 
            &\leq -\frac{1}{2} (1-4\tau \delta_0 (1-\delta_0)^2)\alpha^\ast \norm*{g_{k_1+1}}^2 \\
            &= -\frac{d}{L} \left(\frac{[1 - 4\tau\delta_0(1-\delta_0)]^2}{(n-1)(-\tau + 4\delta_0(1-\delta_0)(1+\tau)) + d(2\delta_0-1)^2)}\right) \norm*{g_{k_1+1}}^2.
    \end{align*}
    Consequently, we have
    \begin{align*}
        f(x_{k_1+2}) - f(x^\ast)
            &\leq f(x_{k_1+1}) - f(x\ast) - \hat \rho \norm*{g_{k_1+1}}^2 \\
            &\leq (1-2\mu\hat\rho) (f(x_{k_1+1}) - f(x\ast)) \\
            &\leq \tilde\rho^4 \rho_{decay}^{2k_1},
    \end{align*}
    with probability at least $1 - e^{-cd\tau^2}$, conditioned on the iterate $k_1 + 1$. We indicate the decay in 1 step with $\tilde \rho^2$, since the gradient term is also squared. Then, $\tilde \rho$ satisfies
    \begin{equation*}
        \tilde \rho^2 = (1 - 2\mu \hat \rho) = 1 - 2 \mu \frac{d}{L} \left(\frac{[1 - 4\tau\delta_0(1-\delta_0)]^2}{(n-1)(-\tau + 4\delta_0(1-\delta_0)(1+\tau)) + d(2\delta_0-1)^2)}\right) < 1.
    \end{equation*}
    Note that for $f(x_{k_1+1}) - f(x_{k_1})$, since $g_{k_1} = v_{k_1}$, we have that the polynomial has a larger root, which means we can use the same step size as for $k_1+1$ step. Therefore, the decay factor will be the same (we use $\tilde \rho$ to denote $\tilde \rho_{k_2}$ above since we find that this factor is independent of the index). Now, we have shown that given the past, the next step still enjoys a decay with a slightly different factor as compared to Phase 1. Nonetheless, the inductive step will still hold and we have $t_i^{(k)} \leq O(\tilde \epsilon_i^{(k)})$. 

    \paragraph{Part 3: Concluding the Proof.} 
    Consider the events as follow.
    \begin{align*}
        \mathcal{A}(\tau) 
            &:= \left\{ \norm*{g_k} \leq C \rho_{decay}^k \quad \forall k \in [k_0, k_1] \right\}\\
        \mathcal{B}_i^{(k)}(\tau_i) 
            &:= \Bigg\{\abs*{\inner*{\hat P_k^\top u_i, \hat P_k^\top g_k} - \inner*{\Proj_{{v_{k_1}^\perp}}u_1, \Proj_{{v_{k_1}^\perp}} g_k}} \\
            &\qquad \leq \tau_i \norm*{\Proj_{{v_{k_1}^\perp}}u_1}\norm*{\Proj_{{v_{k_1}^\perp}}g_k}\Bigg\}, \quad k \in (k_1, K_\epsilon), i \in [n] \\
        \mathcal{C}_k(\tau) 
            &:= \left\{ \norm*{\hat P_k^\top g_k}^2 \approx (1 \pm \tau) \norm*{\Proj_{\hat v_{k_1}}g_k}^2 \right\}, \quad k \in (k_1, K_\epsilon).
    \end{align*}
    where
    \begin{align*}
        \mathbb{P} [\mathcal{A}_1(\tau)] &\geq 1 - e^{-\frac{1}{8} k_1 (1 - p_{decay}(\tau, \delta, n, d))} \\
        \mathbb{P} \left[ \mathcal{B}_i^{(k)}(\tau_i) \right] &\geq 1 - 2e^{-cd\tau_i^2} \\
        \mathbb{P} \left[ \mathcal{C}_k(\tau) \right] &\geq 1 - 2e^{-cd\tau^2}.
    \end{align*}
    Then, with $\mathcal{F}_k = \sigma(\tilde P_0, \cdots, \tilde P_{k-1})$, we know that $\mathcal{A}(\tau) \in \mathcal{F}_{k_1}$, similarly for $\mathcal{B}_i^{(k)}(\tau_i), \mathcal{C}_k(\tau) \in \mathcal{F}_k$. Let $\mathcal{B}_k = \mathcal{C}_k\cap \left(\bigcap_{i \in [n]}\mathcal{B}_i^{(k)}\right)$, we have $\mathcal{B}_k$ holds with probability at least $1 - 2\sum_{i=1}^n e^{-cd\tau_i^2} - 2e^{-cd\tau^2}$. Using Lemma \ref{lem:prod_prob}, we have with probability at least
    \begin{equation}\label{eqn:local_regime_prob}
        \left(1 - 2\sum_{i \in [n]} e^{-cd\tau_i^2} - 2e^{-cd\tau^2}\right)^{K_\epsilon - k_1} \times \left(1 - e^{-\frac{1}{8} k_1 (1 - p_{decay}(\tau, \delta, n, d))}\right)
    \end{equation}
    that the projections onto every other directions $\abs*{u_i^\top g_k}/\abs*{u_1^\top g_k} = O((\lambda_i - \lambda_1)\sqrt{1-\delta_0})$ to be much smaller than the projection onto the dominant direction $u_1$.
\end{proof}

\end{document}